\documentclass{article}

\usepackage{amsmath}
\usepackage{amssymb}
\usepackage{bbm}
\usepackage[margin=1in]{geometry}
\usepackage{graphicx}
\usepackage[amsmath,amsthm,thmmarks]{ntheorem}
\usepackage{pgfplots}
\usepackage{subfig}
\usepackage{xcolor}

\usetikzlibrary{positioning}
\pgfplotsset{compat=1.18}

\theoremstyle{plain}
\newtheorem{theorem}{Theorem}[section]
\newtheorem{corollary}[theorem]{Corollary}
\newtheorem{lemma}[theorem]{Lemma}
\newtheorem{proposition}[theorem]{Proposition}
\theoremstyle{definition}
\newtheorem{definition}[theorem]{Definition}
\newtheorem{exmp}[theorem]{Example}

\newcommand\numberthis{
	\addtocounter{equation}{1}\tag{\theequation}
}

\DeclareMathOperator*{\argmax}{arg\,max}

\def\P{\mathbbm{P}}
\def\E{\mathbbm{E}}
\def\V{\text{Var}}
\def\F{\mathcal{F}}

\def\K{\boldsymbol{\mathfrak{K}}}
\def\L{\mathcal{L}}
\def\N{\mathbbm{N}}
\def\R{\mathbbm{R}}
\def\1{\mathbbm{1}}
\def\X{\{X_n, z_0\}}
\def\Xfull{\{X_n, z_0\}_{n\in\N_1}}

\def\pt{\tilde{\phi}}

\def\Z{\{Z_n, z_0\}}
\def\Zfull{\{Z_n, z_0\}_{n\in\N_1}}
\def\btheta{\boldsymbol{\theta}}

\def\bin{\text{Bin}}
\def\dn{\text{DN}}
\def\geo{\text{Geo}}
\def\poi{\text{Poi}}

\begin{document}

\title{Existence and non-existence of consistent estimators in supercritical controlled branching processes}

\author{
  Peter Braunsteins\footnote{School of Mathematics and Statistics, University of New South Wales, email: \texttt{p.braunsteins@unsw.edu.au}, ORCID: 0000-0003-1864-0703}, Sophie Hautphenne\footnote{School of Mathematics and Statistics, University of Melbourne, email: \texttt{sophiemh@unimelb.edu.au}, ORCID: 0000-0002-8361-1901}, and James Kerlidis\footnote{School of Mathematics and Statistics, University of Melbourne, email: \texttt{jkerlidis@student.unimelb.edu.au}, ORCID: 0009-0004-4904-0262}
}
\date{}

\maketitle

\begin{abstract}
    We consider the problem of estimating the parameters of a supercritical controlled branching process consistently from a single observed trajectory of population size counts. Our goal is to establish which parameters can and cannot be consistently estimated. When a parameter can be consistently estimated, we derive an explicit expression for the estimator. We address these questions in three scenarios: when the distribution of the control function distribution is known, when it is unknown, and when progenitor numbers are observed alongside population size counts. Our results offer a theoretical justification for the common practice in population ecology of estimating demographic and environmental stochasticity using separate observation schemes.
\end{abstract}

\textit{Keywords:} branching process, consistent estimation, control function.

\smallskip

2000 MSC: 60J80, 60J10.

\section{Introduction}\label{sec:Introduction}

Branching processes are stochastic models in which individuals reproduce and die according to probabilistic laws.
They have been used in various applications, particularly in population biology \cite{jagers75, haccou05, kimmel15}.
The simplest branching process is the discrete-time \textit{Bienaym{\'e}-Galton-Watson process} (BGWP) whose population size at each generation $n$ is recursively defined as 
\begin{equation}\label{eqn:BGWP}
    X_0 = z_0, \quad X_n = \sum_{i=1}^{X_{n-1}} \xi_{n,i} \;\text{ for } n \geq 1,
\end{equation}
for some initial value $z_0 > 0$, where $\{\xi_{n,i}\}_{n,i \geq 1}$ are independent random variables with common distribution $\xi$, known as the \textit{offspring distribution}.
These processes exhibit exponential growth, in that $\E(X_n\,|\,Z_0=z_0) = z_0 \, m^n$, where $m := \E(\xi)$ is the offspring mean.

BGWPs are often not suitable models for biological populations. Indeed, many biological populations do not grow exponentially; for example, due to competition for limited resources, they may exhibit logistic growth. In addition, individuals within the same generation may not give birth independently;  for example, this could be due to random population-wide factors, such as weather conditions, that are often referred to as \textit{environmental stochasticity} \cite[Chapter~1.2]{lande03}.
A common extension to the BGWP that overcomes these limitations is the \textit{controlled branching process} (CBP). Using the definition introduced by Yanev \cite{yanev76} (see also \cite{sebastyanov74}), a CBP $\Zfull$ is defined  recursively as
\begin{equation}\label{eqn:CBP}
    Z_0 = z_0, \quad
    Z_n = \sum_{i=1}^{\phi_n(Z_{n-1})} \xi_{n,i} \;\text{ for } n \geq 1,
\end{equation}
where the family of random variables $\{ \phi_n(z) \}_{n \geq1, z \geq0}$ defines the process' \textit{control function}. 
We assume that the $\phi_n(z)$'s are mutually independent, are independent of the $\xi_{n,i}$'s, and that their distribution only depends on $z$ (and not on $n$).

When using a CBP $\Zfull$ to model a population, we often consider a class of CBPs parameterised by some parameter $\btheta\in \Theta$, and use observed population size counts $Z_0,\, Z_1,\, \ldots,\, Z_n$ to estimate $\btheta$ via an estimator $\hat{\btheta}_n := \hat{\btheta}_n(Z_0,\, Z_1,\, \ldots,\, Z_n)$. Here we are interested in the \textit{supercritical} case with $\P(Z_n \to \infty) > 0$. A desirable property of the estimator $\hat{\btheta}_n$ is consistency on $\{Z_n\to\infty\}$, the \textit{set of unbounded growth}: the sequence of estimators $\{ \hat{\btheta}_n \}$ is said to be weakly (resp. strongly) consistent on the set of unbounded growth if, on $\{Z_n\to\infty\}$ and for every initial population size $z_0\in\N_1$,
\begin{align}
    \forall\, \varepsilon > 0, \quad
    \lim_{n \to \infty} \P( | \hat{\btheta}_n - \btheta | > \varepsilon \,|\, Z_0 = z_0) = 0
    \quad & \text{(\textit{weak consistency})} \label{eqn:WeakConsistency} \\
    \P\Big( \lim_{n \to \infty}\hat{\btheta}_n = \btheta \,\Big|\, Z_0 = z_0 \Big) = 1
    \qquad\qquad & \text{(\textit{strong consistency}).} \label{eqn:StrongConsistency}
\end{align}
If $\{ \hat{\btheta}_n \}$ is consistent, then, as more data become available, the sequence converges to the true parameter value $\btheta$. On the other hand, if $\{ \hat{\btheta}_n \}$ is not consistent, then we may question whether a consistent estimator for $\btheta$ actually exists. If not, this may be an indication that the model is over-parametrised.

The goal of this paper is to help determine which parameters of a supercritical CBP can be estimated consistently. We aim to give a complete picture by addressing the following two questions:
\begin{itemize}
    \item \textbf{Q1}: Which parameters of a supercritical CBP cannot be consistently estimated?
    \item \textbf{Q2}: What is an explicit expression for a consistent estimator when a parameter is consistently estimable?
\end{itemize}
Under certain regularity conditions, we answer these questions in three different scenarios:
\begin{itemize}
    \item \textbf{S1} (Section \ref{sec:CBPEstimationKnownPhi}): When the distribution of the random control function $\{ \phi_n(z) \}$ is \textit{known}, and we aim to estimate the parameters of the offspring distribution $\xi$ from observations of the population size counts $Z_0,\, Z_1,\, \ldots,\, Z_n$.
    \item \textbf{S2} (Section \ref{sec:CBPEstimationUnknownControl}): When the distribution of the random control function $\{ \phi_n(z) \}$ is \textit{unknown}, and we aim to estimate the parameters of \textit{both} the random control function $\{ \phi_n(z) \}$ and the offspring distribution $\xi$ from $Z_0,\, Z_1,\, \ldots,\, Z_n$.
    \item \textbf{S3} (Section \ref{sec:CBPEstimationObservedProgenitors}): In the same setting as \textbf{S2}, but where both the population size counts $Z_0,\, Z_1,\, \ldots,$ $Z_n$ and progenitor numbers $\phi_0(Z_0),\, \phi_1(Z_1),\, \ldots,\, \phi_n(Z_n)$ are observed.
\end{itemize}

\textbf{Q1} and \textbf{Q2} have both been resolved for supercritical BGWPs with and without immigration.
Indeed, for a supercritical BGWP without immigration, Lockhart showed (under mild assumptions) that \textit{no parameter other than the offspring mean and variance can be estimated consistently} \cite{lockhart82} (adapted in Theorem \ref{thm:GWPNoConsistentEstimates}; see also \cite[Theorem 1.2]{guttorp91}). In addition, 
Harris \cite[Theorem~7.2]{harris48} and Heyde \cite[Theorem~4]{heyde70} showed that the estimator $\hat{m}_n := \sum_{i=1}^n Z_i / \sum_{i=0}^{n-1} Z_i$ is (weakly, resp.\ strongly) consistent for the offspring mean. 
A strongly consistent estimator for the offspring variance was established in \cite{heyde74}.
For supercritical BGWPs with immigration, Wei and Winnicki \cite[Proposition~3.3]{wei90} showed that no parameter other than the offspring mean and variance can be estimated consistently, thereby extending Lockhart's result (see also \cite[Theorem 4.5]{winnicki91} which considers the critical case). In addition, consistent estimators for the offspring mean and variance of these processes were provided in \cite[Theorem~2.2]{wei90} and \cite[Section~3]{heyde74}, respectively.
Consequently, questions about the existence of consistent estimators for supercritical BGWPs with and without immigration have been largely resolved.

In contrast, far less is known about the existence of consistent estimators for CBPs. In particular, there has been no attempt to address \textbf{Q1} in any of the three scenarios \textbf{S1--S3}.
In this paper, for each of the scenarios, we extend Lockhart's result for supercritical BGWPs (Theorem~\ref{thm:GWPNoConsistentEstimates}) to CBPs in Theorem~\ref{thm:KnownPhiNoConsistentEstimation} (\textbf{S1}), Theorem~\ref{thm:LinearlyDivisibleNoConsistentEstimation} (\textbf{S2}), and Theorem \ref{thm:KnownProgenitorsNoConsistentEstimation} (\textbf{S3}). 
In each scenario, we follow a common framework for proving the non-existence of consistent estimators, which is outlined in Section \ref{sec:Framework}.
In \textbf{S1}, our result directly extends the BGWP case: when the distribution of the control function is known, only the first two moments of the offspring distribution can be estimated. In \textbf{S2} and \textbf{S3}, however, the extension is no longer direct; indeed, the parameters of the control function must now be estimated, and since the control function is a family of distributions indexed by the population size $z$ (which allows for much richer behaviour), this leads to new challenges.
To help with these challenges, we establish our results under the assumption that the control function is \textit{linearly divisible} (see Definition \ref{def:LinearDivisibility}). Roughly speaking, our results provide conditions for non-existence of consistent estimators which are expressed in terms of the difference in the mean and variance of the next step size for a process with parameter $\btheta$ and a `perturbed' process with parameter $\btheta'$.
The key idea of the proof is showing that the difference in the one-step distributions of the original and perturbed processes is hidden by the randomness implied by the CLT as the population grows, in which case the parameter cannot be estimated consistently.

Our answers to \textbf{Q1} help to clarify which parameters might be possible to consistently estimate in each scenario \textbf{S1}--\textbf{S3}. In our answers to \textbf{Q2} we provide explicit expressions for consistent estimators under some additional regularity conditions.
In \textbf{S1} (Theorem \ref{thm:KnownPhiConsistentEstimators}), under minor regularity assumptions, we establish consistent estimators for the mean and variance of the offspring distribution (Theorem \ref{thm:KnownPhiConsistentEstimators}).
Consistent estimators for the offspring mean have been derived in \cite[Theorem 4.2]{gonzalez04} and \cite[Section~6]{sriram06}; however, both assume that the limit $\tau := \lim_{z\to\infty} \E(Z_1 | Z_0 = z) / z$ exists, and that $\lim_{n\to\infty} (\tau m)^{-n} Z_n$ converges to a non-degenerate random variable.
Our consistent estimator for the mean holds without these restrictive assumptions.
For the offspring variance, consistency had only been demonstrated in the special case of a deterministic control function \cite{gonzalez05}.
Our consistent estimator for the variance holds for a random control function.
Our results require a different proof approach than what has previously been used in the literature. In \textbf{S2} (Theorem \ref{thm:LinearMeanVarPhiConsistentEstimators}), under the assumption that $\E(\phi(z)) = \alpha z$ and $\V(\phi(z)) = \beta z$, we establish consistent estimators for the normalised conditional mean and variance of the next step, i.e.\ for $\mathbb{E}(Z_1 | Z_0=z)/z = m \alpha$ and $\text{Var}(Z_1 | Z_0=z)/z=\sigma^2 \alpha + m^2 \beta$, which are the only quantities that can be estimated consistently under the assumptions of Theorem \ref{thm:LinearlyDivisibleNoConsistentEstimation}.
In \textbf{S3} (Theorem \ref{thm:ObservedProgenitorsConsistentEstimators}), under similar assumptions as Theorem \ref{thm:LinearMeanVarPhiConsistentEstimators}, we construct consistent estimators for $m$, $\alpha$, $\sigma^2$, and $\beta$: the only quantities that can be estimated consistently under the assumptions of Theorem \ref{thm:KnownProgenitorsNoConsistentEstimation}.
These are the first estimators proven to be consistent under this observation scheme.

Our results have implications in population ecology. In this field, a common rule of thumb is that \textit{demographic stochasticity}---the randomness inherent in the independent reproduction and lifetime of individuals within a population---and \textit{environmental stochasticity}---the random changes in environmental conditions that impact a population as a whole---should not be estimated simultaneously from a single trajectory of population size counts \cite[Chapter~1.7.1]{lande03}. To the best of the authors' knowledge, the justification for this rule has only been empirical. In practice, ecologists use different observation schemes when estimating both types of stochasticity. For example, in studying a bird population, they might estimate demographic stochasticity by counting the clutch size and then treat these demographic parameters as known when estimating environmental stochasticity using population size counts.
In the context of CBPs, this principle translates to the idea that both the parameters of 
$\xi$ (demographic stochasticity) and those of $\phi(\cdot)$ (environmental stochasticity)
should not be estimated simultaneously from a single trajectory of population sizes. 
Our results provide theoretical support for this principle for supercritical CBPs (\textbf{Q1} for \textbf{S2}), and show that the parameters of $\xi$  and  $\phi(\cdot)$  can only be consistently estimated together under a more detailed observation scheme (\textbf{Q2} for \textbf{S3}), similar to some observation schemes used by ecologists. We believe our arguments can be extended to other stochastic population models such as diffusion models  \cite{lande03} and supercritical branching processes in a random environment \cite{kersting17}.

The paper is organised as follows. In Section 2, we outline the fundamental consistency results for supercritical BGWPs and illustrate them with an example. In Section 3.1 we present a general framework for establishing the non-existence of consistent estimators. In Sections 3.2--3.4 we present our answers to Questions \textbf{Q1} and \textbf{Q2} for scenarios \textbf{S1}--\textbf{S3}. In Section 4, we discuss future work and open questions. In Section 5 we gather the proofs for each of our non-existence results, while proofs for the consistency of estimators can be found in Section 6.

\section{Motivation}

\subsection{Modelling supercritical populations with BGWPs: Whooping cranes}\label{sec:WhoopingCraneExample}

\begin{figure}[t]
    \centering
    \includegraphics[width=8cm]{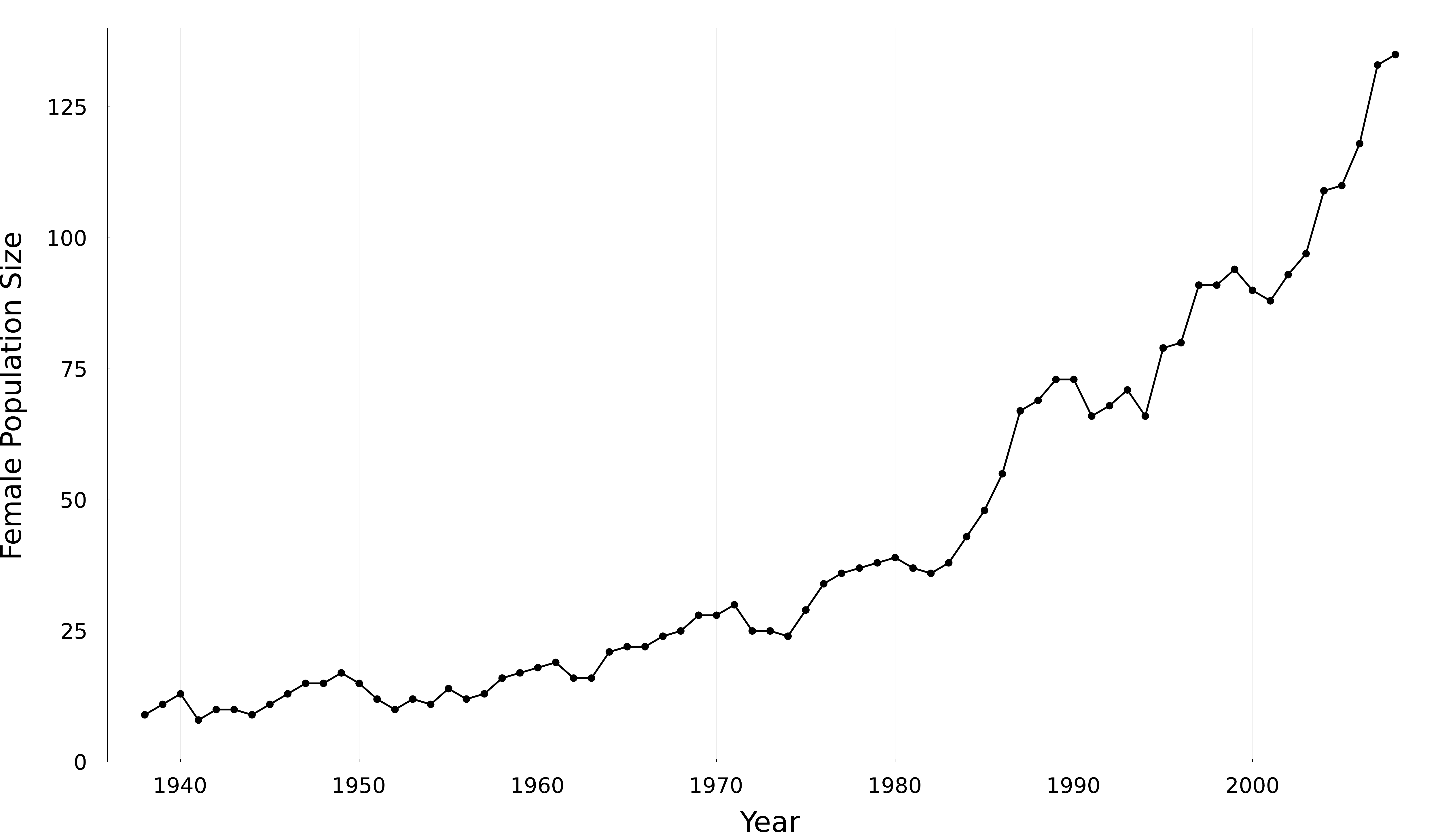}
   \caption{Female population sizes of the whooping crane Aransas-Wood Buffalo flock, 1938--2008 (see~\cite{butler13}). \label{fig:whooping_crane_totals}}
\end{figure}

Consider the annual population-size counts $(z_0, z_1, \dots, z_{70})$ of the females in the Aransas-Wood Buffalo whooping crane flock, displayed in Figure \ref{fig:whooping_crane_totals}.
Because the population growth appears approximately exponential, it is natural to model
this population with a Bienaymé-Galton-Watson branching process (BGWP) $\Xfull$, which is characterised by Equation \eqref{eqn:BGWP}. Recall that $\xi$ is known as the \textit{offspring distribution}, and let $p_k := \P(\xi = k)$ for $k \in \N_0$. 
We fit the data to two parametric BGWP models:
\begin{enumerate}
    \item[(i)] A model where only $p_0$, $p_1$, and $p_2$ can be non-zero, 
    \item[(ii)] A model where only $p_0$, $p_1$, $p_2$, and $p_3$ can be non-zero.
\end{enumerate}
Observe that Model (ii) is more general than Model (i).
Using the Markov property, the likelihood of $(z_0,\, z_1,\, \dots,\, z_{70})$ can be decomposed into a product of factors of the form
\[
    \P(X_n = z_n | X_{n-1} = z_{n-1}) = \P\Bigg( \sum_{i=1}^{z_{n-1}} \xi_{n, i} = z_n \Bigg),
    \quad n \in \{1, \dots, 70\}.
\]
These next step sizes are convolutions of independent random variables, whose generating functions can therefore be computed easily, and then inverted for example using the numerical technique of Abate and Whitt~\cite{abate92}. By maximising the resulting (approximate) likelihoods,
we obtain maximum likelihood estimates (MLEs) for each model:
\begin{enumerate}
    \item[(i)] $\hat{p}_0 = 0.1538$, $\hat{p}_1 = 0.6491$, and $\hat{p}_2 = 0.1971$,
    \item[(ii)] $\hat{p}_0 = 0.1538$, $\hat{p}_1 = 0.6491$, $\hat{p}_2 = 0.1971$, and $\hat{p}_3 = 0.0000$.
\end{enumerate}
We use parametric bootstrap \cite[Section 13.3]{tibshirani93} to obtain 95\% confidence intervals:
\begin{enumerate}
    \item[(i)] $p_0: (0.1006, 0.2150)$, $p_1: (0.5302, 0.7605)$, and $p_2: (0.1340, 0.2566)$, 
    \item[(ii)] $p_0: (0.0730, 0.2012)$, $p_1: (0.5605, 0.8618)$, $p_2: (0.0000, 0.2404)$, and $p_3: (0.0000, 0.0694)$.
\end{enumerate}
Observe that, while the parameter estimates are identical for both models, the confidence intervals for Model (ii) are wider than those for Model (i).
A key question in this paper is whether the width of these confidence intervals will shrink to zero as more data become available.
To investigate this question, in Figure \ref{fig:p_mse} we display the mean squared error (MSE) of the MLEs---again computed using parametric bootstrap---for different trajectory lengths. We observe that, for Model (i), the MSE for each estimate appears to be converging steadily to zero, whereas for Model (ii) this does not seem to be the case.

\begin{figure}[t]
    \centering
    \subfloat[
        Mean squared error for the MLEs of $p_0$, $p_1$, and $p_2$ in Model (i)
    ]{{ \includegraphics[width=10.5cm]{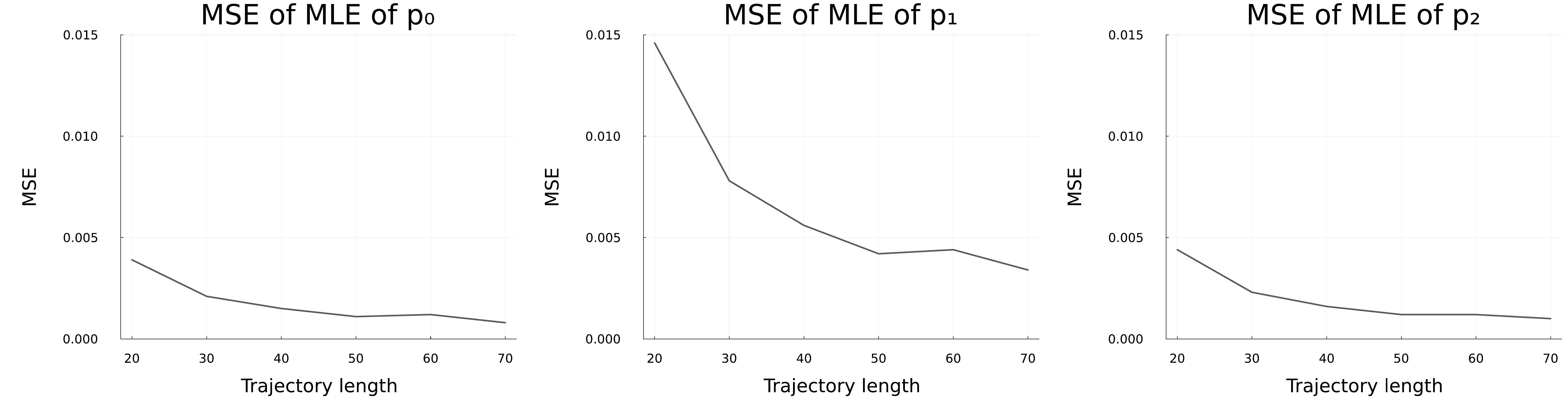} }}
    \\
    \subfloat[
        Mean squared error for the MLEs of $p_0$, $p_1$, $p_2$, and $p_3$ in Model (ii)
    ]{{ \includegraphics[width=14cm]{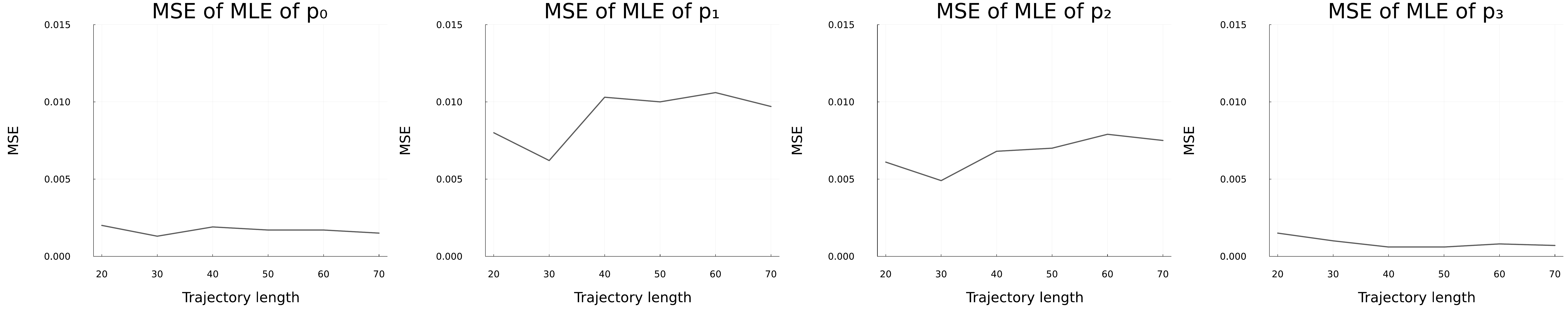} }}
    \caption{
        Mean squared error of the maximum likelihood estimates generated using parametric bootstrap for Models (i) and (ii) with 1000 simulations of trajectories of length 20, 30, \dots, 70.
    }
    \label{fig:p_mse}
\end{figure}

For a supercritical BGWP, it has been shown in \cite{heyde70} and \cite{heyde74}, respectively, that consistent estimators for the offspring mean $m$ and offspring variance $\sigma^2$ exist.
Theorem \ref{thm:GWPNoConsistentEstimates} below, adapted from \cite[Theorem 2]{lockhart82}, demonstrates that $m$ and $\sigma^2$ are the \textit{only} parameters of a supercritical BGWP that can be consistently estimated.
For Model (i), we can formulate consistent estimators for $p_0$, $p_1$, and $p_2$ in terms of those for $m$ and $\sigma^2$, by solving a system of three equations with three unknowns ($\hat{p}_1 + 2\hat{p}_2 = \hat{m}$,\; $\hat{p}_1 + 4\hat{p}_2 = \hat{\sigma}^2 + \hat{m}^2$,\; and $\hat{p}_0 + \hat{p}_1 + \hat{p}_2 = 1$).
This provides theoretical justification for why the MSE of the estimates in Model (i) appears to converge to zero in Figure \ref{fig:p_mse}.
For Model (ii), we aim to estimate $p_0$, $p_1$, $p_2$, and $p_3$ consistently using $\hat m$ and $\hat \sigma^2$,  however, we now have four unknowns but we still have only three equations in our system. Theorem \ref{thm:GWPNoConsistentEstimates} will then demonstrate that $p_0$, $p_1$, $p_2$, and $p_3$ \textit{cannot} all be consistently estimated simultaneously.

We now lay out the setting of Theorem \ref{thm:GWPNoConsistentEstimates}. Let $\Pi^{GW}$ be a set of supercritical ($m>1$) BGWPs in a given parametric family.
For example, in Model (ii), $\Pi$ is set of BGWPs  where only $p_0$, $p_1$, $p_2$, and $p_3$ can be non-zero and $m>1$. For ease of exposition, we assume that all offspring distributions $\xi$ of processes in $\Pi^{GW}$ are of lattice size one. 
We also let $\boldsymbol{\theta}$ be a function from $\Pi^{GW}$ to $\mathbb{R}^d$, representing the quantities of the model which we would like estimate. For example, in Model (ii), if we would like to estimate the full distribution then $\boldsymbol{\theta}=(p_0,p_1,p_2,p_3)$, whereas if we would like to estimate only the third moment then $\theta=p_1+8p_2+27p_3$.  With a slight abuse of language, we refer to $\boldsymbol{\theta}$ as the `parameters' of a process $\Z \in \Pi^{GW}$.
We say that $\hat \btheta_n$ is a weakly consistent estimator for $\btheta$ on the set of unbounded growth if \eqref{eqn:WeakConsistency} holds for all BGWPs $\Z \in \Pi^{GW}$.

\begin{theorem}\label{thm:GWPNoConsistentEstimates}
      If there exist two BGWPs $\Z, \X \in \Pi^{GW}$ with the same offspring mean and variance but with different parameters $\btheta_Z \neq \btheta_X$, then no weakly consistent estimator for $\btheta$ exists on the set of unbounded growth.
\end{theorem}

Let us return to Model (ii) in the whooping crane example, with $\btheta := (p_0,\, p_1,\, p_2,\, p_3)$.
Note that if $\X$ is the BGWP with $p_{0,X} = 0.1538$, $p_{1,X} = 0.6491$, $p_{2,X} = 0.1971$ and $p_{3,X} = 0$ (matching the MLEs found above), and $\Z$ is a BGWP with $p_{0,Z} = 0.0891$, $p_{1,Z} = 0.8432$, $p_{2,Z} = 0.003$ and $p_{3,Z} = 0.0647$,
then both processes have the same mean and variance for their offspring distributions. Thus, by Theorem~\ref{thm:GWPNoConsistentEstimates}, it is not possible to consistently estimate $\btheta$.
This provides a theoretical justification for why the MSE of the estimates in Model (ii) appears \emph{not} to converge to zero in Figure~\ref{fig:p_mse}.

To understand the intuition behind Theorem 2.1, we note that the observations $(z_0,\, z_1,\, \dots,\, z_{70})$ are not taken from the distribution of $\xi$ itself. Instead, they are taken from the distribution of the next-step size $(Z_n | Z_{n-1} = z_{n-1})$, which corresponds to the convolution $\sum_{i=1}^{z_{n-1}} \xi_{n, i}$.
Given that the population size is growing on $\{Z_n\to\infty\}$, and the next-step size distribution is the sum of independent copies of $\xi$, the central limit theorem applies, and thus all information but the mean and variance of $\xi$ is eventually hidden.

\begin{figure}[t]
    \centering
    \subfloat[
        MSE for the MLEs of $m$ and $\sigma^2$ in Model (i)
    ]{{ \includegraphics[width=7.8cm]{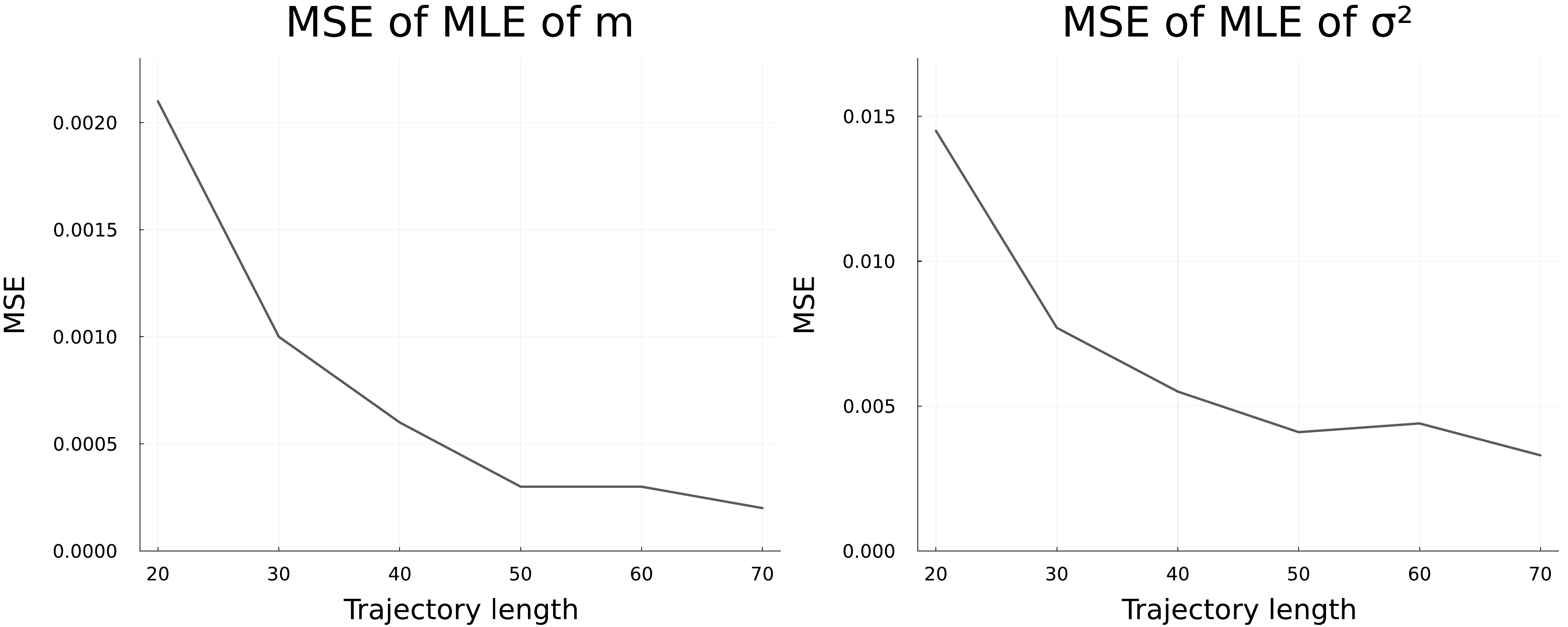} }}
    \hspace{0.3cm}
    \subfloat[
        MSE for the MLEs of $m$ and $\sigma^2$ in Model (ii)
    ]{{ \includegraphics[width=7.8cm]{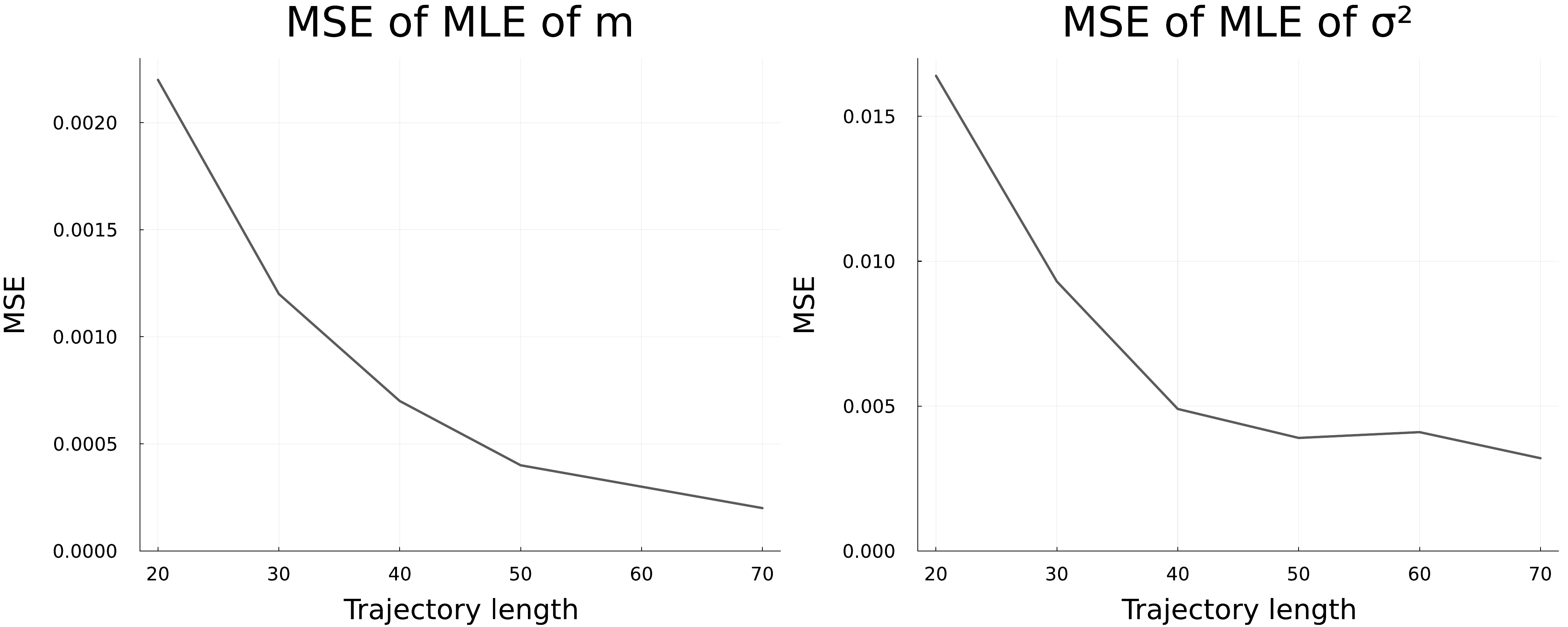} }}
    \caption{
        Mean squared error of the maximum likelihood estimates of $m$ and $\sigma^2$ generated using parametric bootstrapping for Models (i) and (ii); 1000 simulations at trajectory length 20, 30, \dots, 70.
    }
    \label{fig:m_sigma2_mse}
\end{figure}

Despite the fact that the MLEs for $p_0$, $p_1$, $p_2$, and $p_3$ in Model (ii) are not consistent, we can use these estimators to construct a consistent estimator for $m$ and $\sigma^2$ (i.e. $\hat m=\hat p_1+2\hat p_2+3 \hat p_3$ and $\hat{\sigma}^2=\hat p_1+4\hat p_2+9 \hat p_3-\hat m^2$). Figure \ref{fig:m_sigma2_mse} depicts the MSE of the resulting MLEs for $m$ and $\sigma^2$ in Models (i) and (ii), which converges to zero for both models.
Note that for $\btheta = ({m}, {\sigma}^2)$, the conditions of Theorem 2.1 are not satisfied, therefore a consistent estimator for the mean and variance \textit{could potentially} exist (and it does indeed, see \cite{heyde70} and \cite{heyde74}, respectively).

\subsection{Modelling with CBPs}

Controlled branching processes (CBPs), defined by the recursive equation \eqref{eqn:CBP}, are an extension of BGWPs that can capture complex characteristics of biological populations, such as non-exponential growth and dependencies between individuals. 
Recent advances for CBPs propose new methods for estimating many parameters simultaneously, and possibly even the entire distribution of the process \cite{gonzalez13, gonzalez16, gonzalez21}. The focus in these cases is on using Bayesian and algorithmic approaches to obtain parameter estimates, rather than on analysing their asymptotic properties. 
While Theorem \ref{thm:GWPNoConsistentEstimates} establishes a theoretical foundation for understanding the limits of consistent estimation in BGWPs, no analogous framework has been developed for CBPs to assess whether consistent estimators exist in these models.
In the next section, we establish this framework.

\section{Estimation for supercritical CBPs}\label{sec:CBPEstimation}
\subsection{A framework for proving the non-existence of consistent estimators}\label{sec:Framework}

To extend the results from BGWPs to CBPs, we start by defining a class of CBPs with positive probability of unbounded growth. For a given CBP $\Z$ with $Z_0 = z_0 \in \N_1$, such that
for all $n \geq 1$,
\[
    Z_n = \sum_{i=1}^{\phi(Z_{n-1})} \xi_{n,i},
\]
we denote the offspring mean and variance by $m := \E\xi$ and $\sigma^2:=\V(\xi)$, assuming throughout that $\sigma^2 > 0$, and we denote the mean and variance of the control function by $\varepsilon(z) := \E\phi(z)$ and  $\nu^2(z) := \V(\phi(z))$.
Following  \cite[p.76]{gonzalez18}, we define the \textit{mean growth rate} of the process $\Z$ at population size $z$ as
\[
    \tau(z) := \frac{1}{z} \E(Z_1 | Z_0 = z) = \frac{\varepsilon(z)}{z} \cdot m,
\]
and call $\Z$ \textit{supercritical} if
\begin{equation}\label{eqn:supercriticalCBP}
    \liminf_{z\to\infty} \tau(z) > 1.
\end{equation}
Recall that $\Z$ is said to grow unboundedly if $Z_n \to \infty$ as $n \to \infty$.
Unlike BGWPs, supercritical CBPs do not necessarily have a positive probability of unbounded growth (see \cite[Example 3.1]{gonzalez18}).
Theorem 3.2 of \cite{gonzalez18} provides a sufficient condition for $\P(Z_n\to\infty) > 0$, namely that there exist $a, b > 0$ such that
\begin{equation}\label{eqn:BoundedControlFunctionMoments}
    \sup_{z \geq 1} \left\{ \frac{\varepsilon(z)}{z} \right\} \leq a \quad \text{and} \quad
    \sup_{z \geq 1} \left\{ \frac{\nu^2(z)}{z} \right\} \leq b.
\end{equation}
In fact, \cite[Theorem 3.2]{gonzalez18} shows that under \eqref{eqn:BoundedControlFunctionMoments}, $\P(Z_n \to \infty) \to 1$ as the initial population size $z_0$ approaches infinity.

Similar to Section \ref{sec:WhoopingCraneExample}, we let $\Pi$ be a set of supercritical CBPs that satisfy \eqref{eqn:BoundedControlFunctionMoments} 
in a given parametric family, and we let $\boldsymbol{\theta}$ be a function from $\Pi$ to $\mathbb{R}^d$, representing the quantities we would like estimate (referred to as `the parameters').
We say that $\hat \btheta_n:=\hat \btheta_n(Z_0,\, Z_1,\, \ldots,\, Z_n)$ is a weakly consistent estimator for $\btheta$ on the set of unbounded growth if \eqref{eqn:WeakConsistency} holds for all CBPs $\Z \in \Pi$, over every initial population size $z\in\N_1$.

We now relate the total variation distance (TVD) between the distributions of two processes $\Z$, $\X\in\Pi$,
 \[
    || \L_{\Z} - \L_{\X} ||_{TV} := \sup_{C \subseteq \N_0^{\infty}} | \P(\Z \in C) - \P(\X \in C) |,
\]
to the non-existence of consistent estimators for $\btheta$.
If a consistent estimator $\hat \btheta_n$ exists, then we can solve the following simple classification problem: \textit{Given an infinite trajectory generated from either $\Z$ or $\X$ with $\btheta_Z \neq \btheta_X$, can we identify which process generated the trajectory with arbitrarily high accuracy?} If a consistent estimator exists, then the answer is positive. This is because if $\hat \btheta_n$ converges to $\btheta_Z$ (resp. to $\btheta_X$), then we know that the trajectory was generated by $\Z$ (resp. by $\X$). However, if $|| \L_{\Z} - \L_{\X} ||_{TV}<1$, then it is not possible to always make the correct classification. This can be seen through an analogy with the univariate setting: given an observation $x\in\mathbb{R}$, we want to determine whether this observation was generated from either $X_1 \sim f_{\theta_1}(x)$  or $X_2 \sim f_{\theta_2}(x)$. If $|| \L_{X_1} - \L_{X_2} ||_{TV} < 1$, as on the left-hand-side of Figure \ref{fig:TVDPicture}, it is not always possible to correctly classify the observation, whereas if $|| \L_{X_1} - \L_{X_2} ||_{TV} = 1$, as on the right-hand-side of Figure \ref{fig:TVDPicture}, it is always possible. Coming back to CBPs, if we are not able to classify an infinitely long trajectory as coming from $\Z$ or $\X$ (i.e.\ because $|| \L_{\Z} - \L_{\X} ||_{TV}<1$), then no consistent estimator exists for $\btheta$.

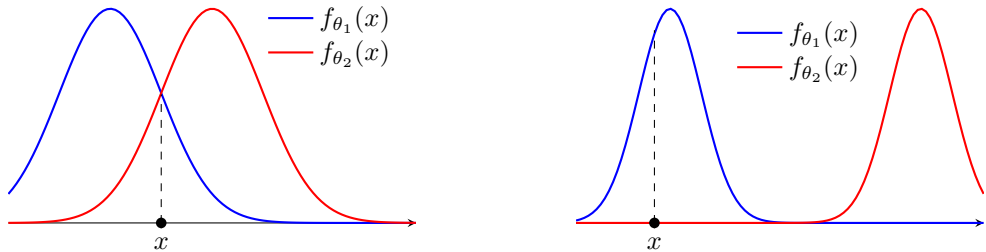
\begin{figure}[t]
    \centering
    \subfloat[]{{
        \begin{tikzpicture}
            \begin{axis}[
                domain=-3:5,
                samples=100,
                axis x line=middle,
                axis y line=none, % Remove y-axis
                legend pos=north east,
                legend style={draw=none},
                width=7cm, height=5cm,
                xtick={0},
                xticklabels={$x$},
                ytick=\empty
            ]
                \addplot[blue, thick] {exp(-0.5*(x+1)^2)};
                \addlegendentry{$f_{\theta_1}(x)$}
                \addplot[red, thick] {exp(-0.5*(x-1)^2)};
                \addlegendentry{$f_{\theta_2}(x)$}
                
                % Highlight x_0
                \draw[dashed] (axis cs:0,0) -- (axis cs:0,0.6);
                \node[anchor=north] at (axis cs:0,-0.05) {};
        
                % Big dot at former y-axis
                \fill (axis cs:0,0) circle[radius=2pt];
            \end{axis}
        \end{tikzpicture}
    }}
    \hspace{2cm}
    \subfloat[]{{
        \begin{tikzpicture}
            \begin{axis}[
                domain=-1:12,
                samples=100,
                axis x line=middle,
                axis y line=none, % Remove y-axis
                legend style={at={(0.55,0.6)}, anchor=south},
                legend style={draw=none},
                width=7cm, height=5cm,
                xtick={1.5},
                xticklabels={$x$},
                ytick=\empty
            ]
                \addplot[blue, thick] {exp(-0.5*(x-2)^2)};
                \addlegendentry{$f_{\theta_1}(x)$}
                \addplot[red, thick] {exp(-0.5*(x-10)^2)};
                \addlegendentry{$f_{\theta_2}(x)$}
                
                % Highlight x_0
                \draw[dashed] (axis cs:1.5,0) -- (axis cs:1.5,0.9);
                \node[anchor=north] at (axis cs:1.5,-0.05) {};
    
                % Big dot at former y-axis
                \fill (axis cs:1.5,0) circle[radius=2pt];
            \end{axis}
        \end{tikzpicture}
    }}
    \caption{Left: Two overlapping densities with $|| \L_{X_1} - \L_{X_2} ||_{TV}<1$; we do not know if the observation $x$ was generated from $f_{\theta_1}$ or $f_{\theta_2}$. Right: Two non-overlapping densities with $|| \L_{X_1} - \L_{X_2} ||_{TV}=1$; we can be certain that the observation $x$ was generated from $f_{\theta_1}$.}
    \label{fig:TVDPicture}
\end{figure}

For two supercritical CBPs in $\Pi$,
the following result relates the total variation distance between the distributions of the two processes and the total variation distance between the distributions of their one-step transitions.

\begin{lemma}\label{lma:CBPTVDLimitApproachesZero}
    Let $\Z, \X \in \Pi$ be two CBPs satisfying
    \begin{equation}\label{eqn:PolynomialDecreasingBound}
        ||\L_{Z_1|Z_0=z_0} - \L_{X_1|X_0=z_0}||_{TV} = O(z_0^{-q})\quad\textrm{for some $q > 0$.}
    \end{equation}
    Then $\lim_{z_0\to\infty} || \L_{\Z} - \L_{\X} ||_{TV} = 0$.
\end{lemma}

In the context of our discussion above, note that if $\lim_{z_0\to\infty} || \L_{\Z} - \L_{\X} ||_{TV} = 0$, then $|| \L_{\Z} - \L_{\X} ||_{TV} < 1$ for all sufficiently large $z_0$.
By combining Lemma \ref{lma:CBPTVDLimitApproachesZero} with the above relationship between the total variation distance and the non-existence of consistent estimators, we obtain the next result, which we will further refine in Sections \ref{sec:CBPEstimationKnownPhi}--\ref{sec:CBPEstimationObservedProgenitors} (see Theorems \ref{thm:KnownPhiNoConsistentEstimation}, \ref{thm:LinearlyDivisibleNoConsistentEstimation}, and~\ref{thm:KnownProgenitorsNoConsistentEstimation}).

\begin{proposition}\label{ppn:CBPNoConsistentEstimation}
    If there exist two CBPs $\Z, \X \in \Pi$ satisfying \eqref{eqn:PolynomialDecreasingBound} but with $\btheta_Z \neq \btheta_X$, then no weakly consistent estimator for $\btheta$ exists on the set of unbounded growth.
\end{proposition}

Another way to express Proposition \ref{ppn:CBPNoConsistentEstimation} is through its contrapositive: if there exists a consistent estimator for $\btheta$ on the set of unbounded growth, then for any two CBPs $\Z, \X \in \Pi$ such that \eqref{eqn:PolynomialDecreasingBound} holds, we must have $\btheta_Z = \btheta_X$.

\subsection{CBPs with a known control function}\label{sec:CBPEstimationKnownPhi}

Suppose we want to estimate the parameters of a supercritical CBP $\Z$ whose control function, $\phi(\cdot)$, is known \textit{a priori} and does not need to be estimated.
In this case, the only unknown is the offspring distribution, $\xi$.
Here, we let $\Pi^{(k)}$  be a set of supercritical CBPs \textit{with a common, known control function $\phi(\cdot)$} in a given parametric family satisfying \eqref{eqn:BoundedControlFunctionMoments}. We assume that all offspring distributions $\xi$ of processes in $\Pi^{(k)}$ have finite third moments and lattice size one. As before, let $\btheta$ be a function from $\Pi^{(k)}$ to $\mathbb{R}^d$ representing the quantities we would like estimate.

When two processes in $\Pi^{(k)}$ have the same offspring mean and variance, we can show that Condition \eqref{eqn:PolynomialDecreasingBound} of Lemma \ref{lma:CBPTVDLimitApproachesZero} holds with $q = 1/2$:

\begin{lemma}\label{lma:KnownPhiOneStepTVDBound}
    If $\Z, \X \in \Pi^{(k)}$ have the same offspring mean $m$ and variance $\sigma^2$, then $|| \L_{Z_1|Z_0 = z_0} - \L_{X_1|X_0 = z_0} ||_{TV} = O(z_0^{-1/2})$.
\end{lemma}
By feeding Lemma \ref{lma:KnownPhiOneStepTVDBound} into Proposition \ref{ppn:CBPNoConsistentEstimation}, we obtain the following result.

\begin{theorem}\label{thm:KnownPhiNoConsistentEstimation}
 If there exist two CBPs $\Z, \X \in \Pi^{(k)}$ with the same offspring mean and variance but with different parameters $\btheta_Z \neq \btheta_X$, then no weakly consistent estimator for $\btheta$ exists on the set of unbounded growth.
\end{theorem}

We can note the resemblance of Theorem \ref{thm:KnownPhiNoConsistentEstimation} to \cite[Theorem~2]{lockhart82} on the non-existence of consistent estimators for BGWPs (see also Theorem \ref{thm:GWPNoConsistentEstimates}).
Similarly, Theorem \ref{thm:KnownPhiNoConsistentEstimation} does not tell us whether we \textit{can} estimate $m$ and $\sigma^2$---only that we cannot consistently estimate anything other than functions of $m$ and $\sigma^2$. \\

We now show that it is possible to construct consistent estimators for $m$ and $\sigma^2$ for processes in $\Pi^{(k)}$.
There are two cases of interest: first, when $m$ is known and $\sigma^2$ unknown, and second, when  both $m$ and $\sigma^2$ are unknown.
Let $I^{\varepsilon}_n := \{ k \in \{1, \dots, n\} : \varepsilon(Z_{k-1}) > 0 \}$.
In the first case, we propose 
\[
    \bar{\sigma}_n^2 := \frac{1}{| I^{\varepsilon}_n |} \sum_{k \in I^{\varepsilon}_n} \frac{(Z_k - m \cdot \varepsilon(Z_{k-1}))^2 - m^2 \cdot \nu^2(Z_{k-1})}{\varepsilon(Z_{k-1})}
\]
as an estimator for $\sigma^2$, while noting that if the distribution of $\phi(z)$ is known, then so too are  $\varepsilon(z) := \E\phi(z)$ and  $\nu^2(z) := \V(\phi(z))$.
In the second case, we propose
\begin{align*}
    \hat{m}_n &:= \frac{1}{| I^{\varepsilon}_n |} \sum_{k \in I^{\varepsilon}_n} \frac{Z_k}{\varepsilon(Z_{k-1})},
    &\quad
    \hat{\sigma}_n^2 &:= \frac{1}{| I^{\varepsilon}_n |} \sum_{k \in I^{\varepsilon}_n} \frac{(Z_k - \tilde{m}_n \cdot \varepsilon(Z_{k-1}))^2 - \tilde{m}_n^2 \cdot \nu^2(Z_{k-1})}{\varepsilon(Z_{k-1})}
\end{align*}
as estimators for $m$ and $\sigma^2$ respectively, where $\tilde{m}_n := Z_n / \varepsilon(Z_{n-1})$. \\

Under mild assumptions, we can show that $\hat{m}_n$, $\bar{\sigma}_n^2$, and $\hat{\sigma}_n^2$ are all consistent estimators. In the next theorem and throughout the rest of the paper, we denote the third central moment of the control function by $\iota(z) := \E(\phi(z) - \E\phi(z))^3$.

\begin{theorem}\label{thm:KnownPhiConsistentEstimators}
    Let $\Z\in \Pi^{(k)}$. Then, on the set $\{Z_n \to \infty \}$,
    \begin{enumerate}
        \item[(i)] $\hat{m}_n$ is a strongly consistent estimator for $m$.
    \end{enumerate}
    If we further assume that $\sup_{z \geq 1} \left\{ \frac{|\iota(z)|}{z} \right\} \leq c$ where and $\sup_{z \geq 1} \Big\{ \frac{\E(\phi(z) - \varepsilon(z))^4}{z^2} \Big\} \leq d$ for positive constants $c$ and $d$, and that $\E(\xi - m)^4$ is finite, then
    \begin{enumerate}
        \item[(ii)] $\bar{\sigma}^2_n$ is a strongly consistent estimator for $\sigma^2$, and
        \item[(iii)] $\hat{\sigma}^2_n$ is a weakly consistent estimator for $\sigma^2$.
    \end{enumerate}
\end{theorem}

\subsection{CBPs with an unknown control function}\label{sec:CBPEstimationUnknownControl}

In most cases, we would \textit{not} expect that the distribution of the control function is known beforehand; hence, it needs to be estimated alongside the offspring distribution.
Recall that a control function $\{\phi(z)\}_{z\in\N_0}$ is specified by a countably infinite set of random variables, one for each population size $z$.
To make the problem of estimating the distribution of $\phi(\cdot)$ tractable, we require some regularity conditions. First, to ensure that the (appropriately scaled) distribution of $\phi(z)$ converges to the normal distribution as $z \to \infty$, we assume that the control function is \textit{linearly-divisible}:
\begin{definition}\label{def:LinearDivisibility}
    The control function $\{\phi(z)\}_{z\in\N_0}$ is \textit{linearly-divisible} if there exists a function $l: \N_0 \to \N_0$ such that $l(z) = \Theta(z)$ (i.e.\ $0< \liminf_{z\to\infty} l(z)/z$ and $\limsup_{z\to\infty} l(z)/z <\infty$)
    and, for each $z \in \N_0$, there exist a set of i.i.d.\ random variables $\{\chi_i^{(z)}\}_{1\leq i \leq l(z)}$ such that $\phi(z) \stackrel{d}{=} \sum_{i=1}^{l(z)} \chi_i^{(z)}$.
\end{definition}

Recall that \eqref{eqn:BoundedControlFunctionMoments} provides a sufficient condition for a supercritical CBP to have a positive probability of unbounded growth. Here, we need to strengthen this condition to include a bound on the growth rate of the third central moment of the control function.
In particular, we  assume that there exist constants $a, b, c \geq 0$ such that
\begin{equation}\label{eqn:BoundedControlFunctionThirdMoment}
    \sup_{z \geq 1} \left\{ \frac{\varepsilon(z)}{z} \right\} \leq a, \quad
    \sup_{z \geq 1} \left\{ \frac{\nu^2(z)}{z} \right\} \leq b, \quad \text{and} \quad
    \sup_{z \geq 1} \left\{ \frac{|\iota(z)|}{z} \right\} \leq c.
\end{equation}

We let $\Pi^{(u)}$ be a set of supercritical CBPs in a given parametric family with a linearly-divisible control function satisfying \eqref{eqn:BoundedControlFunctionThirdMoment}.  We assume that all offspring distributions $\xi$ of processes in $\Pi^{(u)}$ have finite third moment, and lattice size one. As before, let $\btheta$ be a function from $\Pi^{(u)}$ to $\mathbb{R}^d$ representing the quantities we would like estimate. We then have the following analogue of Lemma \ref{lma:KnownPhiOneStepTVDBound}:

\begin{lemma}\label{lma:LinearlyDivisibleOneStepTVDBound}
    If $\Z,\X \in \Pi^{(u)}$ are such that
    \begin{equation}\label{eqn:CBPMeanSDSquareRootBound}
        |\E(Z_1|Z_0=z) - \E(X_1|X_0=z)| = O(z^{r/2})
        \quad \text{and} \quad
        |\V(Z_1|Z_0=z) - \V(X_1|X_0=z)| = O(z^r)
    \end{equation}
    for some $r < 1$, then there exists $q > 0$ such that $||\L_{Z_1|Z_0=z_0} - \L_{X_1|X_0=z_0}||_{TV} = O(z_0^{-q})$.
\end{lemma}
By feeding Lemma \ref{lma:LinearlyDivisibleOneStepTVDBound} into Proposition \ref{ppn:CBPNoConsistentEstimation}, we obtain the analogue of Theorem \ref{thm:KnownPhiNoConsistentEstimation}:

\begin{theorem}\label{thm:LinearlyDivisibleNoConsistentEstimation}
    If there exist two CBPs $\Z, \X \in \Pi^{(u)}$ that satisfy \eqref{eqn:CBPMeanSDSquareRootBound} but with different parameters $\btheta_Z \neq \btheta_X$, then no weakly consistent estimator for $\btheta$ exists on the set of unbounded growth.
\end{theorem}

To understand why the conditions in \eqref{eqn:CBPMeanSDSquareRootBound} are sufficient for the non-existence result in Theorem \ref{thm:LinearlyDivisibleNoConsistentEstimation}, first observe that under the assumption of linear divisibility, for large $z$, we have
\begin{equation}\label{eqn:CBPTransitionProbApproximations}
\begin{aligned}
    (X_1 | X_0 = z) &\stackrel{d}{\approx} \E(X_1 | X_0=z) + \sqrt{\V(X_1 | X_0 = z)} \cdot \text{N}(0,1) \\ 
    (Z_1 | Z_0 = z) &\stackrel{d}{\approx} \E(Z_1 | Z_0=z) + \sqrt{\V(Z_1 | Z_0 = z)} \cdot \text{N}(0,1).
\end{aligned}
\end{equation}
Then, rearranging the second line gives
\[
    (Z_1 | Z_0 = z)
    \stackrel{d}{\approx}
    \E(X_1 | X_0 = z) + \sqrt{\V(X_1 | X_0 = z)} \cdot N\left(
        \frac{\E(Z_1 | Z_0=z) - \E(X_1 | X_0=z)}{\sqrt{\V(X_1 | X_0=z)}} \,, \, \frac{\V(Z_1 | Z_0=z)}{\V(X_1 | X_0=z)}
    \right).
\]
Now, note that if
\begin{equation}\label{eqn:CBPTransitionProbApproximationLimit}
    \frac{\E(Z_1 | Z_0=z) - \E(X_1 | X_0=z)}{\sqrt{\V(X_1 | X_0=z)}} \to 0
    \quad\text{and}\quad
    \frac{\V(Z_1 | Z_0=z)}{\V(X_1 | X_0=z)} \to 1 \quad \text{as } z \to \infty,
\end{equation}
then the two conditional distributions in \eqref{eqn:CBPTransitionProbApproximations} become increasingly similar as $z \to \infty$, which
makes it impossible to identify which of the two processes $\X$ and $\Z$ an infinite trajectory of population sizes comes from with arbitrarily high accuracy.
Condition \eqref{eqn:CBPMeanSDSquareRootBound} (together with \eqref{eqn:BoundedControlFunctionThirdMoment}) actually implies that \eqref{eqn:CBPTransitionProbApproximationLimit} holds.

\begin{exmp}\label{exmp:PoissonControlEstimation}
    Let $\Pi^{(u)}$ be the family of supercritical CBPs with control functions of the form $\phi(z) \sim \poi(z + az^q)$ for a fixed value $q > 0$, and whose offspring distributions have common fixed mean $m > 1$ and variance $\sigma^2 > 0$.
    For a process $\Z \in \Pi^{(u)}$, we consider the estimation of $\theta = a$ for different values of $q$. In this case, $\E(Z_1 | Z_0=z) = m\,(z + az^q)$ and $\V(Z_1 | Z_0=z) = (m^2+\sigma^2)(z + az^q)$.
    \begin{enumerate}
        \item[(i)] When $q < 1/2$, \eqref{eqn:CBPMeanSDSquareRootBound} holds for any $r \in (2q, 1)$ if, for example, we let $\Z$ be the process with $a=0$ and $\X$ be the process with $a=1$. By Theorem \ref{thm:LinearlyDivisibleNoConsistentEstimation}, this implies that $\theta=a$ cannot be consistently estimated when $q < 1/2$.

        \item[(ii)] When $q > 1/2$, \eqref{eqn:CBPMeanSDSquareRootBound} does not hold for any two processes with different values of $\theta=a$. In this case, Theorem \ref{thm:LinearlyDivisibleNoConsistentEstimation} does not rule out the existence of a consistent estimator for $\theta=a$. In fact, we can show that 
        \[
            \bar{a}_n := \frac{Z_n - m Z_{n-1}}{m Z_{n-1}^q}
        \]
        is a consistent estimator for $\theta=a$ on the event $\{ Z_n \to \infty \}$. Indeed, 
        \[
            \E \bar{a}_n
            = \E\bigg(\frac{\E(Z_n | Z_{n-1}) - m Z_{n-1}}{m Z_{n-1}^q}\bigg)
            = a,
        \]
        and
        \[
            \V(\bar{a}_n)
            = \E\bigg( \frac{\V(Z_n | Z_{n-1})}{m^2 Z_{n-1}^{2q}} \bigg)
            = \frac{\sigma^2 + m^2}{m^2} \, \Big( \E\big( Z_{n-1}^{1-2q} \big) + a \cdot \E\big( Z_{n-1}^{-q} \big)
            \Big)
            \to 0 \quad \textrm{as $n\to\infty$},
        \]
        since $Z^{1-2q}_{n-1} \to 0$ and $Z^{-q}_{n-1} \to 0$ as $n \to \infty$ on  $\{ Z_n \to \infty \}$.
        Chebyshev's inequality then implies $\bar{a}_n \stackrel{P}{\to} a$ as $n \to \infty$,  that is, $\bar{a}_n$ is a consistent estimator of $a$ on  $\{ Z_n \to \infty \}$.
    \end{enumerate}
\end{exmp}

Let us now consider a more specific class of supercritical processes, $\Pi^{(u*)}$, which is a subclass of $\Pi^{(u)}$ where the (unknown) control functions have linear mean and variance, that is, $\varepsilon(z) = \alpha z$ and $\nu^2(z) = \beta z$ for some $\alpha, \beta > 0$. This implies that $\E(Z_1 | Z_0 = z) = m \alpha z$ and $\V(Z_1 | Z_0 = z) = \sigma^2 \alpha z + m^2 \beta z$.
Suppose we would like to estimate $\btheta^{(1)}:=(m,\, \sigma^2,\, \alpha,\, \beta)$. If $\Z ,\X \in \Pi^{(u*)}$ are such that $\btheta^{(1)}_Z \neq \btheta^{(1)}_X$,
\begin{equation}\label{eqn:LinearAssumptionSquareRootBound}
    m_Z \alpha_Z = m_X \alpha_X,
    \quad\text{and}\quad
    \sigma^2_Z \alpha_Z + m^2_Z \beta_Z = \sigma^2_X \alpha_X + m^2_X \beta_X,
\end{equation} 
then \eqref{eqn:CBPMeanSDSquareRootBound} is trivially satisfied, since $|\E(Z_1|Z_0=z) - \E(X_1|X_0=z)| = 0$ and $|\V(Z_1|Z_0=z) - \V(X_1|X_0=z)| = 0$.
For a concrete example, take $\Z$ with $\phi_Z(z) \sim \bin(2,\, 1/2)$ and $\xi_Z \sim \bin(8z,\, 1/8)$, and $\X$ with $\phi_X(z) \sim \poi(2z)$ and $\xi_Z \sim \geo(1/2)$.
Therefore, by Theorem \ref{thm:LinearlyDivisibleNoConsistentEstimation}, $\btheta^{(1)}$ cannot be estimated consistently. 

On the other hand, if we would like to estimate $\btheta^{(2)}=(g,\, h):=(m\alpha,\, \sigma^2\alpha+m^2\beta)$ (where $g$ is the \textit{mean growth rate} of the process), then \eqref{eqn:CBPMeanSDSquareRootBound} does not hold for any $\Z ,\X \in \Pi^{(u*)}$ with $\btheta^{(2)}_Z \neq \btheta^{(2)}_X$.
With $I_n := \{ k \in \{1, \dots, n\} : Z_{k-1} > 0 \}$, we let
\[
    \hat{g}_n := \frac{1}{| I_n |} \sum_{k\in I_n} \frac{Z_k}{Z_{k-1}},
\]
\[
    \bar{h}_n := \frac{1}{| I_n |} \sum_{k\in I_n} \frac{(Z_k - m \alpha \cdot Z_{k-1})^2}{Z_{k-1}},
    \quad\text{ and}\quad
    \hat{h}_n := \frac{1}{| I_n |} \sum_{k\in I_n} \frac{(Z_k - \tilde{g}_n \cdot Z_{k-1})^2}{Z_{k-1}},
\]
where $\tilde{g}_n = Z_n / Z_{n-1}$, and we note that $\bar{h}_n$ assumes the value of $m\alpha$ to be known.
We now show that $\hat{g}_n,$ $ \bar{h}_n$, and $ \hat{h}_n$  are consistent for their respective parameters.

\begin{theorem}\label{thm:LinearMeanVarPhiConsistentEstimators}
    Let $\Z \in \Pi^{(u*)}$. Then, on $\{ Z_n \to \infty \}$,
    \begin{enumerate}
        \item[(i)] $\hat{g}_n$ is a strongly consistent estimator for $g=m\alpha$.
    \end{enumerate}
    If we further assume \eqref{eqn:BoundedControlFunctionThirdMoment}, that there exists a positive constant $d$ such that $\sup_{z \geq 1} \Big\{ \frac{\E(\phi(z) - \varepsilon(z))^4}{z^2} \Big\} \leq d$, and that $\E(\xi - m)^4$ is finite, then
    \begin{enumerate}
        \item[(ii)] $\bar{h}_n$ is a strongly consistent estimator for $h=\sigma^2 \alpha + m^2 \beta$.
        \item[(iii)] $\hat{h}_n$ is a weakly consistent estimator for $h=\sigma^2 \alpha + m^2 \beta$.
    \end{enumerate}
\end{theorem}

We note that if $m$ and $\sigma^2$ are known, then consistent estimators for $\alpha$ and $\beta$ are given by $\hat{g}_n / m$ and $(m \, \hat{h}_n - \sigma^2 \, \hat{g}_n) / m^3$, respectively.
If $m$ and $\sigma^2$ are unknown, then estimating them consistently from the data (in addition to $\alpha$ and $\beta$) requires a more detailed observation scheme. We explore this in the next section.

\subsection{CBPs with observed progenitor numbers}\label{sec:CBPEstimationObservedProgenitors}

Here we assume that both the population size \textit{and the number of progenitors} are observed at each generation, that is, we observe the outcomes of $Z_0,\, \phi(Z_0),\, Z_1,\, \phi(Z_1),\, \dots,\, \phi(Z_{n-1}),\, Z_n$.
We consider processes belonging to a set $\Pi^{(p)}$, satisfying the same conditions as $\Pi^{(u)}$ in Section \ref{sec:CBPEstimationUnknownControl}, plus the additional assumptions that the processes in $\Pi^{(p)}$ satisfy $\liminf_{z\to\infty} \nu^2(z)/z > 0$, and that they have linearly divisible control functions, $\phi(z) \stackrel{d}{=} \sum_{i=1}^{l(z)} \chi_i^{(z)}$, such that there exists a constant $\eta > 0$ and a sequence $\{x_z\}_{z\in\N_1}$ with
\begin{equation}\label{eqn:UniformlyLatticeSizeOne}
    \P(\chi^{(z)} = x_z) \wedge \P(\chi^{(z)} = x_z + 1) \geq \eta
\end{equation}
for all $z \in \N_1$. Equation \eqref{eqn:UniformlyLatticeSizeOne} is a technical condition that is satisfied for many natural models.

\begin{lemma}\label{lma:KnownProgenitorsOneStepTVDBound}
    If $\Z,\X\in \Pi^{(p)}$ are such that
    \begin{equation}\label{eqn:KnownProgenitorsOneStepTVDBoundConditions}
        m_Z = m_X, \quad
        \sigma^2_Z = \sigma^2_X, \quad
        |\varepsilon_Z(z) - \varepsilon_X(z)| = O(z^{r/2}),
        \quad \textrm{and} \quad
        |\nu^2_Z(z) - \nu^2_X(z)| = O(z^r)
    \end{equation}
    for some $r < 1$, then there exists $q > 0$ such that
    \begin{equation}\label{eqn:KnownProgenitorsOneStepTVDBound}
        || \L_{(\phi_Z(Z_0),\, Z_1)|Z_0=z_0} - \L_{(\phi_X(X_0),\, X_1)|X_0=z_0}||_{TV} = O(z_0^{-q}).
    \end{equation}
\end{lemma}

With our new observation scheme, Lemma \ref{lma:CBPTVDLimitApproachesZero} and Proposition \ref{ppn:CBPNoConsistentEstimation} cannot be directly applied. However, close equivalents of these results exist, and together with Lemma \ref{lma:KnownProgenitorsOneStepTVDBound} lead to the following result:

\begin{theorem}\label{thm:KnownProgenitorsNoConsistentEstimation}
    Suppose that both the population sizes and the progenitor numbers are observed. If there exist two CBPs $\Z, \X \in \Pi^{(p)}$ that satisfy \eqref{eqn:KnownProgenitorsOneStepTVDBoundConditions} but with different parameters $\btheta_Z \neq \btheta_X$, then no weakly consistent estimator for $\btheta$ exists on the set of unbounded growth.
\end{theorem}

Note that, even under this new observation scheme, Theorem \ref{thm:KnownProgenitorsNoConsistentEstimation} implies that the parameter $\theta=a$ in Example \ref{exmp:PoissonControlEstimation} can still not be estimated consistently when $q<1/2$.

Consider again the class of supercritical CBPs $\Z\in\Pi^{(u*)}$, with control functions $\phi(\cdot)$ satisfying $\varepsilon(z) = \alpha z$ and $\nu^2(z) = \beta z$ for $\alpha, \beta > 0$. Recall from the previous section that if only the population sizes are observed at each generation, then  $\btheta^{(1)}:=(m,\, \sigma^2,\, \alpha,\, \beta)$ cannot be estimated consistently.
Under the current observation scheme and with $I^{\phi}_n := \{ k \in \{1, \dots, n\} : \phi(Z_{k-1}) > 0 \}$ and $I^-_n := \{ k \in \{0, \dots, n - 1 \} : Z_k > 0 \}$, we consider the estimators
\begin{align*}
    \hat{m}_n &:= \frac{1}{| I^{\phi}_n |} \sum_{k \in I^{\phi}_n} \frac{Z_k}{\phi(Z_{k-1})},
    &\quad
    \hat{\alpha}_n &:= \frac{1}{| I^-_n |} \sum_{k \in I^-_n} \frac{\phi(Z_k)}{Z_k}, \\
    \bar{\sigma}_n^2 &:= \frac{1}{| I^{\phi}_n |} \sum_{k \in I^{\phi}_n} \frac{(Z_k - m \cdot \phi(Z_{k-1}))^2}{\phi(Z_{k-1})},
    &\quad
    \bar{\beta}_n &:= \frac{1}{| I^-_n |} \sum_{k \in I^-_n} \frac{(\phi(Z_k) - \alpha \cdot Z_k)^2}{Z_k}, \\
    \hat{\sigma}_n^2 &:= \frac{1}{| I^{\phi}_n |} \sum_{k \in I^{\phi}_n} \frac{(Z_k - \tilde{m}_n \cdot \phi(Z_{k-1}))^2}{\phi(Z_{k-1})},
    &\quad
    \hat{\beta}_n &:= \frac{1}{| I^-_n |} \sum_{k \in I^-_n} \frac{(\phi(Z_k) - \tilde{\alpha}_{n-1} \cdot Z_k)^2}{Z_k},
\end{align*}
where $\tilde{m}_n := Z_n / \phi(Z_{n-1})$ and $\tilde{\alpha}_n := \phi(Z_n) / Z_n$.
Note that $\bar{\sigma}_n^2$ and $\bar{\beta}_n$ require knowledge of $m$ and $\alpha$, respectively, while $\hat{\sigma}_n^2$ and $\hat{\beta}_n$ do not.
The next proposition shows that the above estimators are consistent for their respective parameters.

\begin{theorem}\label{thm:ObservedProgenitorsConsistentEstimators}
    Suppose that both the population sizes  and the progenitor numbers  are observed. If $\Z\in\Pi^{(u*)}$, then on $\{ Z_n \to \infty \}$,
    \begin{enumerate}
        \item[(i)] $\hat{m}_n$ is a strongly consistent estimator for $m$,
        \item[(ii)] $\hat{\alpha}_n$ is a strongly consistent estimator for $\alpha$.
    \end{enumerate}
    If we further assume \eqref{eqn:BoundedControlFunctionThirdMoment}, that there exists a positive constant $d$ such that $\sup_{z \geq 1} \Big\{ \frac{\E(\phi(z) - \varepsilon(z))^4}{z^2} \Big\} \leq d$, and that $\E(\xi - m)^4$ is finite, then
    \begin{enumerate}
        \item[(iii)] $\bar{\sigma}^2_n$ is a strongly consistent estimator for $\sigma^2$, 
        \item[(iv)] $\hat{\sigma}^2_n$ is a weakly consistent estimator for $\sigma^2$,
        \item[(v)] $\bar{\beta}_n$ is a strongly consistent estimator for $\beta$, and
        \item[(vi)] $\hat{\beta}_n$ is a weakly consistent estimator for $\beta$.
    \end{enumerate}
\end{theorem}

\section{Concluding remarks}

 As mentioned in Section \ref{sec:Introduction}, a common rule of thumb in population ecology is that demographic and environmental stochasticity should not be simultaneously estimated from a single trajectory of population size counts. This rule is supported by Theorem \ref{thm:LinearlyDivisibleNoConsistentEstimation} and its application to the processes in $\Pi^{(u*)}$ with parameter $\btheta^{(1)}:=(m,\, \sigma^2,\, \alpha,\, \beta)$. Our results suggest three different ways to address this limitation:
 (i) Use an independent data source, or expert knowledge, to estimate environmental stochasticity (i.e., the control function) first, then estimate demographic stochasticity (i.e., $m$ and $\sigma^2$) while treating the distribution of the control function as known (supported by Theorem \ref{thm:KnownPhiConsistentEstimators} in setting \textbf{S1});
(ii) Rely on expert knowledge on the species to estimate the demographic parameters, and use population size counts to estimate the control function parameters only (for example $\alpha$ and $\beta$, as supported by Theorem \ref{thm:LinearMeanVarPhiConsistentEstimators} and commentary below);
(iii) Collect additional data beyond a single trajectory of population size counts, as in setting \textbf{S3} (Theorem \ref{thm:ObservedProgenitorsConsistentEstimators}).

We propose two future research directions. First, consider the variance decomposition:
\[
    \V(Z_1 | Z_0=z_0) = \E(\phi(z_0)) \cdot \sigma^2 + \V(\phi(z_0)) \cdot m^2,
\]
where the first term represents demographic stochasticity (as it does not depend on  $\V(\phi(z_0))$), and the second term represents environmental stochasticity (as it does not depend on  $\sigma^2$).
Recall that the assumptions in \eqref{eqn:BoundedControlFunctionThirdMoment} effectively imply that the mean and variance of the control function grow linearly in the population size $z$. This means that the demographic and environmental stochasticity both grow linearly.
However, population modellers often assume that the variance of the environmental component (in our case, the control function) grows faster than linearly in $z$, i.e. faster than the variance of the demographic component. It would be valuable to investigate what can and cannot be consistently estimated in this setting.
Second, in settings \textbf{S2} and \textbf{S3}, we may want to know what can and cannot be consistently estimated if we relax the assumption of linear divisibility of the control function.

\section{Proofs of non-existence results}

\subsection{Proofs for Section \ref{sec:Framework}}

We first introduce two lemmas that will be used in the proof of Lemma \ref{lma:CBPTVDLimitApproachesZero} and Proposition \ref{ppn:CBPNoConsistentEstimation}.

\begin{lemma}\label{lma:TVDDecreasingBound}
    Let $\{Z_n\}_{n \in \N_0}$ and $\{X_n\}_{n \in \N_0}$ be two Markov chains taking values on $\N_0$.
    For $j,k \in \N_1$ such that $j < k$, if there exists a monotonically decreasing function $\K: [M, \infty) \to \R_{\geq 0}$, $M \in \N_0$, such that
    \begin{equation}\label{eqn:KStepInductionAssumption}
        || \L_{(Z_{j+1},\dots,Z_k)|Z_j=u_j} - \L_{(X_{j+1},\dots,X_k)|X_j=u_j}||_{TV} \leq \K(u_j)
        \quad \forall u_j \geq M,
    \end{equation}
    then we have 
    \begin{align*}
        &|| \L_{(Z_j,\dots,Z_k)|Z_{j-1}=u_{j-1}} - \L_{(X_j,\dots,X_k)|X_{j-1}=u_{j-1}}||_{TV} \\
        &\leq || \L_{Z_j|Z_{j-1} = u_{j-1}} - \L_{X_j|X_{j-1} = u_{j-1}} ||_{TV} + \P(Z_j \leq N | Z_{j-1}=u_{j-1}) + \K(N+1)
    \end{align*}
    for all $u_{j-1}, N \geq M$.
\end{lemma}

\begin{proof}
    We write $p_{Z_n}(i,j) := \P(Z_n = j|Z_{n-1} = i)$ and $p_{X_n}(i,j) := \P(X_n = j|X_{n-1} = i)$.
    From the sum representation of the total variation distance (see \cite[Proposition 4.2]{levin17}) and the triangle inequality, 
    \begingroup
    \allowdisplaybreaks
    \begin{align*}
        &|| \L_{(Z_j, \dots, Z_k)|Z_{j-1} = u_{j-1}} - \L_{(X_j, \dots, X_k)|X_{j-1} = u_{j-1}} ||_{TV} \\
        &= \frac{1}{2} \sum_{u_j, \dots, u_k \geq 0} \bigg| \prod_{i=j}^k p_{X_i}(u_{i-1}, u_i) - \prod_{i=j}^k p_{Z_i}(u_{i-1}, u_i) \bigg| \\
        &\leq \frac{1}{2} \sum_{u_j, \dots, u_k \geq 0} \bigg| \prod_{i=j}^k p_{X_i}(u_{i-1}, u_i) - p_{Z_j}(u_{j-1}, u_j) \cdot \prod_{i=j+1}^{k} p_{X_i}(u_{i-1}, u_i) \bigg| \\
        &\quad + \frac{1}{2} \sum_{u_j, \dots, u_k \geq 0} \bigg| p_{Z_j}(u_{j-1}, u_j) \cdot \prod_{i=j+1}^{k} p_{X_i}(u_{i-1}, u_i) - \prod_{i=j}^k p_{Z_i}(u_{i-1}, u_i) \bigg| \\
        &\leq \frac{1}{2} \sum_{u_j=0}^{\infty} \big| p_{X_j}(u_{j-1}, u_j) - p_{Z_j}(u_{j-1}, u_j) \big| \\
        &\quad + \frac{1}{2} \sum_{u_j, \dots, u_k \geq 0} p_{Z_j}(u_{j-1}, u_j) \cdot \bigg| \prod_{i=j+1}^{k} p_{X_i}(u_{i-1}, u_i) - \prod_{i=j+1}^k p_{Z_i}(u_{i-1}, u_i) \bigg| \\
        &= || \L_{Z_j|Z_{j-1} = u_{j-1}} - \L_{X_j|X_{j-1} = u_{j-1}} ||_{TV} \\
        &\quad + \sum_{u_j=0}^{\infty} p_{Z_j}(u_{j-1}, u_j) \cdot || \L_{(Z_{j+1}, \dots, Z_k)|Z_j = u_j} - \L_{(X_{j+1}, \dots, X_k)|X_j = u_j} ||_{TV} \\
        &\leq || \L_{Z_j|Z_{j-1} = u_{j-1}} - \L_{X_j|X_{j-1} = u_{j-1}} ||_{TV} \\
        &\quad + \P(Z_j \leq N | Z_{j-1}=u_{j-1}) + \sum_{u_j=N+1}^{\infty} p_{Z_j}(u_{j-1},u_j) \cdot \K(u_j)\qquad\textrm{for $N\geq M$, by \eqref{eqn:KStepInductionAssumption}.}
    \end{align*}
    \endgroup
 It follows from the assumption that $\K$ is monotonically decreasing that
    \[
        \sum_{u_j=N+1}^{\infty} p_{Z_j}(u_{j-1},u_j) \cdot \K(u_j) \leq \K(N+1) \cdot \sum_{u_j=N+1}^{\infty} p_{Z_j}(u_{j-1},u_j) \leq \K(N+1). \\
    \]
    Combining these two inequalities then yields our desired result.
\end{proof}

\begin{lemma}\label{RecursiveEquationLimit}
    Let $\{f_k\}_{k\in\N_1}$, $f_k: \R_{\geq 0} \to \R_{\geq 0}$, be a recursively-defined set of functions such that
    \[
        \lim_{z\to\infty} f_1(z) = 0
        \quad\text{and}\quad
        f_k(z) = \sum_{i=1}^n \frac{c_i}{z^{q_i}} + f_{k-1}(b\cdot z) \;\text{ for } k > 0,
    \]
    for constants $n\in\N_1$, $c_1, \dots, c_n > 0$, $q_1, \dots, q_n > 0$, and $b > 1$. Then $\lim_{z\to\infty} \lim_{k\to\infty} f_k(z) = 0$.
\end{lemma}

\begin{proof}
    For a given $k \in \N_1$, we can expand $f_k$ iteratively to see that
    % \begin{align*}
    %     f_k(z) &= \sum_{i=1}^n \frac{c_i}{z^{q_i}} \cdot \sum_{j=0}^{k-2} \big(b^{-q_i}\big)^j + f_1(b^{k-1} \cdot z) \\
    %     &= \sum_{i=1}^n \frac{c_i}{z^{q_i}} \cdot \frac{1 - b^{-(k-1)q_i}}{1-b^{-q_i}} + f_1(b^{k-1} \cdot z).
    % \end{align*}
     $$
        f_k(z) = \sum_{i=1}^n \frac{c_i}{z^{q_i}} \cdot \sum_{j=0}^{k-2} \big(b^{-q_i}\big)^j + f_1(b^{k-1} \cdot z) = \sum_{i=1}^n \frac{c_i}{z^{q_i}} \cdot \frac{1 - b^{-(k-1)q_i}}{1-b^{-q_i}} + f_1(b^{k-1} \cdot z).
    $$
    Since $b > 1$ and $q_i > 0$ for all $1\leq i \leq n $, each $b^{-q_i} < 1$. Therefore, given that $\lim_{x\to\infty} f_1(x) = 0$, we have that
    \[
        \lim_{k\to\infty} f_k(z) = \sum_{i=1}^n \frac{c_i}{(1-b^{-q_i}) \cdot z^{q_i}},
    \]
    from which our desired result follows by letting $z\to\infty$.
\end{proof}

Given the above lemmas, we now proceed to prove Lemma \ref{lma:CBPTVDLimitApproachesZero}, from which Proposition \ref{ppn:CBPNoConsistentEstimation} follows.

\begin{proof}[Proof of Lemma \ref{lma:CBPTVDLimitApproachesZero}.]
    Since $\Z$ is assumed to be supercritical, $\liminf_{z\to\infty} \tau_Z(z) > 1$. Hence for any $t$ such that $1 < t < \liminf_{z\to\infty} \tau_Z(z)$, there exists $M_1 \in \N_1$ such that for all $z \geq M_1$, $\varepsilon(z) \cdot m > t \cdot z$.
    In addition, if $||\L_{Z_1|Z_0=z_0} - \L_{X_1|X_0=z_0}||_{TV} = O(z_0^{-q})$, then there exists $s > 0$ such that for all $z_0 \geq M_2$, $M_2 \in \N_1$, $||\L_{Z_1|Z_0=z_0} - \L_{X_1|X_0=z_0}||_{TV} \leq s \cdot z_0^{-q}$.

    Given such values of $s$ and $t$, and for $M_3 := M_1 \vee M_2$, we can show by induction that for any $j,k \in \N_1$ with $k \geq 1$ and $1 \leq j \leq k$,
    \begin{equation}\label{eqn:CBPKStepRecursiveRelationship}
        || \L_{(Z_j, \dots, Z_k)|Z_{j-1} = u_{j-1}} - \L_{(X_j, \dots, X_k)|X_{j-1} = u_{j-1}} ||_{TV} \leq \K_{k-j+1}(u_{j-1}),
    \end{equation}
    for a decreasing function $\K_{k-j+1}: [M_3, \infty) \to \R_{\geq 0}$ given by $\K_1(z) :=s \cdot z^{-q}$ and for $j < k$, $\K_{k-j+1}(z) := s \cdot z^{-q} + \frac{(a \sigma^2 + b m^2)}{(1-\alpha)^2 t^2 \cdot z} + \K_{k-j}(\alpha t \cdot z)$, where $a$ and $b$ are given in \eqref{eqn:BoundedControlFunctionMoments},  and where $\alpha \in (1/t,1)$. \\

    \textit{Base case:} Since CBPs are time-homogeneous, it is immediate from our assumption on the one-step TVD bound between $\Z$ and $\X$ that
    \[
        ||\L_{Z_k|Z_{k-1}=u_{k-1}} - \L_{X_k|X_{k-1}=u_{k-1}}||_{TV} \leq \frac{s}{u_{k-1}^q}\qquad\textrm{for any $u_{k-1} \geq M_3$.}
    \]

    \textit{Induction step:} For $j, u_j \in \N_1$ such that $j < k$ and $u_j \geq M_3$, let us assume that
    \begin{equation*}
        || \L_{(Z_{j+1},\dots,Z_k)|Z_j=u_j} - \L_{(X_{j+1},\dots,X_k)|X_j=u_j}||_{TV} \leq \K_{k-j}(u_j),
    \end{equation*}
    where $\K_{k-j}: [M_3, \infty) \to \R_{\geq 0}$ is a monotonically decreasing function. Applying Lemma \ref{lma:TVDDecreasingBound} in the first step and Chebyshev's inequality in the second, we obtain
    \begin{align*}
        &|| \L_{(Z_j,\dots,Z_k)|Z_{j-1}=u_{j-1}} - \L_{(X_{j+1},\dots,X_k)|X_{j-1}=u_{j-1}}||_{TV} \\
        &\leq || \L_{Z_j|Z_{j-1} = u_{j-1}} - \L_{X_j|X_{j-1} = u_{j-1}} ||_{TV} + \P(Z_{j} \leq N | Z_{j-1} = u_{j-1}) + \K_{k-j}(N+1) \\
        &\leq || \L_{Z_j|Z_{j-1} = u_{j-1}} - \L_{X_j|X_{j-1} = u_{j-1}} ||_{TV} + \frac{\varepsilon(u_{j-1}) \cdot \sigma^2 + \nu^2(u_{j-1}) \cdot m^2}{(\varepsilon(u_{j-1}) \cdot m - N)^2} + \K_{k-j}(N+1)
    \end{align*}
    for $u_{j-1} \geq M_3$ and $M_3 \leq N < \varepsilon(u_{j-1}) \cdot m$.
    Given \eqref{eqn:BoundedControlFunctionMoments} and that $||\L_{Z_1|Z_0=z} - \L_{X_1|X_0=z}||_{TV} = O(z^{-q})$, and taking $N := \lfloor \alpha t \cdot u_{j-1} \rfloor$ for $\alpha \in (1/t, 1)$, we can simplify the above bound to
    \begin{align*}
        &|| \L_{(Z_j, \dots, Z_k)|Z_{j-1} = u_{j-1}} - \L_{(X_j, \dots, X_k)|X_{j-1} = u_{j-1}} ||_{TV} \\
        &\leq \frac{s}{u_{j-1}^{q}} + \frac{(a \sigma^2 + b m^2) \cdot u_{j-1}}{(t \cdot u_{j-1} - \lfloor \alpha t \cdot u_{j-1} \rfloor)^2} + \K_{k-j}(\lfloor \alpha t \cdot u_{j-1} \rfloor + 1) \\
        &\leq \frac{s}{u_{j-1}^{q}} + \frac{(a \sigma^2 + b m^2)}{(1-\alpha)^2 t^2 \cdot u_{j-1}} + \K_{k-j}(\alpha t \cdot u_{j-1}) \;=: \K_{k-j+1}(u_{j-1}).
    \end{align*}
    Since we assumed that $\K_{k-j}$ was a decreasing function on $[M_3, \infty)$, we see that $\K_{k-j+1}$ is also a decreasing function on $[M_3, \infty)$. \\

    Having shown the recursive relationship \eqref{eqn:CBPKStepRecursiveRelationship}, and since $\alpha \in (1/t, 1)$ implies $\alpha t > 1$, by taking $f_i = \K_i$ for $1\leq i \leq k$ we see that the sequence of functions $\{ \K_i \}_{1\leq i\leq k}$ satisfies the requirements of Lemma \ref{RecursiveEquationLimit}.
    It hence follows that $\lim_{z\to\infty}\lim_{k\to\infty} \K_k(z) = 0$.
   Since for any $z_0$,
    \[
        0 \leq || \L_{\Z} - \L_{\X} ||_{TV}
        = \lim_{k\to\infty} || \L_{(Z_1,\dots,Z_k)|Z_0=z_0} - \L_{(X_1,\dots,X_k)|X_0=z_0}||_{TV}
        \leq \lim_{k\to\infty}\K_k(z_0),
    \]
    it further follows that $\lim_{z_0\to\infty} || \L_{\X} - \L_{\Z} ||_{TV} = 0$.
\end{proof}

\begin{proof}[Proof of Proposition \ref{ppn:CBPNoConsistentEstimation}.]
    For a given set $\Pi$ of supercritical CBPs satisfying \eqref{eqn:BoundedControlFunctionMoments} and with transition probabilities parameterised by $\btheta$,
    assume there exist processes $\Z, \X \in \Pi$ with $\btheta_Z \neq \btheta_X$ such that $||\L_{Z_1|Z_0=z_0} - \L_{X_1|X_0=z_0}||_{TV} = O(z_0^{-q})$ for some $q > 0$.
    Let $\{ \hat{\btheta}_k \}_{k\in\N_1}$ be a sequence of estimators for $\btheta$.

    Suppose that the sequence of estimators $\{ \hat{\btheta}_k \}_{k\in\N_1}$ is (weakly) consistent for $\btheta$ on the set of unbounded growth of the process.
    Then there exists a subsequence $\{k_j\}_{j\in\N_1}$ such that $\{\hat{\btheta}_{k_j}\}_{j\in\N_1}$ forms a strongly consistent sequence of estimators on the set of unbounded growth---that is, on that set, $\hat{\btheta} := \lim_{j\to\infty} \hat{\btheta}_{k_j}$ exists and equals $\btheta$ almost surely (see, for example, \cite[Theorem 2.3.2]{durrett19}). In a slight abuse of notation, we say that $\hat{\btheta}$ is itself strongly consistent.

    Since
    \[
        || \L_{\Z} - \L_{\X} ||_{TV} = \sup_{C \subseteq \N_0^{\infty}} | \P(\Z \in C) - \P(\X \in C) |,
    \]
    then
    \begin{equation}\label{eqn:ThetaHatTVDBound}
        || \L_{\Z} - \L_{\X} ||_{TV} \geq | \P(\hat{\btheta}(\Z) = \btheta_X,\; Z_n \to \infty) - \P(\hat{\btheta}(\X) = \btheta_X,\; X_n \to \infty) |.
    \end{equation}
    But, since $\hat{\btheta}$ is a strongly consistent estimator for $\btheta$ on the set of unbounded growth, and $\btheta_Z \neq \btheta_X$,
    \[
        \P(\hat{\btheta}(\Z) = \btheta_X,\; Z_n \to \infty) = 0,\qquad\textrm{while}\qquad \P(\hat{\btheta}(\X) = \btheta_X,\; X_n \to \infty) = \P(X_n \to \infty).
    \]
    Therefore, given that $\hat{\btheta}$ is strongly consistent, we can rewrite \eqref{eqn:ThetaHatTVDBound} as
    \[
        || \L_{\Z} - \L_{\X} ||_{TV} \geq \P(X_n \to \infty).
    \]
    However, we note the following two facts:
    \begin{enumerate}
        \item[(i)] Since $||\L_{Z_1|Z_0=z_0} - \L_{X_1|X_0=z_0}||_{TV} = O(z_0^{-q})$, then $\lim_{z_0\to\infty} || \L_{\Z} - \L_{\X} ||_{TV} = 0$ by Lemma \ref{lma:CBPTVDLimitApproachesZero}, and
        \item[(ii)] Given \eqref{eqn:BoundedControlFunctionMoments}, \cite[Theorem 3.2]{gonzalez18} tells us that $\P(Z_n \to \infty) \to 1$ as the initial population size $z_0$ approaches infinity.
    \end{enumerate}
    This creates a contradiction, since (i) and (ii) tell us that we can find $z_0 \in \N_1$ sufficiently large such that $|| \L_{\Z} - \L_{\X} ||_{TV} < \P(X_n \to \infty)$. So, in conclusion, no consistent estimator for $\btheta$ exists on the set of unbounded growth.
\end{proof}

\subsection{Proofs for Section \ref{sec:CBPEstimationKnownPhi}}

\begin{proof}[Proof of Lemma \ref{lma:KnownPhiOneStepTVDBound}.]
    Since our two processes $\Z$ and $\X$ belong to the set $\Pi^{(k)}$, they must have a common control function, $\phi(\cdot)$.
    If we consider the realisations of $\phi(\cdot)$ as forming interstitial states within the processes $\Z$ and $\X$, we can recognise the expanded processes $\{\phi(z_0),\, Z_1,\, \phi(Z_1),\, Z_2,\, \dots \}$ and $\{\phi(z_0),\, X_1,\, \phi(X_1),\, X_2,\, \dots \}$, consisting of alternating progenitor numbers and population sizes, as two time-inhomogeneous Markov chains.
    We can see from the definition of the total variation distance that
    \[
        || \L_{Z_1 | Z_0 = z_0} - \L_{X_1 | X_0 = z_0} ||_{TV}
        \leq || \L_{(\phi(Z_0), Z_1) | Z_0 = z_0} - \L_{(\phi(X_0), X_1) | X_0 = z_0} ||_{TV}.
    \]
    Additionally, since $\xi_Z$ and $\xi_X$ both have the same mean $m$ and variance $\sigma^2$, finite third moments, and lattice size one, we know from \cite[Theorem 9]{petrov64} that, for a given $u\in\N_1$, there exists a constant $c$ depending on $\xi_Z$ and $\xi_X$ such that
    \[
        || \L_{Z_1 | \phi(Z_0)=u} - \L_{X_1 | \phi(X_0)=u} ||_{TV}
        = || \L_{\sum_{i=1}^u \xi_{Z,i}} - \L_{\sum_{i=1}^u \xi_{X,i}} ||_{TV}
        \leq \frac{c}{\sqrt{u}}.
    \]
    Hence, for any $N\in\N_0$, it follows from Lemma \ref{lma:TVDDecreasingBound} that
    \begin{align*}
        &\quad || \L_{(\phi(Z_0), Z_1) | Z_0 = z_0} - \L_{(\phi(X_0), X_1) | X_0 = z_0} ||_{TV} \\
        &\leq || \L_{\phi(Z_0) | Z_0 = z_0} - \L_{\phi(X_0)| X_0 = z_0} ||_{TV} + \P(\phi_Z(Z_0) \leq N | Z_0 = z_0) + \frac{c}{\sqrt{N+1}} \\
        &= || \L_{\phi(z_0)} - \L_{\phi(z_0)} ||_{TV} + \P(\phi(z_0) \leq N) + \frac{c}{\sqrt{N+1}} \\
        &= \P(\phi(z_0) \leq N) + \frac{c}{\sqrt{N+1}}.
    \end{align*}
    Then, taking $N := \lfloor \alpha \cdot \varepsilon(u) \rfloor$ for $\alpha \in (0, 1)$, we can use Chebyshev's inequality to further bound
    \begin{align*}
        || \L_{(\phi(Z_0), Z_1) | Z_0 = z_0} - \L_{(\phi(X_0), X_1) | X_0 = z_0} ||_{TV}
        &\leq \frac{\nu^2(z_0)}{\big( \varepsilon(z_0) - \lfloor \alpha \cdot \varepsilon(z_0) \rfloor \big)^2} + \frac{c}{\sqrt{\lfloor \alpha \cdot \varepsilon(z_0) \rfloor + 1}} \\
        &\leq \frac{\nu^2(z_0)}{(1-\alpha)^2 \varepsilon^2(z_0)} + \frac{c}{\sqrt{\alpha \cdot \varepsilon(z_0)}}.
    \end{align*}
    Under assumption \eqref{eqn:BoundedControlFunctionMoments} there exists a constant $b$ such that $\nu^2(z) \leq bz$ for all $z \in \N_1$, while it follows from the assumption of supercriticality that there exists $M>0$ such that $\varepsilon(z) > m \cdot z$ for all $z > M$.
    Hence for $z > M$,
    \[
        || \L_{Z_1|Z_0 = z_0} - \L_{X_1|X_0 = z_0} ||_{TV}
        \leq \frac{b}{(1-\alpha)^2 m^2 \cdot z_0} + \frac{c}{\sqrt{\alpha m \cdot z_0}} = O(z_0^{-1/2}).
    \]
\end{proof}

\subsection{Proofs for Section \ref{sec:CBPEstimationUnknownControl}}

The proof of Lemma \ref{lma:LinearlyDivisibleOneStepTVDBound} relies on the control functions of our CBPs converging to a discretised normal distribution as the population size gets large. We use the following convention to describe this distribution:

\begin{definition}
    We say that a random variable $W$ has a discretised normal distribution with parameters $m$ and $\sigma^2$, written $W \sim \dn(m, \sigma^2)$, if for every $k\in\mathbbm{Z}$,
	\[
		\P(W = k)
		= \frac{1}{\sqrt{2\pi\sigma^2}} \int_{k - \frac{1}{2}}^{k + \frac{1}{2}} e^{-\frac{(u - m)^2}{2\sigma^2}} du.
	\]
\end{definition}

We can then leverage Chen's Stein's method result \cite[Theorem 7.4]{chen11} to find a total variation distance bound between a sum of i.i.d.\ random variables and a discretised normal.

\begin{lemma}\label{lma:DiscretisedNormalTVDBound}
    Let $X, X_1, \dots, X_n$, $n \in \N_1$, be i.i.d.\ random variables on $\N_0$ with $\E X := m$, $\V(X) := \sigma^2 > 0$, and finite third absolute central moment $\rho := \E|X - m|^3$.
    Define $S_n := \sum_{i=1}^n X_i$, and let $W_n$ be a random variable with a $\dn(nm, n\sigma^2)$ distribution. Then
    \begin{align*}
        ||\L_{S_n} - \L_{W_n} ||_{TV} &\leq \sqrt{\frac{2}{\pi}} \left( \frac{3\rho}{\sigma^2} + 2 \right) \Big( 1 + 4 (n-1)\left( 1- ||\L_{X} - \L_{(X + 1)} ||_{TV} \right) \Big)^{-\frac{1}{2}} \\
        &\quad + \left( 5 + 3\sqrt{\frac{\pi}{8}} \right)\frac{\rho}{\sqrt{n}\sigma^3} + \frac{1}{2\sqrt{2\pi n}\sigma}.
	\end{align*}
\end{lemma}

\begin{proof}
    We can use \cite[Theorem 7.4]{chen11}, simplified to the case of an i.i.d.\ sum, to see that
	\begin{equation*}
		||\L_{S_n} - \L_{W_n} ||_{TV}
		\leq \left( \frac{3\rho}{2\sigma^2} + 1 \right) || \L_{S_{n-1}} - \L_{(S_{n-1} + 1)} ||_{TV}
		+ \left( 5 + 3\sqrt{\frac{\pi}{8}} \right)\frac{\rho}{\sqrt{n}\sigma^3}
		+ \frac{1}{2\sqrt{2\pi n}\sigma}.
	\end{equation*}
    A bound on $|| \L_{S_{n-1}} - \L_{(S_{n-1} + 1)} ||_{TV}$ is then provided by \cite[Corollary 1.6]{mattner07}, as
	\begin{align*}
		||\L_{S_{n-1}} - \L_{(S_{n-1} + 1)} ||_{TV}
		&\leq \sqrt{\frac{2}{\pi}} \left( \frac{1}{4} + (n-1)\left( 1 - ||\L_{X} - \L_{(X + 1)} ||_{TV} \right) \right)^{-\frac{1}{2}} \\
		&= 2 \sqrt{\frac{2}{\pi}} \Big( 1 + 4 (n-1)\left( 1 - ||\L_{X} - \L_{(X + 1)} ||_{TV} \right) \Big)^{-\frac{1}{2}}.
	\end{align*}
    Put together, this yields the desired result.
\end{proof}

Lemma \ref{lma:DiscretisedNormalTVDBound} requires that the third \textit{absolute} central moment of our random variables be bounded. Given these random variables take values on $\N_0$, we can show that it will be finite if the third central moment of the random variables is.

\begin{lemma}\label{lma:ThirdAbsoluteCentralMomentBound}
    Let $X$ be a random variable taking values on $\N_0$ with mean $m$, variance $\sigma^2$, and $\gamma := \E(X - m)^3$ finite. Then
    \[
        \E|X - m|^3 \leq 8(\gamma + 3m\sigma^2 + m^3).
    \]
\end{lemma}

\begin{proof}
    By Minkowski's inequality, $(\E|X - m|^3)^{1/3} \leq (E|X|^3)^{1/3} + (E|-m|^3)^{1/3} = (E|X|^3)^{1/3} + |m|$.
    Using Jensen's inequality and subsequently H{\"o}lder's inequality, we see that $|m| \leq \E|X| \leq (E|X|^3)^{1/3}$. It follows that $(\E|X - m|^3)^{1/3} \leq 2 (E|X|^3)^{1/3}$, from which we can conclude $\E|X - m|^3 \leq 8 \cdot \E|X|^3$.

    However, since $X$ is supported on $\N_0$, $\E|X|^3 = \E X^3$, and we can show by expanding $\E(X-m)^3$ that $\E X^3 = \gamma + 3m\sigma^2 + m^3$. The desired result follows.
\end{proof}

We can formulate an upper bound on the total variation distance between two discretised normals with different parameters as follows.

\begin{lemma}\label{lma:TVDBetweenDiscretisedNormals}
    Suppose that $W \sim \dn(m, \sigma^2)$ and $\tilde{W} \sim \dn(\tilde{m}, \tilde{\sigma}^2)$, for $m,\tilde{m} \in \R$ and $\sigma^2, \tilde{\sigma}^2 > 0$. Then
    \[
        || \L_W - \L_{\tilde{W}} ||_{TV}
        \leq \frac{3| \sigma^2 - \tilde{\sigma}^2 |}{2(\sigma^2 \vee \tilde{\sigma}^2)} + \frac{| m - \tilde{m} |}{2(\sigma \vee \tilde{\sigma})}.
    \]
\end{lemma}

\begin{proof}
    We can find
    \begin{align*}
		|| \L_{W} - \L_{\tilde{W}} ||_{TV}
		&= \frac{1}{2} \sum_{n\in\mathbbm{Z}} \big| \P(W = n) - \P(\tilde{W} = n) \big| \\
		&= \frac{1}{2} \sum_{n\in\mathbbm{Z}} \left| \int_{n - \frac{1}{2}}^{n + \frac{1}{2}}
			\frac{\exp\Big\{ -\frac{(u - m)^2}{2\sigma^2} \Big\}}{\sqrt{2\pi \sigma^2}} du
			- \int_{n - \frac{1}{2}}^{n + \frac{1}{2}} \frac{\exp\Big\{-\frac{(u - \tilde{m})^2}
			{2\tilde{\sigma}^2} \Big\}}{\sqrt{2\pi \tilde{\sigma}^2}} du \right| \\
		&\leq \frac{1}{2} \sum_{n\in\mathbbm{Z}} \int_{n - \frac{1}{2}}^{n + \frac{1}{2}} 
            \left| \frac{\exp\Big\{ -\frac{(u - m)^2}{2\sigma^2} \Big\}}{\sqrt{2\pi \sigma^2}} - \frac{\exp\Big\{-\frac{(u - \tilde{m})^2}{2\tilde{\sigma}^2} \Big\}}{\sqrt{2\pi \tilde{\sigma}^2}} \right| du \\
		&= || N( m, \sigma^2 ) - N( \tilde{m}, \tilde{\sigma}^2 ) ||_{TV}.
    \end{align*}
    That is, the TVD between two discretised normal distributions is bounded above by the TVD between two normal distributions with the same parameters.
    The result then follows from \cite[Theorem 1.3]{devroye22}, which states that
    \[
        || N( m, \sigma^2 ) - N( \tilde{m}, \tilde{\sigma}^2 ) ||_{TV}
        \leq \frac{3| \sigma^2 - \tilde{\sigma}^2 |}{2(\sigma^2 \vee \tilde{\sigma}^2)} + \frac{| m - \tilde{m} |}{2(\sigma \vee \tilde{\sigma})}.
    \]
\end{proof}

Having completed this initial set-up, we are ready to prove Lemma \ref{lma:LinearlyDivisibleOneStepTVDBound}.

\begin{proof}[Proof of Lemma \ref{lma:LinearlyDivisibleOneStepTVDBound}.]
    Suppose that $\Z,\X\in \Pi^{(u)}$, such that their respective control functions $\phi_Z(\cdot)$ and $\phi_X(\cdot)$ are both linearly-divisible and satisfy \eqref{eqn:BoundedControlFunctionThirdMoment},
    their offspring distributions $\xi_Z$ and $\xi_X$ both have finite third moments and lattice size one, and that there exists $r < 1$ such that $|\E(Z_1|Z_0=z) - \E(X_1|X_0=z)| = O(z^{r/2})$ and $|\V(Z_1|Z_0=z) - \V(X_1|X_0=z)| = O(z^r)$.

    Since $\phi_Z(\cdot)$ is linearly-divisible, for a given $z\in\N_0$, we can decompose $\phi_Z(z) \stackrel{d}{=} \sum_{i=1}^{l_Z(z)} \tilde{\phi}_{Z,i}(z)$ for $\{ \tilde{\phi}_{Z,i}(z) \}_{i\in\N_1}$ all i.i.d., so that $\E\tilde{\phi}_Z(z) = \varepsilon_Z(z) / l_Z(z)$ and $\V(\tilde{\phi}_Z(z)) = \nu_Z^2(z) / l_Z(z)$.
    Define\[
        \tilde{\xi}_Z(z) := \sum_{i=1}^{\tilde{\phi}_Z(z)} \xi_{Z,i},
    \]
    and denote $\tilde{m}_Z(z) := \E\tilde{\xi}_Z(z)$, $\tilde{\sigma}_Z^2(z) := \V\big(\tilde{\xi}_Z(z)\big)$, and $\tilde{\gamma}_Z(z) := \E(\tilde{\xi}_Z(z) - \E\tilde{\xi}_Z(z))^3$. We find that
    \[
        \tilde{m}_Z(z) = \frac{\varepsilon_Z(z) \cdot m_Z}{z} \cdot \frac{z}{l_Z(z)}
        \quad\text{and}\quad
        \tilde{\sigma}_Z^2(z) =  \frac{\varepsilon_Z(z) \cdot \sigma_Z^2 + \nu_Z^2(z) \cdot m_Z^2}{z} \cdot \frac{z}{l_Z(z)},
    \]
    while
    \[
        \tilde{\gamma}_Z(z)
        = \frac{\varepsilon_Z(z) \cdot \gamma_Z + \iota_Z(z) \cdot m_Z^3 + \nu_Z^2(z) \cdot m_Z \cdot \sigma_Z^2}{z} \cdot \frac{z}{l_Z(z)}.
    \]

    Since $l_Z(z) = \Theta(z)$, there exist $\lambda_Z > 1$ and $M_1 \in \N_1$ such that for all $z \geq M_1$, $z / l_Z(z) \in [1/\lambda_Z, \lambda_Z] \cap \N_1$. Given \eqref{eqn:BoundedControlFunctionThirdMoment}, for $z \geq M_1$, we obtain
    \begin{equation}\label{FirstThreeMomentsUpperBoundsLittleO}
        \tilde{m}_Z(z) \leq \lambda_Z a_Z m_Z, \quad
        \tilde{\sigma}_Z^2(z) \leq \lambda_Z( a_Z \sigma_Z^2 + b_Z m_Z^2 ), \quad\text{and}\quad
        \tilde{\gamma}_Z(z) \leq \lambda_Z( a_Z \gamma_Z + c_Z m_Z^3 + b_Z m_Z \sigma_Z^2),
    \end{equation}
    where the constants $a_Z$, $b_Z$, and $c_Z$ are the constants in \eqref{eqn:BoundedControlFunctionThirdMoment} for the process $\Z$.
    In addition:
    \begin{enumerate}
        \item[(i)] Given the upper bounds for $\tilde{m}_Z(z)$, $\tilde{\sigma}_Z^2(z)$, and $\tilde{\gamma}_Z(z)$ in \eqref{FirstThreeMomentsUpperBoundsLittleO},
        it follows from Lemma \ref{lma:ThirdAbsoluteCentralMomentBound} that there exists a finite constant $R_Z > 0$ such that $\tilde{\rho}_Z(z) := \E|\tilde{\xi}_Z(z) - \E\tilde{\xi}_Z(z)|^3 \leq R_Z$ for $z \geq M_1$.

        \item[(ii)] We have that $\tilde{\sigma}_Z^2(z) \geq \frac{\varepsilon_Z(z) \cdot \sigma_Z^2}{z} \cdot \frac{z}{l_Z(z)}$ for all $z \geq 0$, and that $z / l_Z(z) \geq 1/\lambda_Z$ for $z \geq M_1$.
        In addition, since $\Z$ is supercritical, for $t_Z$ such that $1 < t_Z < \liminf_{z\to\infty} \tau_Z(z)$ there exists $M_2 \in \N_1$ such that $\varepsilon_Z(z) \cdot m_Z > t_Z \cdot z$ for all $z \geq M_2$.
        Then, for $z \geq M_3 := M_1 \vee M_2$, it follows that $\tilde{\sigma}_Z^2(z) > \frac{t_Z \sigma_Z^2}{\lambda_Z m_Z}$.

        \item[(iii)]Since $\xi_Z$ has lattice size one, there exists $x \in \N_1$ such that $\P(\xi_Z = x) > 0$ and $\P(\xi_Z = x - 1) > 0$. Hence we see that
        \begin{align*}
            ||\L_{\tilde{\xi}_Z(z)} - \L_{(\tilde{\xi}_Z(z) + 1)} ||_{TV}
            &= 1 - \sum_{j=0}^{\infty}\P\bigg( \sum_{i=1}^{\tilde{\phi}_Z(z)} \xi_{Z,i} = j \bigg) \wedge \P\bigg( \sum_{i=1}^{\tilde{\phi}_Z(z)} \xi_{Z,i} = j - 1 \bigg) \\
            &\leq 1 - \P\big(\tilde{\phi}_Z(z) \geq  1\big) \big( \P(\xi_{Z} = x) \wedge \P(\xi_{Z} = x - 1) \big).
        \end{align*}
        For all $z \geq M_3$, our assumptions guarantee that $\E\tilde{\phi}_Z(z) > t_Z / (\lambda_Z m_Z)$ and $\V(\tilde{\phi}_Z(z)) \leq \lambda_Z b_Z$. Given these bounds, $\P\big(\tilde{\phi}_Z(z) \geq 1 \big)$ is minimised if $\tilde{\phi}_Z(z)$ has a two point distribution (see \cite[Lemma 6.7]{braunsteins25}), with mass at zero and a point $y \in \R_{>0}$ satisfying
        \[
            y \cdot \P(\tilde{\phi}_Z(z) = y) = \frac{t_Z}{\lambda_Z m_Z}
            \quad\text{and}\quad
            y^2 \cdot \P(\tilde{\phi}_Z(z) = y) - \big( y \cdot \P(\tilde{\phi}_Z(z) = y) \big)^2 = \lambda_Z b_Z,
        \]
        which we can solve to see that $\P\big(\tilde{\phi}_Z(z) \geq 1 \big) > t_Z^2 / (t_Z^2 + \lambda_Z^3 m_Z^2 b_Z) > 0$ for all $z \geq M_3$.
        Hence there exists a positive constant $\eta_Z$ such that, for all $z \geq M_3$, $||\L_{\tilde{\xi}_Z(z)} - \L_{(\tilde{\xi}_Z(z) + 1)} ||_{TV} \leq \eta_Z < 1$.
        Further, for $\eta'_Z := \eta_Z \vee 3/4$, we obtain $1 + 4 (l_Z(z)-1)(1 - \eta_Z) \geq 4(1-\eta'_Z)l_Z(z)$.
    \end{enumerate}

    Inputting these results into Lemma \ref{lma:DiscretisedNormalTVDBound}, with $W_Z(z) \sim \dn\big(l_Z(z) \cdot \tilde{m}_Z(z),\; l_Z(z) \cdot \tilde{\sigma}_Z^2(z) \big)$ and $z \geq M_3$, we obtain
    \begin{align*}
        &||\L_{Z_1|Z_0=z} - \L_{W_Z(z)} ||_{TV} \\
        &\leq \sqrt{\frac{2}{\pi}} \left( \frac{3\tilde{\rho}_Z(z)}{\tilde{\sigma}_Z^2(z)} + 2 \right) \Big( 1 + 4 (l_Z(z) - 1)\left( 1 - ||\L_{\tilde{\xi}_Z(z)} - \L_{(\tilde{\xi}_Z(z) + 1)} ||_{TV} \right) \Big)^{-\frac{1}{2}} \\
        &\quad + \left( 5 + 3\sqrt{\frac{\pi}{8}} \right)\frac{\tilde{\rho}_Z(z)}{\sqrt{l_Z(z)}\tilde{\sigma}_Z^3(z)} + \frac{1}{2\sqrt{2\pi l_Z(z)}\tilde{\sigma}_Z(z)} \\
        &\leq \left( \frac{3m_Z\lambda_Z R_Z}{t_Z \sigma_Z^2} + 2 \right) \frac{\sqrt{\lambda_Z}}{\sqrt{2\pi( 1 - \eta'_Z) \cdot z}}
        + \left( 5 + 3\sqrt{\frac{\pi}{8}} \right)\frac{\lambda_Z^2 \sqrt{m_Z^3}R_Z}{\sigma_Z^3 \sqrt{t_Z^3 \cdot z}} + \frac{\lambda_Z\sqrt{m_Z}}{2\sigma_Z\sqrt{2\pi t_Z \cdot z}},
	\end{align*}
    and hence see that
    $
        ||\L_{Z_1|Z_0=z} - \L_{W_Z(z)} ||_{TV} = O(z^{-1/2}).
    $
 Repeating the same argument for $\X$, we obtain
    $
        ||\L_{X_1|X_0=z} - \L_{W_X(z)} ||_{TV} = O(z^{-1/2})
    $ with $W_X(z) \sim \dn\big(l_X(z) \cdot \tilde{m}_X(z),\; l_X(z) \cdot \tilde{\sigma}_X^2(z) \big)$.
    
    It follows from the triangle inequality that
    \begin{align*}
        &||\L_{Z_1|Z_0=z} - \L_{X_1|X_0=z}||_{TV} \\
        &\leq ||\L_{Z_1|Z_0=z} - \L_{W_Z(z)} ||_{TV} + || \L_{W_Z(z)} - \L_{W_X(z)} ||_{TV} + ||\L_{X_1|X_0=z} - \L_{W_X(z)} ||_{TV},
    \end{align*}
    so the result will follow if we can show that there exists $s>0$ such that $|| \L_{W_Z(z)} - \L_{W_X(z)} ||_{TV} = O(z^{-s})$.
Lemma \ref{lma:TVDBetweenDiscretisedNormals} allows us to bound
    \begin{align*}
        &|| \L_{W_Z(z)} - \L_{W_X(z)} ||_{TV} \\
        &\leq \frac{3| l_Z(z) \cdot \tilde{\sigma}_Z^2(z) - l_X(z) \cdot \tilde{\sigma}_X^2(z) |}{2(l_Z(z) \cdot \tilde{\sigma}_Z^2(z))}
        + \frac{| l_Z(z) \cdot \tilde{m}_Z(z) - l_X(z) \cdot \tilde{m}_X(z) |}{2\sqrt{l_Z(z) \cdot \tilde{\sigma}_Z^2(z)}} \\
        &\leq \frac{ 3\lambda_Z^2 m_Z }{ 2 t_Z \sigma_Z^2 } \cdot \frac{|\V(Z_1|Z_0=z) - \V(X_1|X_0=z)|}{z}
        + \frac{ \lambda_Z \sqrt{m_Z} }{ 2 \sigma_Z \sqrt{t_Z} } \cdot \frac{|\E(Z_1|Z_0=z) - \E(X_1|X_0=z)|}{\sqrt{z}}
    \end{align*}
    for $z \geq M_3$, and since there exists $r < 1$ such that $|\E(Z_1|Z_0=z) - \E(X_1|X_0=z)| = O(z^{r/2})$ and $|\V(Z_1|Z_0=z) - \V(X_1|X_0=z)| = O(z^r)$, we see that $|| \L_{W_Z(z)} - \L_{W_X(z)} ||_{TV} = O(z^{(r-1)/2})$, so that $|| \L_{W_Z(z)} - \L_{W_X(z)} ||_{TV} = O(z^{-s})$ for $s=(1 - r)/2$. Therefore, $||\L_{Z_1|Z_0=z} - \L_{X_1|X_0=z}||_{TV}= O(z^{-q}) $ for $q:=(1-r)/2.$
\end{proof}

\subsection{Proofs for Section \ref{sec:CBPEstimationObservedProgenitors}}

To prove Lemma \ref{lma:KnownProgenitorsOneStepTVDBound}, we first introduce the following lemma:

\begin{lemma}\label{lma:ControlFunctionsOneStepTVDBound}
    Let $\phi_Z(z)$ and $\phi_X(z)$ be two control functions satisfying \eqref{eqn:BoundedControlFunctionThirdMoment}, linearly-divisible such that \eqref{eqn:UniformlyLatticeSizeOne} is satisfied, and each with $\liminf_{z\to\infty} \nu^2(z)/z > 0$.
    Then, if there exists $r < 1$ such that
    \[
        |\varepsilon_Z(z) - \varepsilon_X(z)| = O(z^{r/2})
        \quad \text{and} \quad
        |\nu^2_Z(z) - \nu^2_X(z)| = O(z^r),
    \]
    then there exists $q > 0$ such that $|| \L_{\phi_Z(z)} - \L_{\phi_X(z)} ||_{TV} = O(z^{-q})$.
\end{lemma}

\begin{proof}
    Let $\phi_Z(z)$ and $\phi_X(z)$ be as stated in the lemma.
    The triangle inequality allows us to form the bound
    \begin{align*}
        &|| \L_{\phi_Z(z)} - \L_{\phi_X(z)} ||_{TV} \\
        &\leq ||\L_{\phi_Z(z)} - \L_{W_{\phi_Z}(z)} ||_{TV} + || \L_{W_{\phi_Z}(z)} - \L_{W_{\phi_X}(z)} ||_{TV} + ||\L_{\phi_X(z)} - \L_{W_{\phi_X}(z)} ||_{TV},
    \end{align*}
    where $W_{\phi_Z}(z) \sim \dn\big( \varepsilon_Z(z),\; \nu^2_Z(z) \big)$ and $W_{\phi_X}(z) \sim \dn\big( \varepsilon_X(z),\; \nu^2_X(z) \big)$. \\

    Since $\phi_Z(\cdot)$ is linearly-divisible, for a given $z \in \N_0$, we can decompose $\phi_Z(z) \stackrel{d}{=} \sum_{i=1}^{l_Z(z)} \pt_{Z,i}(z)$, where $\{\pt_{Z,i}(z)\}_{i\in\N_1}$ are all i.i.d.
    We define $\tilde{\varepsilon}_Z(z) := \E\tilde{\phi}_Z(z)$, $\tilde{\nu}^2_Z(z) := \V(\tilde{\phi}_Z(z))$, and $\tilde{\iota}_Z(z) := \E(\tilde{\phi}_Z(z) - \E\tilde{\phi}_Z(z))^3$, so that
    \[
        \tilde{\varepsilon}_Z(z) = \frac{\varepsilon_Z(z)}{z} \cdot \frac{z}{l_Z(z)},
        \quad
        \tilde{\nu}^2_Z(z) = \frac{\nu^2_Z(z)}{z} \cdot \frac{z}{l_Z(z)},
        \quad\text{and}\quad
        \tilde{\iota}_Z(z) = \frac{\iota_Z(z)}{z} \cdot \frac{z}{l_Z(z)}.
    \]
    Since $l_Z(z) = \Theta(z)$, there exist $\lambda_Z > 1$ and $M_1 \in \N_1$ such that for all $z \geq M_1$, $z/l_Z(z) \in [1/\lambda_Z,\, \lambda_Z] \cap \N_1$.
    Since we assume that $\phi_Z(\cdot)$ satisfies \eqref{eqn:BoundedControlFunctionThirdMoment}, we can see that
    \begin{equation}\label{eqn:ControlFirstThreeMomentsUpperBounds}
        \tilde{\varepsilon}_Z(z) \leq a_Z \lambda_Z, \quad
        \tilde{\nu}_Z^2(z) \leq b_Z \lambda_Z, \quad\text{and}\quad
        \tilde{\iota}_Z(z) \leq c_Z \lambda_Z
    \end{equation}
    for $z \geq M_1$.
    We can then note that
    \begin{enumerate}
        \item[(i)] Given the upper bounds in \eqref{eqn:ControlFirstThreeMomentsUpperBounds},
        it follows from Lemma \ref{lma:ThirdAbsoluteCentralMomentBound} that there exists a finite constant $R_X > 0$ such that $\tilde{\rho}_Z(z) := \E|\tilde{\phi}_Z(z) - \E\tilde{\phi}_Z(z)|^3 \leq R_Z$ for $z \geq M_1$.
    
        \item[(ii)] By assumption, $\liminf_{z\to\infty} \nu^2_Z(z)/z > 0$. Hence there exists $t_Z > 0$ and $M_2 \in \N_1$ such that for all $z \geq M_2$, $\nu^2_Z(z) > t_Z z$. Then, since $l_Z(z) < \lambda_Z z$ for $z \geq M_1$, for $z \geq M_3 := M_1 \vee M_2$, $\tilde{\nu}_Z^2(z) > t_Z / \lambda_Z$.
    
        \item[(iii)] Because the $\pt_{Z,i}(z)$'s satisfy \eqref{eqn:UniformlyLatticeSizeOne}, we have $|| \L_{\pt_Z(z)} - \L_{(\pt_Z(z)+1)} ||_{TV} \leq 1 - \eta_Z$.
    \end{enumerate}
    Then, by Lemma \ref{lma:DiscretisedNormalTVDBound},
    \begin{align*}
        &||\L_{\phi_Z(z)} - \L_{W_{\phi_Z}(z)} ||_{TV} \\
        &\leq \sqrt{\frac{2}{\pi}} \left( \frac{3\lambda_Z R_Z}{ t_Z} + 2 \right) \frac{1}{\sqrt{1 + 4\eta_Z (z/\lambda_Z - 1)}}
        + \left( 5 + 3\sqrt{\frac{\pi}{8}} \right)\frac{\sqrt{\lambda_Z^3} R_Z}{\sqrt{t_Z^3 \cdot z}} + \frac{\sqrt{\lambda_Z}}{2 \sqrt{2\pi t_Z \cdot z}}
    \end{align*}
    for $z \geq M_3$, so that we see $||\L_{\phi_Z(z)} - \L_{W_{\phi_Z}(z)} ||_{TV} = O(z^{-1/2})$.
    Analogously, we can find $||\L_{\phi_X(z)} - \L_{W_{\phi_X}(z)} ||_{TV} = O(z^{-1/2})$.

    It remains to bound $|| \L_{W_{\phi_Z}(z)} - \L_{W_{\phi_X}(z)} ||_{TV}$. By Lemma \ref{lma:TVDBetweenDiscretisedNormals}, and since $\nu^2_Z(z) > t_Z z$ for all $z \geq M_2$, we see that
    \begin{align*}
        || \L_{W_{\phi_Z}(z)} - \L_{W_{\phi_X}(z)} ||_{TV}
        &\leq \frac{3 | \nu^2_Z(z) - \nu^2_X(z) |}{2 t_Z \cdot z}
        + \frac{| \varepsilon_Z(z) - \varepsilon_X(z) |}{2\sqrt{t_Z \cdot z}}
    \end{align*}
    for $z \geq M_2$.
    Since $|\varepsilon_Z(z) - \varepsilon_X(z)| = O(z^{r/2})$ and $ |\nu^2_Z(z) - \nu^2_X(z)| = O(z^r)$, we see that $|| \L_{W_{\phi_Z}(z)} - \L_{W_{\phi_X}(z)} ||_{TV} = O(z^{(r-1)/2})$.
    Hence $|| \L_{\phi_Z(z)} - \L_{\phi_X(z)} ||_{TV} = O(z^{-q})$, for $q := (1-r)/2$.
\end{proof}

\begin{proof}[Proof of Lemma \ref{lma:KnownProgenitorsOneStepTVDBound}.]
    Let $\Z$ and $\X$ be two supercritical CBPs with control functions satisfying \eqref{eqn:BoundedControlFunctionThirdMoment} and with $\liminf_{z\to\infty} \nu^2(z)/z > 0$, that are linearly-divisible into random variables satisfying \eqref{eqn:UniformlyLatticeSizeOne}, and with offspring distributions having finite third moments and lattice size one.
    Assume that $m_Z = m_X$ and $\sigma^2_Z = \sigma^2_X$, and there exists $r < 1$ such that $|\varepsilon_Z(z) - \varepsilon_X(z)| = O(z^{r/2})$ and $|\nu^2_Z(z) - \nu^2_X(z)| = O(z^r)$. \\

    Since $m_Z = m_X$, $\sigma^2_Z = \sigma^2_X$, and $\xi_Z$ and $\xi_X$ both have finite third moments and lattice size one, it follows from \cite[Theorem 9]{petrov64} that, for a given $u\in\N_1$, there exists a constant $c$ depending on $\xi_Z$ and $\xi_X$ such that
    \[
        || \L_{Z_1 | \phi_Z(Z_0)=u} - \L_{X_1 | \phi_X(X_0)=u} ||_{TV}
        = || \L_{\sum_{i=1}^u \xi_{Z,i}} - \L_{\sum_{i=1}^u \xi_{X,i}} ||_{TV}
        \leq \frac{c}{\sqrt{u}}.
    \]
    Then, since $\{ \phi_Z(Z_0),\, Z_1,\, \phi_Z(Z_1),\, Z_2,\, \dots \}$ and $\{ \phi_X(X_0),\, X_1,\, \phi_X(X_1),\, X_2,\, \dots \}$ both form time-inhomogeneous Markov chains,
    we will have from Lemma \ref{lma:TVDDecreasingBound} that, for any $N\in\N_0$,
    \begin{align*}
        &|| \L_{(\phi_Z(Z_0),\, Z_1)|Z_0=z_0} - \L_{(\phi_X(X_0),\, X_1)|X_0=z_0}||_{TV} \\
        &\leq || \L_{\phi_Z(z_0)} - \L_{\phi_X(z_0)} ||_{TV} + \P(\phi_Z(z_0) \leq N) + \frac{c}{\sqrt{N+1}}.
    \end{align*}
    Then, taking $N := \lfloor \alpha \cdot \varepsilon(u) \rfloor$ for $\alpha \in (0, 1)$, we can use Chebyshev's inequality to further bound
    \begin{align*}
        &|| \L_{(\phi_Z(Z_0),\, Z_1)|Z_0=z_0} - \L_{(\phi_X(X_0),\, X_1)|X_0=z_0}||_{TV} \\
        &\leq || \L_{\phi_Z(z_0)} - \L_{\phi_X(z_0)} ||_{TV} + \frac{\nu^2_Z(z_0)}{\big( \varepsilon_Z(z_0) - \lfloor \alpha \cdot \varepsilon_Z(z_0) \rfloor \big)^2} + \frac{c}{\sqrt{\lfloor \alpha \cdot \varepsilon_Z(z) \rfloor + 1}} \\
        &\leq || \L_{\phi_Z(z_0)} - \L_{\phi_X(z_0)} ||_{TV} + \frac{\nu^2_Z(z_0)}{(1-\alpha)^2 \varepsilon_Z^2(z_0)} + \frac{c}{\sqrt{\alpha \cdot \varepsilon_Z(z_0)}}.
    \end{align*}
    Under assumption \eqref{eqn:BoundedControlFunctionThirdMoment} there exists a constant $b$ such that $\nu^2_Z(z) \leq bz$ for all $z \in \N_1$, while it follows from the assumption of supercriticality that there exists $M>0$ such that $\varepsilon_Z(z) > m_Z \cdot z$ for all $z > M$.
    Hence for $z > M$,
    \begin{align*}
        &|| \L_{(\phi_Z(Z_0),\, Z_1)|Z_0=z_0} - \L_{(\phi_X(X_0),\, X_1)|X_0=z_0}||_{TV} \\
        &\leq || \L_{\phi_Z(z_0)} - \L_{\phi_X(z_0)} ||_{TV} + \frac{b}{(1-\alpha)^2 m_Z^2 \cdot z_0} + \frac{c}{\sqrt{\alpha m_Z \cdot z_0}}.
    \end{align*}
    From Lemma \ref{lma:ControlFunctionsOneStepTVDBound}, we know that there exists $\tilde{q} > 0$ such that $|| \L_{\phi_Z(z_0)} - \L_{\phi_X(z_0)} ||_{TV}  = O(z_0^{-\tilde{q}})$,
    so that, for $q := \tilde{q} \wedge 1/2$, $|| \L_{(\phi_Z(Z_0),\, Z_1)|Z_0=z_0} - \L_{(\phi_X(X_0),\, X_1)|X_0=z_0}||_{TV} = O(z_0^{-q})$.
\end{proof}

\begin{proof}[Proof of Theorem \ref{thm:ObservedProgenitorsConsistentEstimators}]
    Given there exist two CBPs $\Z, \X \in \Pi^{(p)}$ that satisfy \eqref{eqn:KnownProgenitorsOneStepTVDBoundConditions}, by Lemma \ref{lma:KnownProgenitorsOneStepTVDBound} there exists $q > 0$ such that $|| \L_{(\phi_Z(Z_0),\, Z_1)|Z_0=z_0} - \L_{(\phi_X(X_0),\, X_1)|X_0=z_0}||_{TV} = O(z_0^{-q})$ (that is, equation \eqref{eqn:KnownProgenitorsOneStepTVDBound} applies). Then, to prove Theorem \ref{thm:ObservedProgenitorsConsistentEstimators}, we require equivalents to Lemma \ref{lma:CBPTVDLimitApproachesZero} and Proposition \ref{ppn:CBPNoConsistentEstimation} that apply under our extended observation scheme.
    The equivalent results are direct: \eqref{eqn:KnownProgenitorsOneStepTVDBound} implies that
    \begin{equation}\label{eqn:KnownProgentiorsTVDApproachsZero}
        \lim_{z_0\to\infty} || \L_{\{ Z_0,\,\phi(Z_0),\,Z_1,\,\phi(Z_1),\dots | Z_0 = z_0 \}} - \L_{\{ X_0,\,\phi(X_0),\,X_1,\,\phi(X_1),\dots |X_0 = z_0 \}} ||_{TV} = 0,
    \end{equation}
    which can be seen by altering the proof of Lemma \ref{lma:CBPTVDLimitApproachesZero} by replacing $Z_j$, for $j \in \N_1$, by $( \phi_Z(Z_{j-1}),\, Z_j )$ and $X_j$ by $( \phi_X(X_{j-1}),\, X_k )$ in each total variation distance.
    Similarly, we can alter Proposition \ref{ppn:CBPNoConsistentEstimation} to use \eqref{eqn:KnownProgentiorsTVDApproachsZero} rather than Lemma \ref{lma:CBPTVDLimitApproachesZero} in (i), from which the result follows.
\end{proof}

\section{Proofs of consistency results}

We detail here the proofs of Theorems \ref{thm:KnownPhiConsistentEstimators}, \ref{thm:LinearMeanVarPhiConsistentEstimators}, and \ref{thm:ObservedProgenitorsConsistentEstimators}.
These three proofs rely on a number of initial results which we state below. The first of these is the following classical convergence result for martingale difference sequences:

\begin{theorem}\label{thm:SLLNforMDS}
    Let $\{U_n\}_{n\in\N_0}$ be a martingale difference sequence adapted to a filtration $\{\F_n\}_{n\in\N_0}$ (i.e.\ $\E(U_n | \F_{n-1}) = 0$ and $\E|U_n| < \infty$), and let $\{ J_n \}_{n\in\N_0}$ be a non-decreasing sequence of positive random variables such that each $J_n$ is $\F_{n-1}$-measurable. If
    \[
        \lim_{n\to\infty} J_n = \infty \;\;\text{a.s.}
        \quad\text{and}\quad
        \sum_{n=1}^{\infty} J_n^{-2} \E(U_n^2 | \F_{n-1}) < \infty \;\;\text{a.s.},
    \]
    then $J_n^{-1} \sum_{k=1}^n U_k \stackrel{a.s.}{\longrightarrow} 0$ as $n \to \infty$.
\end{theorem}

A more general version of this result, alongside its proof, can be found in \cite[Theorem 2.18]{hall80}. We will often use Theorem \ref{thm:SLLNforMDS} in the special case where $J_n = n$ a.s., in which case the theorem simplifies to the following:
\begin{corollary}\label{crly:SLLNforMDS}
    Let $\{U_n\}_{n\in\N_0}$ be a martingale difference sequence adapted to a filtration $\{\F_n\}_{n\in\N_0}$, such that $\sum_{n=1}^\infty \frac{1}{n^2} \E(U^2_n | \F_{n-1}) < \infty$ a.s.. Then $\frac{1}{n} \sum_{k=1}^n U_k \stackrel{a.s.}{\longrightarrow} 0$ as $n \to \infty$.
\end{corollary}
This special case is proven as a theorem in its own right in \cite[Section VII.9, Theorem 3]{feller71}. In addition to these classical results, we will also make use of Lemma \ref{lma:NormedSumBoundedInProbability}, which we prove with the assistance of Lemma \ref{lma:DeterministicNormedSumBound} below.

\begin{lemma}\label{lma:DeterministicNormedSumBound}
    Let $\{z_k\}_{k\in\N_0}$ be a sequence of non-negative integers satisfying $z_k \stackrel{k \to \infty}{\longrightarrow} \infty$. Given $M > 0$, define $K_M := \min_{k \in \N_0} \{ k : z_k > M\}$. If $z_k \geq s \cdot z_{k-1}$ for some $s > 1$ and all $k \geq K_M$, then for all $n > K_M$,
    \[
        \sum_{k=1}^n \frac{z_{k-1}}{z_{n-1}} < \frac{K_M}{s^{n - K_M - 1}} + \frac{1}{1 - s^{-1}}.
    \]
\end{lemma}

\begin{proof}
    Given a sequence satisfying the conditions of the lemma, let $n > K_M$. Then
    \begin{align*}
        \sum_{k=1}^n \frac{z_{k-1}}{z_{n-1}}
        &= \sum_{k=1}^{K_M} \frac{z_{k-1}}{z_{K_M}} \cdot \frac{z_{K_M}}{z_{n-1}} + \sum_{k=K_M + 1}^n \frac{z_{k-1}}{z_{n-1}} \\
        &\leq \sum_{k=1}^{K_M} \frac{M}{z_{K_M}} \cdot \frac{z_{K_M}}{z_{n-1}} + \sum_{k=K_M + 1}^n \frac{z_{k-1}}{z_{n-1}} \\
        &< K_M \cdot \frac{z_{K_M}}{z_{n-1}} + \sum_{k=K_M + 1}^n \frac{z_{k-1}}{z_{n-1}},
    \end{align*}
    where, since $\frac{z_j}{z_{n-1}} \leq \frac{1}{s^{n-j-1}}$ for $j \geq K_M$,
    \[
        K_M \cdot \frac{z_{K_M}}{z_{n-1}} \leq \frac{K_M}{s^{n - K_M - 1}}.
    \]
    In addition,
    \[
        \sum_{k=K_M + 1}^n \frac{z_{k-1}}{z_{n-1}}
        \leq \sum_{k=K_M + 1}^n \frac{1}{s^{n-k}}
        = \sum_{k=0}^{n-K_M-1} \frac{1}{s^k}
        < \frac{1}{1 - s^{-1}}.
    \]
\end{proof}

\begin{lemma}\label{lma:NormedSumBoundedInProbability}
    Let $\Z$ be a supercritical CBP satisfying \eqref{eqn:BoundedControlFunctionMoments}, and with $\sigma^2$ finite.
    Let $s$ be such that $1 < s < \liminf_{z\to\infty}\tau(z)$.
    Then, on the event $\{Z_n \to \infty\}$,
    \[
        \lim_{n\to\infty} \P\bigg(
            \sum_{k=1}^n \frac{Z_{k-1}}{Z_{n-1}} \leq \frac{1}{1 - s^{-1}}
        \bigg) = 1.
    \]
\end{lemma}

\begin{proof}
    Given $1 < s < \liminf_{z\to\infty}\tau(z)$ (where we recall $\tau(z) := \frac{m\varepsilon(z)}{z}$), we will first show that
    \begin{equation}\label{eqn:LinearMeanVarianceIncreasingTail}
        \lim_{M\to\infty} \P(Z_k > s \cdot Z_{k-1} \;\forall k \in \N_1 | Z_0 > M) = 1.
    \end{equation}
    Let $t$ be such that $s < t < \liminf_{z\to\infty}\tau(z)$. There therefore exists $N > 0$ such that $m \varepsilon(z) > tz$ for all $z \geq N$. Then, for $k \in \N_1$ such that $Z_{k-1} > N$, we note that
    \begin{align*}
        \P(Z_k > s \cdot Z_{k-1} | Z_{k-1})
        &= 1 - \P(Z_k \leq s \cdot Z_{k-1} | Z_{k-1}) \\
        &\geq 1 - \P\big(
            |Z_k - m\varepsilon(Z_{k-1})| \geq (t - s) \cdot Z_{k-1}
        \big| Z_{k-1} \big) \\
        &\geq 1 - \frac{\sigma^2 \cdot \varepsilon(Z_{k-1}) + m^2 \cdot \nu^2(Z_{k-1})}{(t - s)^2 \cdot Z^2_{k-1}} \tag{Chebyshev's inequality} \\
        &\geq 1 - \frac{\sigma^2 a + m^2 b}{(t - s)^2 \cdot Z_{k-1}} \tag{by \ref{eqn:BoundedControlFunctionMoments}} \\
        &= 1 - \frac{A_{t,s}}{Z_{k-1}},
    \end{align*}
    where $A_{t,s} := \frac{\sigma^2 a + m^2 b}{(t - s)^2}$. Hence, for $M > N$,
    \[
        \P\big(
            Z_k > s \cdot Z_{k-1} \;\forall k \in \N_1
        \big| Z_0 > M \big)
        = \prod_{k=1}^{\infty} \P\big(
            Z_k > s \cdot Z_{k-1}
        \big| Z_{k-1} > s^{k-1} \cdot M \big)
        > \prod_{k=1}^{\infty} \bigg(
            1 - \frac{A_{t,s}}{s^{k-1} M}
        \bigg).
    \]
    Given the continuity and monotonicity of the logarithm function, to show that $\prod_{k=1}^{\infty} \big( 1 - \frac{A_{t,s}}{s^{k-1}M} \big) \stackrel{M \to \infty}{\longrightarrow} 1$, it is equivalent to show that $\sum_{k=0}^{\infty} \log\big(1 - \frac{A_{t,s}}{s^k M} \big) \stackrel{M \to \infty}{\longrightarrow} 0$. Indeed, for $x \in [0,1)$, using the Taylor expansion $\log(1 - x) = - \sum_{n=1}^{\infty} \frac{x^n}{n}$, we can see that, for $M > A_{t,s} \vee N$,
    \begin{align*}
        0 > \sum_{k=0}^{\infty} \log\bigg(1 - \frac{A_{t,s}}{s^k M} \bigg)
        &= - \sum_{k=0}^{\infty} \sum_{n=1}^{\infty} \frac{1}{n} \cdot \frac{A_{t,s}^n}{s^{nk} M^n} \\
        &= - \sum_{n=1}^{\infty} \frac{A_{t,s}^n}{n M^n} \sum_{k=0}^{\infty}  \frac{1}{s^{nk}} \\
        &= - \sum_{n=1}^{\infty} \frac{A_{t,s}^n}{n M^n} \cdot \frac{1}{1-s^{-n}} \\
        &> - \sum_{n=1}^{\infty} \frac{A_{t,s}^n}{n M^n} \cdot \frac{1}{1-s^{-1}} \\
        &= \log\Big( 1 - \frac{A_{t,s}}{M} \Big) \cdot \frac{1}{1-s^{-1}} \\
        &\stackrel{M\to\infty}{\longrightarrow} 0.
    \end{align*}
    Thus, having shown \eqref{eqn:LinearMeanVarianceIncreasingTail}, we can return to the task of proving that $\lim_{n\to\infty} \P\Big( \sum_{k=1}^n \frac{Z_{k-1}}{Z_{n-1}} \leq \frac{1}{1 - s^{-1}} \Big) = 1$.
    For $n > 0$, define
    \[
        K^{(n)} := \argmax_{k < \log(n)} \{Z_k\}
        \quad\text{and}\quad
        M^{(n)} := Z_{K^{(n)}} - 1.
    \]
    We can decompose
    \begin{align*}\label{eqn:NormedSumBoundedInProbabilityDecomposition}
        &\lim_{n \to \infty} \P\bigg(
            \sum_{k=1}^n \frac{Z_{k-1}}{Z_{n-1}} \leq \frac{1}{1 - s^{-1}}
        \bigg) \\
        &\leq \P\bigg(
            \sum_{k=1}^n \frac{Z_{k-1}}{Z_{n-1}} \leq \frac{1}{1 - s^{-1}}
            \;,\;
            Z_k \geq s \cdot Z_{k-1} \;\forall k \geq K^{(n)}
        \bigg) \\
        &= \lim_{n \to \infty} \P\bigg(
            \sum_{k=1}^n \frac{Z_{k-1}}{Z_{n-1}} \leq \frac{1}{1 - s^{-1}}
        \bigg| Z_k \geq s \cdot Z_{k-1} \;\forall k \geq K^{(n)} \bigg)
            \cdot \P\big( Z_k \geq s \cdot Z_{k-1} \;\forall k \geq K^{(n)} \big). \numberthis
    \end{align*}    
    Using the fact that $\frac{K^{(n)}}{s^{n - K^{(n)} - 1}} \leq \frac{\log(n)}{s^{n - \log(n) - 1}} \stackrel{n \to \infty}{\longrightarrow} 0$ in the first step, and Lemma \ref{lma:DeterministicNormedSumBound} in the second (with $K_M := K^{(n)}$ and $M := M^{(n)}$), we can find that
    \begin{align*}\label{eqn:NormedSumBoundedInProbabilityPart1}
        &\lim_{n \to \infty} \P\bigg(
            \sum_{k=1}^n \frac{Z_{k-1}}{Z_{n-1}} \leq \frac{1}{1 - s^{-1}}
        \bigg| Z_k \geq s \cdot Z_{k-1} \;\forall k \geq K^{(n)} \bigg) \\
        &= \lim_{n \to \infty} \P\bigg(
            \sum_{k=1}^n \frac{Z_{k-1}}{Z_{n-1}} < \frac{K^{(n)}}{s^{n - K^{(n)} - 1}} + \frac{1}{1 - s^{-1}}
        \bigg| Z_k \geq s \cdot Z_{k-1} \;\forall k \geq K^{(n)} \bigg) \\
        &= 1. \numberthis
    \end{align*}
    Additionally, since $M^{(n)} \to \infty$ on $\{ Z_n \to \infty \}$, a combination of the strong Markov property and \eqref{eqn:LinearMeanVarianceIncreasingTail} will yield that
    \begin{align*}\label{eqn:NormedSumBoundedInProbabilityPart2}
        \lim_{n \to \infty}\P\big( Z_k \geq s \cdot Z_{k-1} \;\forall k \geq K^{(n)} \big)
        &= \lim_{n \to \infty} \P\big( Z_k \geq s \cdot Z_{k-1} \;\forall k \in \N_1 \big| Z_0 = M^{(n)} \big) \\
        &= \lim_{M \to \infty} \P\big( Z_k \geq s \cdot Z_{k-1} \;\forall k \in \N_1 \big| Z_0 = M \big) \\
        &= 1. \numberthis
    \end{align*}
    Substituting \eqref{eqn:NormedSumBoundedInProbabilityPart1} and \eqref{eqn:NormedSumBoundedInProbabilityPart2} into \eqref{eqn:NormedSumBoundedInProbabilityDecomposition} then yields the result.
\end{proof}

\begin{proof}[Proof of Theorem \ref{thm:KnownPhiConsistentEstimators}:]
    Let $\Z$ be a supercritical CBP with control function $\phi(\cdot)$ known, satisfying \eqref{eqn:BoundedControlFunctionMoments} and with $\sigma^2$ finite. Take $\F_n$ to be the $\sigma$-algebra generated by $(Z_0, \dots, Z_n)$.
    Note also the following, which will be used frequently throughout the proof:
    \begin{align*}\label{eqn:CBPConditionalVariance}
        \V( Z_n | \F_{n-1} )
        &= \V\big(
            \E(Z_n | \F_{n-1} \,,\, \phi(Z_{n-1}))
        \big)
        + \E\big(
            \V(Z_n | \F_{n-1} \,,\, \phi(Z_{n-1}))
        \big) \\
        &= \V(m \cdot \phi(Z_{n-1}))+\E(\sigma^2 \cdot \phi(Z_{n-1})) \\
        &= m^2 \cdot \nu^2(Z_{n-1})+\sigma^2 \cdot \varepsilon(Z_{n-1}), \numberthis
    \end{align*}
    Additionally, since \eqref{eqn:BoundedControlFunctionMoments} plus the supercriticality of the process $\Z$ imply that $\varepsilon(z) = \Theta(z)$ (recall this means that $\varepsilon(z)$ is of \textit{exact order} $z$) and $\nu^2(z) = O(z)$, there exists $C > 0$ such that
    \begin{equation}\label{eqn:NuSquaredOverEpsilonBound}
        \frac{\nu^2(z)}{\varepsilon(z)} \leq C
        \quad \text{for all}\; z \in \N_0 \;\text{such that}\; \varepsilon(z) > 0.
    \end{equation} 

    \smallskip

    \textbf{(i) Strong consistency of $\boldsymbol{\hat{m}_n}$:}
    To show that $\hat{m}_n$ is consistent on $\{ Z_n \to \infty \}$, 
    we 
    %can assume, without loss of generality, that $\P( Z_n \to \infty)>0$. We 
    consider the following two cases: (a) when $\varepsilon(z) > 0$ for all $z \in \N_0$, and (b) when $\varepsilon(z) = 0$ for some some $z$.
    % In case (a), under our assumptions, $\P(Z_n\to\infty)=1$; proving consistency on the set of unbounded growth is then equivalent to proving consistency without that restriction. In case (b), $\varepsilon(0) = 0$ implies that $0$ is an absorbing state, in which case $\P(Z_n\to\infty)<1$, and we have to carefully deal with the restriction to the set of unbounded growth.
    
    \begin{enumerate}
        \item[(a)] Under our assumptions, $\P(Z_n\to\infty)=1$; proving consistency on the set of unbounded growth is then equivalent to proving consistency without that restriction. Since $\varepsilon(z) > 0$ for all $z \in \N_0$ in this case, we have $I^{\varepsilon}_n = \{ 1, \dots, n \}$, so that $\hat{m}_n$ simplifies to
        \[
            \hat{m}_n := \frac{1}{n} \sum_{k = 1}^n \frac{Z_k}{\varepsilon(Z_{k-1})}.
        \]
        Let $U_n := \tilde{m}_n - m$, where we recall that $\tilde{m}_n := \frac{Z_n}{\varepsilon(Z_{n-1})}$.  We then have $\E(U_n | \F_{n-1}) = 0$ and $\E|U_n| \leq 2m$,
        while
        \begin{align*}
            \E( U_n^2 | \F_{n-1} )
            &= \varepsilon^{-2}(Z_{n-1}) \cdot \V( Z_n | \F_{n-1} ) \\
            &= \frac{\sigma^2}{\varepsilon(Z_{n-1})} + \frac{m^2 \nu^2(Z_{n-1})}{\varepsilon^2(Z_{n-1})} \tag{by \ref{eqn:CBPConditionalVariance}} \\
            &\leq \frac{\sigma^2 + m^2 C}{\varepsilon(Z_{n-1})} \tag{by \ref{eqn:NuSquaredOverEpsilonBound}}.
        \end{align*}
        Since $\varepsilon(z) > 0$ for all $z$ by assumption, and since $\lim_{z\to\infty} \varepsilon(z) = \infty$ by the supercriticality of the process, we can reason that $\{ \varepsilon(z) \}_{z \in \N_0}$ has its minimum bounded away from zero. Hence there exists a positive constant $C_1$ such that $\E( U_n^2 | \F_{n-1} ) \leq C_1$, so that
        \[
            \sum_{n=1}^\infty \frac{1}{n^2} \E( U_n^2 | \F_{n-1} ) < \sum_{n=1}^\infty \frac{C_1}{n^2} < \infty.
        \]
        Then by Corollary \ref{crly:SLLNforMDS}, $\frac{1}{n} \sum_{k=1}^n U_k \stackrel{a.s.}{\longrightarrow} 0$ and thus $\hat{m}_n \stackrel{a.s.}{\longrightarrow} m$.
        \medskip

        \item[(b)] In this case, there are two main differences with (a): (i) due to the existence of some $z$ such that $\varepsilon(z) = 0$, the index set $I^{\varepsilon}_n$ is no longer equivalent to $\{ 1, \dots, n \}$, and (ii) if $\varepsilon(0) = 0$, then $0$ is an absorbing state, which implies $\P(Z_n\to\infty)<1$, and we have to carefully deal with the restriction to the set of unbounded growth. 
        
        We deal with both (i) and (ii) by defining a modified process that satisfies case (a) and has the same asymptotic properties as the original process on $\{ Z_n \to \infty \}$.
        Without loss of generality, we assume that $\varepsilon(z_0) > 0$ (otherwise, a similar argument can be applied).
       % In this case, $\varepsilon(z) = 0$ for some $z$. We distinguish situations where $\varepsilon(z_0) > 0$ or $\varepsilon(z_0) = 0$. Assume first that $\varepsilon(z_0) > 0$, and 
        
        We let $L_\varepsilon=\{z\in \N_0: \varepsilon(z)=0\}$ and $\{ Z^{\uparrow}_n, z_0 \}$ be a process with
        \[
            \xi^{\uparrow} \stackrel{d}{=} \xi
            \quad\text{and}\quad
            \phi^{\uparrow}(z) \stackrel{d}{=} \begin{cases}
                z_0 \;\;\text{a.s.,} & \text{if}\;\; z\in L_\varepsilon \\
                \phi(z), & \text{otherwise.}
            \end{cases}
        \]
        This implies that, for all $z$ such that $\varepsilon(z) > 0$,
        \[
            (Z^{\uparrow}_n | Z^{\uparrow}_{n-1} = z) \stackrel{d}{=} (Z_n | Z_{n-1} = z),
        \]
        but $\{ Z^{\uparrow}_n, z_0 \}$ has $\varepsilon^{\uparrow}(z) > 0$ for all $z \in \N_0$, and therefore constitutes a process satisfying case (a).
        Hence
        \[
            \frac{1}{n} \sum_{k=1}^n U^{\uparrow}_k \stackrel{a.s.}{\longrightarrow} 0
            \;\;\text{as}\;\;
            n \to \infty,
            \quad\text{for}\;\;
            U^{\uparrow}_k := \frac{Z^{\uparrow}_k}{\varepsilon^{\uparrow}( Z^{\uparrow}_{k-1})} - m .
        \]
        We define $T^{\uparrow}_0$, respectively $T_0$, as the last time $\{ Z^{\uparrow}_n, z_0 \}$, resp. $\{ Z_n, z_0 \}$, visits a state in $L_\varepsilon$ before escaping to infinity, with the convention that $T^{\uparrow}_0=0$, resp. $T_0=0$,  if the process never visits $L_\varepsilon$.
        Because $\P( Z^{\uparrow}_n \to \infty) = 1$, $T^{\uparrow}_0$ is almost surely finite.
        Similarly, on $\{ Z_n \to \infty \}$, $T_0$ is almost surely finite. We then have
        \[
            \{ Z^{\uparrow}_n, z_0 \}_{n \geq T^{\uparrow}_0} \stackrel{d}{=} \{Z_n, z_0\}_{n \geq T_0}
            \quad\text{on}\quad
            \{ Z_n \to \infty \},
        \]
        so that
        \begin{equation}\label{eqn:EquivalentSumLimits}
            \lim_{n \to \infty} \frac{1}{n} \sum_{k=T^{\uparrow}_0}^{n} U^{\uparrow}_k
            =
            \lim_{n \to \infty} \frac{1}{n} \sum_{k=T_0}^n U_k
            \quad\text{a.s.\ on}\;
            \{ Z_n \to \infty \}.
        \end{equation}
        Since there are an a.s.\ finite number of generations (i.e., terms in the sum) before $T^{\uparrow}_0$ and $T_0$, we have
        \begin{equation}\label{eqn:PartialSumsEqualCompleteSums}
            \frac{1}{n} \sum_{k=1}^n U^{\uparrow}_k \stackrel{a.s.}{\longrightarrow} 0
            \implies \frac{1}{n} \sum_{k=T^{\uparrow}_0}^n U^{\uparrow}_k \stackrel{a.s.}{\longrightarrow} 0
            \quad\text{and}\quad
            \frac{1}{n} \sum_{k=T_0}^n U_k \stackrel{a.s.}{\longrightarrow} 0
            \implies \frac{1}{| I^{\varepsilon}_n |} \sum_{k \in I^{\varepsilon}_n} U_k \stackrel{a.s.}{\longrightarrow} 0.
        \end{equation}
        By combining \eqref{eqn:EquivalentSumLimits} and \eqref{eqn:PartialSumsEqualCompleteSums}, it follows that $\hat{m}_n \stackrel{a.s.}{\longrightarrow} m$ on the event $\{ Z_n \to \infty \}$. \\

        % The second case to consider is when $\varepsilon(z_0) = 0$. In this case $\Z$ must have $\varepsilon(0) > 0$ to have a positive probability of escaping to infinity. In this case we define the random variable $\phi^+(0)$ by
        % \[
        %     \P( \phi^+(0) = 0) = 0,
        %     \quad\text{and}\quad
        %     \P( \phi^+(0) = k) = \frac{\P( \phi(0) = k) }{1 - \P(\phi(0) = 0)}
        %     \quad\text{for}\;\;
        %     k \geq 1,
        % \]
        % and define $\{ Z^{\uparrow}_n, z_0 \}$ to have
        % \[
        %     \xi^{\uparrow} \stackrel{d}{=} \xi
        %     \quad\text{and}\quad
        %     \phi^{\uparrow}(z) \stackrel{d}{=} \begin{cases}
        %         \phi^+(0), & \text{if}\;\; \varepsilon(z) = 0 \\
        %         \phi(z), & \text{if}\;\;  \varepsilon(z) > 0,
        %     \end{cases}
        % \]
        % and then follow essentially the same procedure as for the previous case to show that, here too, $\hat{m}_n \stackrel{a.s.}{\longrightarrow} m$ on the event $\{ Z_n \to \infty \}$.
    \end{enumerate}

    \medskip

    With the consistency of $\hat{m}_n$ shown, it remains to prove the consistency of $\bar{\sigma}^2_n$ and $\hat{\sigma}^2_n$. Accordingly, we now further assume that that there exist positive constants $c$ and $d$ such that $\sup_{z \geq 1} \left\{ \frac{|\iota(z)|}{z} \right\} \leq c$ and $\sup_{z \geq 1} \Big\{ \frac{\E(\phi(z) - \varepsilon(z))^4}{z^2} \Big\} \leq d$, and that $\E(\xi - m)^4$ is finite.
    In addition, to avoid repeating ourselves, we hereafter also assume that $\varepsilon(z) > 0$ for all $z \in \N_0$, so that our two estimators simplify to
    \[
        \bar{\sigma}^2_n := \frac{1}{n} \sum_{k = 1}^n \frac{(Z_k - m \cdot \varepsilon(Z_{k-1}))^2 - m^2 \cdot \nu^2(Z_{k-1})}{\varepsilon(Z_{k-1})}
    \]
    and
    \[
        \hat{\sigma}^2_n := \frac{1}{n} \sum_{k = 1}^n \frac{(Z_k - \tilde{m}_n \cdot \varepsilon(Z_{k-1}))^2 - \tilde{m}_n^2 \cdot \nu^2(Z_{k-1})}{\varepsilon(Z_{k-1})}
    \]    
    We omit the extension to the case where $\varepsilon(z) = 0$ for some $z \in \N_0$, which follows the same argument as in in case (b) above.

    \bigskip

    \textbf{(ii) Strong consistency of $\boldsymbol{\bar{\sigma}^2_n}$:}
    In this case, we consider $m$ to be known. We introduce the estimator
    \[
        \tilde{\sigma}_n^2 := \frac{(Z_n - m \cdot \varepsilon(Z_{n-1}))^2 - m^2 \cdot \nu^2(Z_{n-1})}{\varepsilon(Z_{n-1})},
    \]
    such that $\bar{\sigma}_n^2 = \frac{1}{n} \sum_{k=1}^n \tilde{\sigma}_k^2$, and the random variable
    \[
        U_n := \tilde{\sigma}_n^2 - \sigma^2,
    \]
    for which we use \eqref{eqn:CBPConditionalVariance} to see that
    \[
        \E(U_n | \F_{n-1})
        = \varepsilon^{-1}(Z_{n-1}) \cdot \Big\{
            \V( Z_n | \F_{n-1} ) - m^2 \cdot \nu^2(Z_{n-1})
        \Big\} - \sigma^2=0.
    \]
    Thus $\E U_n = \E\big( \E( U_n | \F_{n-1} ) \big) = 0$, and given that $\nu^2(Z_{n-1}) / \varepsilon(Z_{n-1}) < C$ by \eqref{eqn:NuSquaredOverEpsilonBound}, we have
    \begin{align*}
        \E|U_n|
        &\leq \E\Bigg|
            \frac{(Z_n - m \cdot \varepsilon(Z_{n-1}))^2}{\varepsilon(Z_{n-1})}
        \Bigg|
        + \E\Bigg|
            \frac{m^2 \cdot \nu^2(Z_{n-1})}{\varepsilon(Z_{n-1})}
        \Bigg| + \sigma^2 \\
        &= \E U_n + 2 m^2 \cdot \E\Bigg(
            \frac{\nu^2(Z_{n-1})}{\varepsilon(Z_{n-1})}
        \Bigg) + 2 \sigma^2 \\
        &< \infty.
    \end{align*}
    In addition,
    \[
        \sum_{n=1}^{\infty} \frac{1}{n^2} \E U_n^2
        = \sum_{n=1}^{\infty} \frac{1}{n^2} \Big(
            \E\tilde{\sigma}_n^4 - \sigma^4
        \Big)
        < \sum_{n=1}^{\infty} \frac{1}{n^2} \E\tilde{\sigma}_n^4.
    \]
    We have that
    \begin{align*}
        \E\big( \tilde{\sigma}_n^4 \big| \F_{n-1} \big)
        &= \varepsilon^{-2}(Z_{n-1}) \cdot \E\Big[
            \big( Z_n - m \cdot \varepsilon(Z_{n-1}) \big)^4
        \Big| \F_{n-1} \Big]
        - 2 m^2 \sigma^2 \cdot \frac{\nu^2(Z_{n-1})}{\varepsilon(Z_{n-1})}
        - m^4 \cdot \frac{\nu^4(Z_{n-1})}{\varepsilon^2(Z_{n-1})} \\
        &< \varepsilon^{-2}(Z_{n-1}) \cdot \E\Big[
            \big( Z_n - m \cdot \varepsilon(Z_{n-1}) \big)^4
        \Big| \F_{n-1} \Big],
    \end{align*}
    and by a lengthy yet elementary expansion can find that
    \begin{align*}\label{eqn:FourthCentralMomentBound}
        \E\Big[
            \big( Z_n - m \cdot \varepsilon(Z_{n-1}) \big)^4
        \Big| \F_{n-1} \Big]
        &= m^4 \E(\phi(Z_{n-1}) - \varepsilon(Z_{n-1}))^4 + 6 \sigma^2 m^2 \iota(Z_{n-1}) \\
        &\quad - 12m^2\sigma^2 \varepsilon(Z_{n-1}) \nu^2(Z_{n-1}) \\
        &\quad + \big( 4 \gamma m + 3\sigma^4 + 18 \sigma^2 m^2 \big) \nu^2(Z_{n-1}) + 3\sigma^4 \varepsilon^2(Z_{n-1}) \\
        &\quad + \big( \E(\xi - m)^4 - 3\sigma^4 \big) \varepsilon(Z_{n-1}), \numberthis
    \end{align*}
    so that
    \begin{align*}
        \E\big( \tilde{\sigma}_n^4 \big| \F_{n-1} \big)
        &< \frac{m^4 \E(\phi(Z_{n-1}) - \varepsilon(Z_{n-1}))^4}{\varepsilon^2(Z_{n-1})} + \frac{6 \sigma^2 m^2 \iota(Z_{n-1})}{\varepsilon^2(Z_{n-1})}  \\
        &\quad + \frac{\big( 4 \gamma m + 3\sigma^4 + 18 \sigma^2 m^2 \big) \nu^2(Z_{n-1})}{\varepsilon^2(Z_{n-1})} + 3\sigma^4 + \frac{\E(\xi - m)^4 - 3\sigma^4}{\varepsilon(Z_{n-1})}.
    \end{align*}
    Given \eqref{eqn:BoundedControlFunctionMoments} and our assumptions on the bounds on $|\iota(Z_{n-1})|$, $\E(\phi(Z_{n-1}) - \varepsilon(Z_{n-1}))^4$ and $\E(\xi - m)^4$,
    there will therefore exist a positive constant $C_1$ such that $\E\big( \tilde{\sigma}_n^4 \big| \F_{n-1} \big) \leq C_1$.
    Hence
    \[
        \sum_{n=1}^{\infty} \frac{1}{n^2} \E U_n^2
        < \sum_{n=1}^{\infty} \frac{1}{n^2} \E\big(\E\big( \tilde{\sigma}_n^4 \big| \F_{n-1} \big)\big)
        < \sum_{n=1}^\infty \frac{C_1}{n^2}
        < \infty.
    \]
    Consequently, by Corollary \ref{crly:SLLNforMDS}, $\frac{1}{n} \sum_{k=1}^n U_k \stackrel{a.s.}{\longrightarrow} 0$ and thus $\bar{\sigma}^2_n \stackrel{a.s.}{\longrightarrow} \sigma^2$.

    \bigskip

    \textbf{(iii) Weak consistency of $\boldsymbol{\hat{\sigma}^2_n}$:} Given the decomposition
    \[
        \hat{\sigma}^2_n
        = \bar{\sigma}^2_n
        + \underbrace{
            \frac{2(m - \tilde{m}_n)}{n} \sum_{k=1}^n \big( Z_k - m \cdot \varepsilon(Z_{k-1}) \big)
        }_{\text{I}}
        + \underbrace{
            \frac{(m - \tilde{m}_n)^2}{n} \sum_{k=1}^n \varepsilon(Z_{k-1})
        }_{\text{II}}
        + \underbrace{
            \frac{m^2 - \tilde{m}^2_n}{n} \sum_{k=1}^n \frac{\nu^2(Z_{k-1})}{\varepsilon(Z_{k-1})},
        }_{\text{III}}
    \]
    and since we showed in (ii) that $\bar{\sigma}^2_n \stackrel{a.s.}{\longrightarrow} \sigma^2$, the result $\hat{\sigma}^2_n \stackrel{P}{\longrightarrow} \sigma^2$ will follow if we can show that I, II, and III converge in probability to zero.

    \begin{enumerate}
        \item[(I)] Recalling that $\tilde{m}_n := Z_n / \varepsilon(Z_{n-1})$, we can decompose
        \[
            \hspace*{-0.6cm}
            \frac{2(m - \tilde{m}_n)}{n} \sum_{k=1}^n \big( Z_k - m \cdot \varepsilon(Z_{k-1}) \big)
            = -2 \underbrace{
                \sqrt{\frac{\sum_{k=1}^n \varepsilon(Z_{k-1})}{\varepsilon(Z_{n-1})}}
            }_{\text{(a)}} \cdot \underbrace{
                \frac{Z_n - m \varepsilon(Z_{n-1})}{n^{1/4} \cdot \sqrt{\varepsilon(Z_{n-1})}}
            }_{\text{(b)}} \cdot \underbrace{
                 \frac{\sum_{k=1}^n(Z_k - m \varepsilon(Z_{k-1}))}{n^{3/4}\sqrt{\sum_{k=1}^n \varepsilon(Z_{k-1})}}
            }_{\text{(c)}},
        \]
         and consider each of the components $(a)$, $(b)$ and $(c)$ in turn. \\

        \begin{enumerate}
            \item[(a)] Given \eqref{eqn:BoundedControlFunctionMoments}, $\varepsilon(z) \leq a z$. Additionally, for $t$ such that $1 < t < \liminf_{z\to\infty} \tau(z)$, there exists $N > 0$ such that, for $z \geq N$, $\varepsilon(z) > \frac{t}{m} z$. Therefore, on $\{ Z_n \to \infty \}$,
            \begin{equation}\label{eqn:NormedEpsilonSumLimitBound}
                \lim_{n \to \infty} \frac{\sum_{k=1}^n \varepsilon(Z_{k-1})}{\varepsilon(Z_{n-1})}
                \leq \lim_{n \to \infty} \frac{am}{t} \cdot \frac{\sum_{k=1}^n Z_{k-1}}{Z_{n-1}} \quad\text{a.s.}
            \end{equation}
            
            Then, as a result of Lemma \ref{lma:NormedSumBoundedInProbability}, for any $s$ such that $1 < s < t$, we have that
            \[
                \lim_{n\to\infty} \P\Bigg(
                    \sqrt{\frac{\sum_{k=1}^n \varepsilon(Z_{k-1})}{\varepsilon(Z_{n-1})}}
                    \leq \sqrt{\frac{t}{am(1 - s^{-1})}}
                \Bigg) = 1,
            \]
            that is, $\sqrt{\frac{\sum_{k=1}^n \varepsilon(Z_{k-1})}{\varepsilon(Z_{n-1})}}$ is asymptotically bounded by a constant in probability.
            \medskip

            \item[(b)] Given that
            \[
                \E\bigg(
                    \frac{Z_n - m \varepsilon(Z_{n-1})}{n^{1/4} \cdot \sqrt{\varepsilon(Z_{n-1})}}
                \bigg| \F_{n-1} \bigg) = 0,
            \]
            and given that, by \eqref{eqn:BoundedControlFunctionMoments} and \eqref{eqn:NuSquaredOverEpsilonBound},
            \[
                \V\bigg(
                    \frac{Z_n - m \varepsilon(Z_{n-1})}{n^{1/4} \cdot \sqrt{\varepsilon(Z_{n-1})}}
                \bigg| \F_{n-1} \bigg)
                = \frac{\sigma^2 \cdot \varepsilon(Z_{n-1}) + m^2 \cdot \nu^2(Z_{n-1})}{\sqrt{n} \cdot \varepsilon(Z_{n-1})}
                \leq \frac{\sigma^2 + m^2 C}{\sqrt{n}}
                = O\bigg( \frac{1}{\sqrt{n}} \bigg),
            \]
            we see that
            \[
                \E\bigg(
                    \frac{Z_n - m \varepsilon(Z_{n-1})}{n^{1/4} \cdot \sqrt{\varepsilon(Z_{n-1})}}
                \bigg)
                = 0
                \quad\text{and}\quad
                \V\bigg(
                    \frac{Z_n - m \varepsilon(Z_{n-1})}{n^{1/4} \cdot \sqrt{\varepsilon(Z_{n-1})}}
                \bigg)
                \stackrel{n\to\infty}{\longrightarrow} 0,
            \]
            so, as a consequence of Chebyshev's inequality, it follows that
            \[
                \frac{Z_n - m \varepsilon(Z_{n-1})}{n^{1/4} \cdot \sqrt{\varepsilon(Z_{n-1})}}
                \stackrel{P}{\longrightarrow} 0.
            \]

            \item[(c)] Defining
            \[
                U_n := Z_n - m \varepsilon(Z_{n-1})
                \quad\text{and}\quad
                J_n := n^{3/4}\sqrt{\sum_{k=1}^n \varepsilon(Z_{k-1}}),
            \]
            we can see that $\{U_n\}_{n\in\N_0}$ is a martingale difference sequence, since $\E(U_n | \F_{n-1}) = 0$ for all $n \in \N_1$ and
            \begin{align*}
                \E|U_n|
                &\leq \sqrt{\E(U_n^2)} \tag{H\"{o}lder's inequality} \\
                &= \sqrt{\E(\V(Z_n | \F_{n-1}))} \\
                &\leq \sqrt{(\sigma^2 a + m^2 b) \cdot \E(Z_{n-1})} \tag{by \ref{eqn:BoundedControlFunctionMoments}} \\
                &< \infty,
            \end{align*}
            and $\{J_n\}_{n\in\N_0}$ is positive, non-decreasing, with $\lim_{n\to\infty} J_n = \infty$.
            Additionally, using \eqref{eqn:CBPConditionalVariance} and \eqref{eqn:NuSquaredOverEpsilonBound},
            \begin{align*}
                \sum_{n=1}^{\infty} J_n^{-2} \E(U_n^2 | \F_{n-1})
                &= \sum_{n=1}^{\infty} \frac{\V(Z_n | \F_{n-1})}{n^{3/2} \cdot \sum_{k=1}^n \varepsilon(Z_{k-1})} \\
                &< \sum_{n=1}^{\infty} \frac{\V(Z_n | \F_{n-1})}{n^{3/2} \cdot \varepsilon(Z_{n-1})} \\
                &\leq \sum_{n=1}^{\infty} \frac{\sigma^2 + m^2 C}{n^{3/2}} \\
                &< \infty \;\;\text{a.s.},
            \end{align*}
            so Theorem \ref{thm:SLLNforMDS} applies, such that
            $J_n^{-1} \sum_{k=1}^n U_k = \frac{\sum_{k=1}^n(Z_k - m \varepsilon(Z_{k-1}))}{n^{3/4}\sqrt{\sum_{k=1}^n \varepsilon(Z_{k-1})}} \stackrel{a.s.}{\longrightarrow} 0$ as $n \to \infty$. \\
        \end{enumerate}

        Then, since $(a)$ is asymptotically bounded by a constant in probability, and $(b)$ and $(c)$ both converge to zero in probability, $\frac{2(m - \tilde{m}_n)}{n} \sum_{k=1}^n \big( Z_k - m \cdot \varepsilon(Z_{k-1}) \big) \stackrel{P}{\longrightarrow} 0$. \\

        \item[(II)] We can decompose
        \[
            \frac{(m - \tilde{m}_n)^2}{n} \sum_{k=1}^n \varepsilon(Z_{k-1})
            = \frac{(Z_n - m \varepsilon(Z_{n-1}))^2}{n \cdot \varepsilon(Z_{n-1})} \cdot \sum_{k=1}^n \frac{\varepsilon(Z_{k-1})}{\varepsilon(Z_{n-1})},
        \]
        where we know from Lemma \ref{lma:NormedSumBoundedInProbability} and \eqref{eqn:NormedEpsilonSumLimitBound}  that, on $\{ Z_n \to \infty \}$ and for $s,\, t$ such that $1 < s < t < \liminf_{z\to\infty} \tau(z)$,
        \[
            \lim_{n\to\infty} \P\Bigg(
                \sum_{k=1}^n \frac{\varepsilon(Z_{k-1})}{\varepsilon(Z_{n-1})}
                \leq \frac{t}{am(1 - s^{-1})}
            \Bigg) = 1.
        \]
        Therefore the result $\frac{(m - \tilde{m}_n)^2}{n} \sum_{k=1}^n \varepsilon(Z_{k-1}) \stackrel{P}{\longrightarrow} 0$ will follow if we show that
        \begin{equation}\label{eqn:KnownPhiIISufficientCondition}
            \frac{(Z_n - m \varepsilon(Z_{n-1}))^2}{n \cdot \varepsilon(Z_{n-1})} \stackrel{P}{\longrightarrow} 0.
        \end{equation}
        Since $\frac{(Z_n - m \varepsilon(Z_{n-1}))^2}{n \cdot \varepsilon(Z_{n-1})}$ is non-negative, and given \eqref{eqn:NuSquaredOverEpsilonBound},
        \[
            \E\Bigg( \frac{(Z_n - m \varepsilon(Z_{n-1}))^2}{n \cdot \varepsilon(Z_{n-1})} \Bigg)
            = \E\Bigg( \E\bigg(
                \frac{(Z_n - m \varepsilon(Z_{n-1}))^2}{n \cdot \varepsilon(Z_{n-1})}
            \bigg| \F_{n-1} \bigg) \Bigg)
            \leq \frac{\sigma^2 + m^2 C}{n}
            \stackrel{n\to\infty}{\longrightarrow} 0.
        \]
        Hence \eqref{eqn:KnownPhiIISufficientCondition} follows from Markov's inequality. \\

        \item[(III)] We want to show that $\frac{m^2 - \tilde{m}^2_n}{n} \sum_{k=1}^n \frac{\nu^2(Z_{k-1})}{\varepsilon(Z_{k-1})} \stackrel{P}{\longrightarrow} 0$.
        Since there exists $C > 0$ such that $\frac{\nu^2(z)}{\varepsilon(z)} \leq C$ for all $z$ by \eqref{eqn:NuSquaredOverEpsilonBound}, for any $n \in \N_1$,
        \[
            \bigg|
                \frac{m^2 - \tilde{m}^2_n}{n} \sum_{k=1}^n \frac{\nu^2(Z_{k-1})}{\varepsilon(Z_{k-1})}
            \bigg|
            \leq C \cdot |m^2 - \tilde{m}^2_n|.
        \]
        We note that
        \[
            \E(\tilde{m}_n)
            = \E\bigg(
                \frac{Z_n}{\varepsilon(Z_{n-1})}
            \bigg)
            = \E\big(
                \varepsilon^{-1}(Z_{n-1}) \cdot \E(Z_n | \F_{n-1})
            \big)
            = m
        \]
        and
        \begin{align*}
            \V(\tilde{m}_n)
            &= \E\bigg(
                \V\bigg(
                    \frac{Z_n}{\varepsilon(Z_{n-1})}
                \bigg| \F_{n-1} \bigg)
            \bigg)
            + \V\bigg(
                \E\bigg(
                    \frac{Z_n}{\varepsilon(Z_{n-1})}
                \bigg| \F_{n-1} \bigg)
            \bigg) \\
            &= \E\big(
                \varepsilon^{-2}(Z_{n-1}) \cdot \V( Z_n | \F_{n-1} )
            \big)
            + \V(m) \\
            &= \E\bigg(
                \frac{\sigma^2}{\varepsilon(Z_{n-1})}
            \bigg)
            + \E\bigg(
                \frac{m^2 \cdot \nu^2(Z_{n-1})}{\varepsilon^2(Z_{n-1})}
            \bigg).
        \end{align*}
        We have argued in (I) that there exist $N \in \N_1,\; t > 1$ such that $\varepsilon(Z_{n-1}) > \frac{t Z_{n-1}}{m}$ for all $n > N$. Then, given that $\nu^2(z) \leq b \cdot z$ by \eqref{eqn:BoundedControlFunctionMoments},
        \[
            \lim_{n\to\infty} \V(\tilde{m}_n)
            \leq \lim_{n\to\infty} \frac{m\sigma^2}{t} \cdot \E(Z^{-1}_{n-1})
            + \lim_{n\to\infty} \frac{b m^4}{t^2} \cdot \E(Z^{-1}_{n-1})
            = 0
        \]
        on $\{ Z_n \to \infty \}$, and thus $\tilde{m}_n \stackrel{P}{\longrightarrow} m$. It follows by the continuous mapping theorem that $m^2 - \tilde{m}^2_n \stackrel{P}{\longrightarrow} 0$.
        Thus we see that
        \[
            \frac{m^2 - \tilde{m}^2_n}{n} \sum_{k=1}^n \frac{\nu^2(Z_{k-1})}{\varepsilon(Z_{k-1})}
            \stackrel{P}{\longrightarrow} 0.
        \]
    \end{enumerate}
    Since we have shown that (I), (II), and (III) all converge to zero in probability, we have shown that $\hat{\sigma}^2_n \stackrel{P}{\longrightarrow} \sigma^2$.
\end{proof}

\bigskip

The proof of Theorem \ref{thm:LinearMeanVarPhiConsistentEstimators} follows the same arguments as that of Theorem \ref{thm:KnownPhiConsistentEstimators}, albeit in a simpler setting.

\begin{proof}[Proof of Theorem \ref{thm:LinearMeanVarPhiConsistentEstimators}:]

    Let $\Z$ be a supercritical CBP with unknown control function $\phi(\cdot)$ having $\varepsilon(z) = \alpha z$ and $\nu^2(z) = \beta z$, and with $\sigma^2$ finite. Take $\F_n$ to be the $\sigma$-algebra generated by $(Z_0, \dots, Z_n)$. In this setting, the conditional variance of $Z_n$ given $\F_{n-1}$ has the following form:
    \begin{equation}\label{eqn:LinearMeanVarianceCBPConditionalVariance}
        \V( Z_n | \F_{n-1} )
        = \sigma^2 \alpha + m^2 \beta.
    \end{equation}
    Throughout the proof, we 
    %assume without loss of generality that $\P(Z_n\to\infty) > 0$, and 
    make the simplifying assumption that $\P(Z_1 > 0 \,|\, Z_0 = z) = 1$ for all $z \in \N_1$, so that $I_n = \{1, \dots, n\}$. Under this assumption, $\P(Z_n\to\infty)=1$, and the estimators simplify to
    \[
        \hat{g}_n := \frac{1}{n} \sum_{k=1}^n \frac{Z_k}{Z_{k-1}},
    \]
    \[
        \bar{h}_n := \frac{1}{n} \sum_{k = 1}^n \frac{(Z_k - m \alpha \cdot Z_{k-1})^2}{Z_{k-1}},
        \quad\text{and}\quad
        \hat{h}_n := \frac{1}{n} \sum_{k = 1}^n \frac{(Z_k - \tilde{g}_n \cdot Z_{k-1})^2}{Z_{k-1}}.
    \]
    The extension of the proof to the case where there exists $z \in \N_1$ such that $\P(Z_1 = 0 \,|\, Z_0 = z) > 0$ follows the same arguments as in the proof of Theorem \ref{thm:KnownPhiConsistentEstimators}, part (i), case (b). To avoid repeating ourselves, we omit it. \\

    \textbf{(i) Strong consistency of $\boldsymbol{\hat{g}_n}$:}
    Let $U_n := \tilde{g}_n - m\alpha$, where we recall that $\tilde{g}_n := Z_n / Z_{n-1}$. We then have $\E(U_n) = \E(\E(U_n | \F_{n-1})) = 0$ and $\E|U_n| \leq 2m\alpha$,
    while we can use \eqref{eqn:LinearMeanVarianceCBPConditionalVariance} to see that
    \[
        \E( U_n^2 | \F_{n-1} )
        = Z_{n-1}^{-2} \cdot \V( Z_n | \F_{n-1} )
        = \frac{\sigma^2\alpha + m^2 \beta}{Z_{n-1}}.
    \]
    Our assumption that $\P(Z_1 > 0 | Z_0 = z) = 1$ for all $z \in \N_1$ implies that, for all $k \in \N_1$, $\P(Z_k = 0 | Z_0 = z) = 0$, and thus that $Z_k \geq 1$ a.s..
    Hence
    \[
        \frac{\sigma^2\alpha + m^2 \beta}{Z_{n-1}} \leq \sigma^2\alpha + m^2 \beta,
    \]
    and therefore
    \[
        \sum_{n=1}^\infty \frac{1}{n^2} \E( U_n^2 | \F_{n-1} )
        \leq \sum_{n=1}^\infty \frac{\sigma^2\alpha + m^2 \beta}{n^2}
        < \infty.
    \]
    Then, by Corollary \ref{crly:SLLNforMDS}, $\frac{1}{n} \sum_{k=1}^n U_k \stackrel{a.s.}{\longrightarrow} 0$ and thus $\hat{g}_n \stackrel{a.s.}{\longrightarrow} m\alpha$.

    \bigskip

    \textbf{(ii) Strong consistency of $\boldsymbol{\bar{h}_n}$:}
    Under the assumption that $m\alpha$ is known, we introduce the random variable
    \[
        U_n := \frac{(Z_n - m\alpha \cdot Z_{n-1})^2}{Z_{n-1}}
        - \sigma^2\alpha - m^2 \beta,
    \]
    for which we see that
    \[
        \E(U_n | \F_{n-1})
        = Z_{n-1}^{-1} \cdot \V( Z_n | \F_{n-1} ) - \sigma^2\alpha - m^2\beta
        = 0,
    \]
    so that consequently $\E(U_n) = 0$, and
    \[
        \E|U_n|
        \leq \E\Bigg|
            \frac{(Z_n - m\alpha \cdot Z_{n-1})^2}{Z_{n-1}}
        \Bigg|
        + \sigma^2\alpha + m^2\beta 
        = 2\sigma^2\alpha + 2m^2\beta
        < \infty.
    \]
    In addition, by a lengthy yet elementary expansion, we can find that
    \begin{align*}
        \E(U_n^2 | \F_{n-1} )
        &= \frac{m^4 \E(\phi(Z_{n-1}) - \varepsilon(Z_{n-1}))^4}{Z_{n-1}^2}
        + \frac{6\sigma^2 m^2 \iota(Z_{n-1})}{Z_{n-1}^2} + \frac{\alpha \E(\xi - m)^4}{Z_{n-1}} \\
        &\quad + \frac{4\gamma m\beta + 3 \sigma^4\beta + 18\sigma^2 m^2 \beta + m^4 \alpha^4 - 3 \sigma^4 \alpha}{Z_{n-1}}
        + 3\sigma^4\alpha^2 - 12 \sigma^2 m^2 \alpha \beta + 12 \sigma^2 m^2 \alpha^2 \\
        &\quad + 4m^4 \alpha^2 - (\sigma^2\alpha + m^2\beta)^2.
    \end{align*}
    Given our assumptions on the bounds on $\E(\phi(Z_{n-1}) - \varepsilon(Z_{n-1}))^4$ and $\E(\xi - m)^4$,
    there will therefore exist a positive constant $C$ such that $\E\big( U_n^2 \big| \F_{n-1} \big) \leq C$. Hence
    \[
        \sum_{n=1}^{\infty} \frac{1}{n^2} \E(U_n^2 | \F_{n-1} ) \leq \sum_{n=1}^\infty \frac{C}{n^2} < \infty.
    \]
    Then by Corollary \ref{crly:SLLNforMDS}, $\frac{1}{n} \sum_{k=1}^n U_k \stackrel{a.s.}{\longrightarrow} 0$ and thus $\bar{h}_n \stackrel{a.s.}{\longrightarrow} \sigma^2\alpha + m^2\beta$.

    \bigskip

    \textbf{(iii) Weak consistency of $\boldsymbol{\hat{h}_n}$:} We abandon the assumption that $m\alpha$ is known. We can decompose
    \[
        \hat{h}_n
        = \bar{h}_n
        + \underbrace{
            \frac{2(m\alpha - \tilde{g}_n)}{n} \sum_{k=1}^n (Z_k - m \alpha Z_{k-1})
        }_{\text{I}}
        + \underbrace{
            \frac{(m\alpha - \tilde{g}_n)^2}{n} \sum_{k=1}^n Z_{k-1},
        }_{\text{II}}
    \]
    where we have shown in (ii) that $\bar{h}_n \stackrel{a.s.}{\longrightarrow} \sigma^2\alpha + m^2\beta$.
    It remains to show that I and II both converge in probability to zero. \\

    \begin{enumerate}
        \item[(I)]
        We can decompose
        \[
            \frac{2(m\alpha - \tilde{g}_n)}{n} \sum_{k=1}^n (Z_k - m \alpha Z_{k-1})
            = -2 \underbrace{
                \sqrt{\frac{\sum_{k=1}^n Z_{k-1}}{Z_{n-1}}}
            }_{\text{(a)}} \cdot \underbrace{
                \frac{Z_n - m\alpha Z_{n-1}}{n^{1/4} \cdot \sqrt{Z_{n-1}}}
            }_{\text{(b)}} \cdot \underbrace{
                 \frac{\sum_{k=1}^n(Z_k - m \alpha Z_{k-1})}{n^{3/4}\sqrt{\sum_{k=1}^n Z_{k-1}}}
            }_{\text{(c)}},
        \]
         and consider each of the components $(a)$, $(b)$ and $(c)$ in turn. \\

        \begin{enumerate}
            \item[(a)] As a result of Lemma \ref{lma:NormedSumBoundedInProbability}, for any $s$ such that $1 < s < m\alpha$, we have that
            \[
                \lim_{n\to\infty} \P\Bigg(
                    \sqrt{\frac{\sum_{k=1}^n Z_{k-1}}{Z_{n-1}}} \leq \sqrt{\frac{1}{1 - s^{-1}}}
                \Bigg) = 1,
            \]
            that is, $\sqrt{\frac{\sum_{k=1}^n Z_{k-1}}{Z_{n-1}}}$ is asymptotically bounded by a constant in probability.
            \medskip

            \item[(b)] Given that
            \[
                \E\bigg( \frac{Z_n - m\alpha Z_{n-1}}{n^{1/4} \cdot \sqrt{Z_{n-1}}} \bigg| \F_{n-1} \bigg) = 0
                \quad\text{and}\quad
                \V\bigg( \frac{Z_n - m\alpha Z_{n-1}}{n^{1/4} \cdot \sqrt{Z_{n-1}}} \bigg| \F_{n-1} \bigg) = \frac{\sigma^2 \alpha + m^2 \beta}{\sqrt{n}},
            \]
            we see that
            \[
                \E\bigg(
                    \frac{Z_n - m\alpha Z_{n-1}}{n^{1/4} \cdot \sqrt{Z_{n-1}}}
                \bigg)
                = 0
                \quad\text{and}\quad
                \V\bigg(
                    \frac{Z_n - m\alpha Z_{n-1}}{n^{1/4} \cdot \sqrt{Z_{n-1}}}
                \bigg)
                \stackrel{n\to\infty}{\longrightarrow} 0,
            \]
            so, by Chebyshev's inequality, it follows that
            \[
                \frac{Z_n - m\alpha Z_{n-1}}{n^{1/4} \cdot \sqrt{Z_{n-1}}}
                \stackrel{P}{\longrightarrow} 0.
            \]

            \item[(c)] Defining
            \[
                U_n := Z_n - m\alpha Z_{n-1}
                \quad\text{and}\quad
                J_n := n^{3/4}\sqrt{\sum_{k=1}^n Z_{k-1}},
            \]
            we can see that $\{U_n\}_{n\in\N_0}$ is a martingale difference sequence, since $\E(U_n | \F_{n-1}) = 0$ for all $n \in \N_1$ and
            \begin{align*}
                \E|U_n|
                &\leq \sqrt{\E(U_n^2)} \tag{H\"{o}lder's inequality} \\
                &= \sqrt{\E(\V(Z_n | \F_{n-1}))} \\
                &= \sqrt{(\sigma^2 \alpha + m^2 \beta) \cdot \E(Z_{n-1})} \\
                &< \infty,
            \end{align*}
            and $\{J_n\}_{n\in\N_0}$ is positive, non-decreasing, with $\lim_{n\to\infty} J_n = \infty$.
            Additionally,
            \begin{align*}
                \sum_{n=1}^{\infty} J_n^{-2} \E(U_n^2 | \F_{n-1})
                &= \sum_{n=1}^{\infty} \frac{\V(Z_n | \F_{n-1})}{n^{3/2} \cdot \sum_{k=1}^n Z_{k-1}} \\
                &= \sum_{n=1}^{\infty} \frac{(\sigma^2 \alpha + m^2 \beta) \cdot Z_{n-1}}{n^{3/2} \cdot \sum_{k=1}^n Z_{k-1}} \\
                &< \sum_{n=1}^{\infty} \frac{(\sigma^2 \alpha + m^2 \beta)}{n^{3/2}} \\
                &< \infty \;\;\text{a.s.},
            \end{align*}
            so Theorem \ref{thm:SLLNforMDS} applies, such that
            $J_n^{-1} \sum_{k=1}^n U_k = \frac{\sum_{k=1}^n (Z_k - m \alpha Z_{k-1})}{n^{3/4}\sqrt{\sum_{k=1}^n Z_{k-1}}} \stackrel{a.s.}{\longrightarrow} 0$ as $n \to \infty$. \\
        \end{enumerate}

        Then, since $(a)$ is asymptotically bounded by a constant in probability, and $(b)$ and $(c)$ both converge to zero in probability, $\frac{2(m\alpha - \tilde{g}_n)}{n} \sum_{k=1}^n (Z_k - m \alpha Z_{k-1}) \stackrel{P}{\longrightarrow} 0$. \\

        \item[(II)] We can decompose
        \[
            \frac{(m\alpha - \tilde{g}_n)^2}{n} \sum_{k=1}^n Z_{k-1}
            = \frac{(Z_n - m\alpha Z_{n-1})^2}{n \cdot Z_{n-1}} \cdot \sum_{k=1}^n \frac{Z_{k-1}}{Z_{n-1}},
        \]
        and have already shown in Lemma \ref{lma:NormedSumBoundedInProbability} that $\sum_{k=1}^n \frac{Z_{k-1}}{Z_{n-1}}$ is asymptotically bounded in probability by a constant. Therefore the result $\frac{(m\alpha - \tilde{g}_n)^2}{n} \sum_{k=1}^n Z_{k-1} \stackrel{P}{\longrightarrow} 0$ will follow if we show that $\frac{(Z_n - m\alpha Z_{n-1})^2}{n \cdot Z_{n-1}} \stackrel{P}{\longrightarrow} 0$.
        Noting that $\frac{(Z_n - m\alpha Z_{n-1})^2}{n \cdot Z_{n-1}}$ is non-negative, and that
        \[
            \E\Bigg( \frac{(Z_n - m\alpha Z_{n-1})^2}{n \cdot Z_{n-1}} \Bigg)
            = \E\Bigg( \E\bigg(
                \frac{(Z_n - m\alpha Z_{n-1})^2}{n \cdot Z_{n-1}}
            \bigg| \F_{n-1} \bigg) \Bigg)
            = \frac{\sigma^2 \alpha + m^2 \beta}{n}
            \stackrel{n\to\infty}{\longrightarrow} 0,
        \]
        this follows from Markov's inequality.
    \end{enumerate}
\end{proof}

\bigskip

Theorem \ref{thm:ObservedProgenitorsConsistentEstimators} will require the use of the following lemma, which we state without proof, noting that its proof is effectively unchanged from that of Lemma \ref{lma:NormedSumBoundedInProbability}, replacing $Z_n$ with $\phi(Z_n)$.
To help illustrate this, note that, given that $\varepsilon(z) = \alpha z$ for all $z \in \N_0$,
\[
    \E(Z_n \,|\, Z_{n-1}) = m\alpha \cdot Z_{n-1}
    \quad\text{and}\quad
    \E(\phi(Z_n) \,|\, \phi(Z_{n-1})) = m\alpha \cdot \phi(Z_{n-1}).
\]

\begin{lemma}\label{lma:PhiNormedSumBoundedInProbability}
    Let $\Z$ be a supercritical CBP with control function $\phi(\cdot)$ having $\varepsilon(z) = \alpha z$ and $\nu^2(z) = \beta z$, and with $\sigma^2$ finite. Then, on the event $\{Z_n \to \infty\}$ (so that $\phi(Z_n) \to \infty$ as well) and for $s$ such that $1 < s < m\alpha$,
    \[
        \lim_{n\to\infty} \P\bigg(
            \sum_{k=1}^n \frac{\phi(Z_{k-1})}{\phi(Z_{n-1})} \leq \frac{1}{1 - s^{-1}}
        \bigg) = 1.
    \]
\end{lemma}

\begin{proof}[Proof of Theorem \ref{thm:ObservedProgenitorsConsistentEstimators}:]

    Let $\Z$ be a supercritical CBP with unknown control function $\phi(\cdot)$ having $\varepsilon(z) = \alpha z$ and $\nu^2(z) = \beta z$, and with $\sigma^2$ finite. Assume that progenitor counts are observed alongside population numbers, and accordingly take $\F_n$ to be the $\sigma$-algebra generated by $\big( Z_0,\, \phi(Z_0), \dots, \phi(Z_{n-1}),\, Z_n \big)$ and $\F_{\phi_n}$ the $\sigma$-algebra generated by $\big( Z_0,\, \phi(Z_0), \dots, Z_n,\, \phi(Z_n) \big)$, while $\F_{\phi_{-1}}$ is taken to represent the trivial $\sigma$-algebra.
    
    Throughout the proof, we 
   % assume without loss of generality that $\P(Z_n \to \infty) > 0$, and 
    make the simplifying assumption that $\P(\xi > 0) = 1$ and $\P(\phi(z) > 0) = 1$ for all $z \in \N_1$. This implies that $\P(Z_1 > 0 \,|\, Z_0 = z) = 1$ for all $z \in \N_1$, that $I^{\phi}_n = \{1, \dots, n\}$, and that $I^-_n = \{0, \dots, n-1\}$. Additionally, under this assumption, $\P(Z_n \to \infty) = 1$, and the estimators simplify to
    \begin{align*}
        \hat{m}_n &:= \frac{1}{n} \sum_{k = 1}^n \frac{Z_k}{\phi(Z_{k-1})},
        &\quad
        \hat{\alpha}_n &:= \frac{1}{n} \sum_{k = 0}^{n-1} \frac{\phi(Z_k)}{Z_k}, \\
        \bar{\sigma}_n^2 &:= \frac{1}{n} \sum_{k = 1}^n \frac{(Z_k - m \cdot \phi(Z_{k-1}))^2}{\phi(Z_{k-1})},
        &\quad
        \bar{\beta}_n &:= \frac{1}{n} \sum_{k = 0}^{n-1} \frac{(\phi(Z_k) - \alpha \cdot Z_k)^2}{Z_k}, \\
        \hat{\sigma}_n^2 &:= \frac{1}{n} \sum_{k = 1}^n \frac{(Z_k - \tilde{m}_n \cdot \phi(Z_{k-1}))^2}{\phi(Z_{k-1})},
        &\quad
        \hat{\beta}_n &:= \frac{1}{n} \sum_{k = 0}^{n-1} \frac{(\phi(Z_k) - \tilde{\alpha}_{n-1} \cdot Z_k)^2}{Z_k}.
    \end{align*}
 We omit the extension of our proof to the case where $\P(\xi = 0) > 0$ and/or there exists $z \in \N_1$ such that $\P(\phi(z) = 0) > 0$, which follows similar arguments as in the proof of Theorem \ref{thm:KnownPhiConsistentEstimators}, part (i), case (b).

    Throughout the rest of the proof, we will also make use of the following:
    \begin{equation}\label{eqn:OneOverPhiExpectationBound}
        \E\big( \phi^{-1}(z) \big) \leq 1
        \quad\text{for all}\; z \in \N_1.
    \end{equation}
    This follows as a consequence of the assumption $\P(\phi(z) > 0) = 1$ for all $z \in \N_1$, which tells us that
    $\phi(z) \geq 1$ a.s., and therefore $\phi^{-1}(z) \leq 1$ a.s..

    \bigskip

    \textbf{(i) Strong consistency of $\boldsymbol{\hat{m}_n}$:}
    Let $U_n := \tilde{m}_n - m$, where we recall $\tilde{m}_n := Z_n / \phi(Z_{n-1})$, so that
    \[
        \E(U_n | \F_{n-1} )
        = \E\big(
            \E(U_n | \F_{\phi_{n-1}} )
        \big| \F_{n-1} \big)
        = 0
    \]
    and $\E|U_n| \leq 2m$, while, using \eqref{eqn:OneOverPhiExpectationBound},
    \[
        \E( U_n^2 | \F_{n-1} )
        = \E\big(
            \V( U_n | \F_{\phi_{n-1}} )
        \big| \F_{n-1} \big)
        = \E\bigg(
            \frac{\sigma^2}{\phi(Z_{n-1})}
        \bigg| \F_{n-1} \bigg)
        \leq \sigma^2.
    \]
    Therefore
    \[
        \sum_{n=1}^\infty \frac{1}{n^2} \E( U_n^2 | \F_{n-1} ) < \infty,
    \]
    so that, by Corollary \ref{crly:SLLNforMDS}, $\frac{1}{n} \sum_{k=1}^n U_k \stackrel{a.s.}{\longrightarrow} 0$ and thus $\hat{m}_n \stackrel{a.s.}{\longrightarrow} m$.

    \bigskip

    \textbf{(ii) Strong consistency of $\boldsymbol{\hat{\alpha}_n}$:}
    Let $U_{\phi_n} := \tilde{\alpha}_n - \alpha$, where we recall $\tilde{\alpha}_n := \phi(Z_n) / Z_n$, so that
    \[
        \E(U_{\phi_n} | \F_{\phi_{n-1}} )
        = \E\big(
            \E(U_{\phi_n} | \F_{n} )
        \big| \F_{\phi_{n-1}} \big)
        = 0
    \]
    and $\E|U_{\phi_n}| \leq 2\alpha$, while
    \[
        \E\big( U_{\phi_n}^2 \big| \F_{\phi_{n-1}} \big)
        = \E\big(
            \V( U_{\phi_n} | \F_n )
        \big| \F_{\phi_{n-1}} \big)
        = \E\bigg(
            \frac{\beta}{Z_n}
        \bigg| \F_n \bigg)
        \leq \beta
    \]
    since $Z_n \geq 1$ under our simplifying assumption that $\P(\phi(z) = 0) = 0$ and $\P(\xi = 0) = 0$. Therefore
    \[
        \sum_{n=1}^\infty \frac{1}{n^2} \E\big( U_{\phi_{n-1}}^2 \big| \F_{\phi_{n-2}} \big) < \infty,
    \]
    so that, by Corollary \ref{crly:SLLNforMDS}, $\frac{1}{n} \sum_{k=1}^n U_{k-1} \stackrel{a.s.}{\longrightarrow} 0$ and thus $\hat{\alpha}_n \stackrel{a.s.}{\longrightarrow} \alpha$.

    \bigskip

    \textbf{(iii) Strong consistency of $\boldsymbol{\bar{\sigma}^2_n}$:}
    Under the assumption that $m$ is known, we introduce the estimator
    \[
        \tilde{\sigma}^2_n := \frac{(Z_n - m \cdot \phi(Z_{n-1}))^2}{\phi(Z_{n-1})},
    \]
    such that $\bar{\sigma}^2_n = \frac{1}{n} \sum_{k=1}^n \tilde{\sigma}^2_k$, and define
    \[
        U_n := \tilde{\sigma}^2_n - \sigma^2,
    \]
    such that $\E(U_n | \F_{n-1} ) = \E\big( \E(U_n | \F_{\phi_{n-1}} ) \big| \F_{n-1} \big) = 0$ and $\E|U_n| \leq 2 \sigma^2$, while
    \[
        \sum_{n=1}^{\infty} \frac{1}{n^2} \E\big( U^2_n \big| \F_{n-1} \big)
        = \sum_{n=1}^{\infty} \frac{1}{n^2} \Big( \E\big( \tilde{\sigma}^4_n \big| \F_{n-1} \big) - \sigma^4 \Big)
        < \sum_{n=1}^{\infty} \frac{1}{n^2} \E\big( \tilde{\sigma}^4_n \big| \F_{n-1} \big).
    \]
    By using the expansion for the fourth central moment of an i.i.d.\ sum and using \eqref{eqn:OneOverPhiExpectationBound}, we can find that
    \begin{align*}
        \E\big( \tilde{\sigma}^4_n \big| \F_{n-1} \big)
        &= \E\Bigg(
            \frac{
                \E\big( (Z_n - m \cdot \phi(Z_{n-1}))^4 \big| \F_{\phi_{n-1}} \big)
            }{
                \phi^2(Z_{n-1})
            }
        \Bigg| \F_{n-1} \Bigg) \\
        &= \E\Bigg(
            \frac{\E(\xi - m)^4 - 3 \sigma^4}{\phi(Z_{n-1})} + 3\sigma^4
        \Bigg| \F_{n-1} \Bigg) \\
        &\leq \E(\xi - m)^4.
    \end{align*}
    Since we have assumed that $\E(\xi - m)^4$ is finite, it follows that
    \[
        \sum_{n=1}^{\infty} \frac{1}{n^2} \E\big( U^2_n \big| \F_{n-1} \big)
        \leq \sum_{n=1}^{\infty} \frac{\E(\xi - m)^4}{n^2}
        < \infty.
    \]
    Therefore, by Corollary \ref{crly:SLLNforMDS}, $\frac{1}{n} \sum_{k=1}^n U_k \stackrel{a.s.}{\longrightarrow} 0$ and thus $\bar{\sigma}^2_n \stackrel{a.s.}{\longrightarrow} \sigma^2$.

    \bigskip

    \textbf{(iv) Weak consistency of $\boldsymbol{\hat{\sigma}^2_n}$:} We abandon the assumption that $m$ is known. We can decompose
    \[
        \hat{\sigma}^2_n
        = \bar{\sigma}^2_n
        +\underbrace{ \frac{2(m - \tilde{m}_n)}{n} \sum_{k=1}^n (Z_k - m \cdot \phi(Z_{k-1}))}_{\text{I}}
        + \underbrace{\frac{(m - \tilde{m}_n)^2}{n} \sum_{k=1}^n \phi(Z_{k-1})}_{\text{II}},
    \]
    where we have shown in (iii) that $\bar{\sigma}^2_n \stackrel{a.s.}{\longrightarrow} \sigma^2$. It remains to show that I and II both converge in probability to zero. \\

    \begin{enumerate}
        \item[(I)]
        We can decompose
        \[
            \hspace*{-0.6cm}
            \frac{2(m - \tilde{m}_n)}{n} \sum_{k=1}^n (Z_k - m \cdot \phi(Z_{k-1}))
            = -2 \underbrace{
                \sqrt{\frac{\sum_{k=1}^n Z_{k-1}}{\phi(Z_{n-1})}}
            }_{\text{(a)}} \cdot \underbrace{
                \frac{Z_n - m \cdot \phi(Z_{n-1})}{n^{1/4} \cdot \sqrt{\phi(Z_{n-1})}}
            }_{\text{(b)}} \cdot \underbrace{
                 \frac{\sum_{k=1}^n (Z_k - m \cdot \phi(Z_{k-1}))}{n^{3/4}\sqrt{\sum_{k=1}^n Z_{k-1}}}
            }_{\text{(c)}},
        \]
         and consider each of the components $(a)$, $(b)$ and $(c)$ in turn. \\

        \begin{enumerate}
            \item[(a)] As a result of Lemma \ref{lma:NormedSumBoundedInProbability}, for any $s$ such that $1 < s < m\alpha$, we have that
            \[
                \lim_{n\to\infty} \P\Bigg(
                    \sqrt{\frac{\sum_{k=1}^n Z_{k-1}}{Z_{n-1}}} \leq \sqrt{\frac{1}{1 - s^{-1}}}
                \Bigg) = 1.
            \]
            Using Chebyshev's inequality, for $r < \alpha$,
            \begin{align*}
                \P( \phi(Z_{n-1}) > r \cdot Z_{n-1} \,|\, Z_{n-1} )
                &\geq 1 - \P\big( |\phi(Z_{n-1}) - \alpha \cdot Z_{n-1} | \geq (\alpha - r) \cdot Z_{n-1} \,\big|\, Z_{n-1} \big) \\
                &\geq 1 - \frac{\beta}{(\alpha - r)^2 \cdot Z_{n-1}}
            \end{align*}
            By conditioning on the event $\{ \phi(Z_{n-1}) > r \cdot Z_{n-1} \}$, we then find that
            \begin{align*}
                &\lim_{n\to\infty} \P\Bigg(
                    \sqrt{\frac{\sum_{k=1}^n Z_{k-1}}{\phi(Z_{n-1})}} \leq \sqrt{\frac{r}{1 - s^{-1}}}
                \Bigg) \\
                &= \lim_{n\to\infty} \P\Bigg(
                    \sqrt{\frac{\sum_{k=1}^n Z_{k-1}}{\phi(Z_{n-1})}} \leq \sqrt{\frac{r}{1 - s^{-1}}}
                \,\Bigg|\, \phi(Z_{n-1}) > r \cdot Z_{n-1} \Bigg)
                \cdot \P( \phi(Z_{n-1}) > r \cdot Z_{n-1} ) \\
                &= 1
            \end{align*}
            on the event $\{ Z_n \to \infty \}$.
            Hence we see that $\sqrt{\frac{\sum_{k=1}^n Z_{k-1}}{\phi(Z_{n-1})}}$ is asymptotically bounded by a constant in probability.
            \medskip

            \item[(b)] Given that
            \[
                \E\bigg( \frac{Z_n - m \cdot \phi(Z_{n-1})}{n^{1/4} \cdot \sqrt{\phi(Z_{n-1})}} \bigg| \F_{\phi_{n-1}} \bigg) = 0
                \quad\text{and}\quad
                \V\bigg( \frac{Z_n - m \cdot \phi(Z_{n-1})}{n^{1/4} \cdot \sqrt{\phi(Z_{n-1})}} \bigg| \F_{\phi_{n-1}} \bigg) = \frac{\sigma^2}{\sqrt{n}},
            \]
            we see that
            \[
                \E\bigg(
                    \frac{Z_n - m \cdot \phi(Z_{n-1})}{n^{1/4} \cdot \sqrt{\phi(Z_{n-1})}}
                \bigg)
                = 0
                \quad\text{and}\quad
                \V\bigg(
                    \frac{Z_n - m \cdot \phi(Z_{n-1})}{n^{1/4} \cdot \sqrt{\phi(Z_{n-1})}}
                \bigg)
                \stackrel{n\to\infty}{\longrightarrow} 0,
            \]
            so, as a consequence of Chebyshev's inequality, it follows that
            \[
                \frac{Z_n - m \cdot \phi(Z_{n-1})}{n^{1/4} \cdot \sqrt{\phi(Z_{n-1})}}
                \stackrel{P}{\longrightarrow} 0.
            \]

            \item[(c)] Defining
            \[
                U_n := Z_n - m \cdot \phi(Z_{n-1})
                \quad\text{and}\quad
                J_n := n^{3/4}\sqrt{\sum_{k=1}^n Z_{k-1}},
            \]
            we can see that $\{U_n\}_{n\in\N_0}$ is a martingale difference sequence with respect to the filtration $\{ \F_n \}_{n\in\N_0}$, since $\E(U_n | \F_{n-1}) = 0$ for all $n \in \N_1$ and
            \begin{align*}
                \E|U_n|
                &\leq \sqrt{\E(U_n^2)} \tag{H\"{o}lder's inequality}\\
                &= \sqrt{\E(\V(Z_n | \F_{\phi{n-1}}))} \\
                &= \sqrt{\sigma^2 \cdot \E(\phi(Z_{n-1}))} \\
                &= \sqrt{\sigma^2 \alpha \cdot \E(Z_{n-1})} \\
                &< \infty,
            \end{align*}
            and $\{J_n\}_{n\in\N_0}$ is positive, non-decreasing, with $\lim_{n\to\infty} J_n = \infty$.
            Additionally,
            \begin{align*}
                \sum_{n=1}^{\infty} J_n^{-2} \E(U_n^2 | \F_{n-1})
                &= \sum_{n=1}^{\infty} \frac{\E(\V(Z_n | \F_{\phi_{n-1}}) | \F_{n-1})}{n^{3/2} \cdot \sum_{k=1}^n Z_{k-1}} \\
                &= \sum_{n=1}^{\infty} \frac{\sigma^2 \alpha \cdot Z_{n-1}}{n^{3/2} \cdot \sum_{k=1}^n Z_{k-1}} \\
                &< \sum_{n=1}^{\infty} \frac{\sigma^2 \alpha}{n^{3/2}} \\
                &< \infty \;\;\text{a.s.},
            \end{align*}
            so Theorem \ref{thm:SLLNforMDS} applies, such that
            $J_n^{-1} \sum_{k=1}^n U_k = \frac{\sum_{k=1}^n (Z_k - m \cdot \phi(Z_{k-1}))}{n^{3/4}\sqrt{\sum_{k=1}^n Z_{k-1}}} \stackrel{a.s.}{\longrightarrow} 0$ as $n \to \infty$. \\
        \end{enumerate}

        \item[(II)] We can decompose
        \[
            \frac{(m - \tilde{m}_n)^2}{n} \sum_{k=1}^n \phi(Z_{k-1})
            = \frac{(Z_n - m \cdot \phi(Z_{n-1}))^2}{n \cdot \phi(Z_{n-1})} \cdot \sum_{k=1}^n \frac{\phi(Z_{k-1})}{\phi(Z_{n-1})}.
        \]
        and have already shown in Lemma \ref{lma:PhiNormedSumBoundedInProbability} that $\sum_{k=1}^n \frac{\phi(Z_{k-1})}{\phi(Z_{n-1})}$ is asymptotically bounded in probability by a constant. Therefore the result $\frac{(m - \tilde{m}_n)^2}{n} \sum_{k=1}^n \phi(Z_{k-1}) \stackrel{P}{\longrightarrow} 0$ will follow if we show that $\frac{(Z_n - m \cdot \phi(Z_{n-1}))^2}{n \cdot \phi(Z_{n-1})} \stackrel{P}{\longrightarrow} 0$.
        Noting that $\frac{(Z_n - m \cdot \phi(Z_{n-1}))^2}{n \cdot \phi(Z_{n-1})}$ is non-negative, and that
        \[
            \E\Bigg( \frac{(Z_n - m \cdot \phi(Z_{n-1}))^2}{n \cdot \phi(Z_{n-1})} \Bigg)
            = \E\Bigg( \frac{
                \V\big( Z_n \big| \F_{\phi_{n-1}}) \big)
            }{
                n \cdot \phi(Z_{n-1})
            } \Bigg)
            = \frac{\sigma^2}{n}
            \stackrel{n\to\infty}{\longrightarrow} 0,
        \]
        this follows from Markov's inequality.
    \end{enumerate}

    \bigskip

    \textbf{(v) Strong consistency of $\boldsymbol{\bar{\beta}_n}$:}
    Under the assumption that $\alpha$ is known, we introduce the estimator
    \[
        \tilde{\beta}_n := \frac{(\phi(Z_n) - \alpha \cdot Z_n)^2}{Z_n},
    \]
    such that $\bar{\beta}_n = \frac{1}{n} \sum_{k=0}^{n-1} \tilde{\beta}_k$, and define
    \[
        U_{\phi_n} := \tilde{\beta}_n - \beta,
    \]
    such that $\E( U_{\phi_n} | \F_{\phi_{n-1}} ) = \E\big( \E( U_{\phi_n} | \F_n ) \big| \F_{\phi_{n-1}} \big) = 0$ and $\E| U_{\phi_n} | \leq 2 \beta$, while
    \[
        \sum_{n=1}^{\infty} \frac{1}{n^2} \E (U^2_{\phi_n} | \F_{\phi_{n-1}} )
        = \sum_{n=1}^{\infty} \frac{1}{n^2} \Big( \E ( \tilde{\beta}^2_n | \F_{\phi_{n-1}} ) - \beta^2 \Big)
        < \sum_{n=1}^{\infty} \frac{1}{n^2} \E ( \tilde{\beta}^2_n | \F_{\phi_{n-1}} ).
    \]
    Under the assumption that there exists $d > 0$ such that $\sup_{z\geq 1}\Big\{ \frac{\E(\phi(z) - \varepsilon(z))^4}{z^2} \Big\} \leq d$, we have that $\E ( \tilde{\beta}^2_n | \F_{\phi_{n-1}} ) \leq d$, so that
    \[
        \sum_{n=1}^{\infty} \frac{1}{n^2} \E (U^2_{\phi_n} | \F_{\phi_{n-1}} )
        < \sum_{n=1}^{\infty} \frac{d}{n^2}
        < \infty.
    \]

    Therefore, by Corollary \ref{crly:SLLNforMDS}, $\frac{1}{n} \sum_{k=1}^n U_{\phi_{k-1}} \stackrel{a.s.}{\longrightarrow} 0$ and thus $\bar{\beta}_n \stackrel{a.s.}{\longrightarrow} \beta$.

    \bigskip

    \textbf{(vi) Weak consistency of $\boldsymbol{\hat{\beta}_n}$:} We abandon the assumption that $\alpha$ is known. We can decompose
    \[
        \hat{\beta}_n
        = \bar{\beta}_n
        + \underbrace{
            \frac{2(\alpha - \tilde{\alpha}_{n-1})}{n} \sum_{k=0}^{n-1} (\phi(Z_k) - \alpha Z_k)
        }_{\text{I}}
        + \underbrace{
            \frac{(\alpha - \tilde{\alpha}_{n-1})^2}{n} \sum_{k=0}^{n-1} Z_k,
        }_{\text{II}}
    \]
    where we have shown in (v) that $\bar{\beta}_n \stackrel{a.s.}{\longrightarrow} \beta$. It remains to show that I and II both converge in probability to zero. \\

    \begin{enumerate}
        \item[(I)]
        We can decompose
        \[
            \hspace*{-0.6cm}
            \frac{2(\alpha - \tilde{\alpha}_{n-1})}{n} \sum_{k=0}^{n-1} (\phi(Z_k) - \alpha Z_k)
            = -2 \underbrace{
                \sqrt{\frac{\sum_{k=1}^{n-1} \phi(Z_{k-1})}{Z_{n-1}}}
            }_{\text{(a)}} \cdot \underbrace{
                \frac{\phi(Z_{n-1}) - \alpha Z_{n-1}}{n^{1/4} \cdot \sqrt{Z_{n-1}}}
            }_{\text{(b)}} \cdot \underbrace{
                 \frac{\sum_{k=0}^{n-1} (\phi(Z_k) - \alpha Z_k)}{n^{3/4}\sqrt{\sum_{k=1}^{n-1} \phi(Z_{k-1})}}
            }_{\text{(c)}},
        \]
         and consider each of the components $(a)$, $(b)$ and $(c)$ in turn. \\

        \begin{enumerate}
            \item[(a)] As a result of Lemma \ref{lma:PhiNormedSumBoundedInProbability}, for any $s$ such that $1 < s < m\alpha$, we have that
            \[
                \lim_{n\to\infty} \P\Bigg(
                    \sqrt{\frac{\sum_{k=1}^{n-1} \phi(Z_{k-1})}{\phi(Z_{n-2})}} \leq \sqrt{\frac{1}{1 - s^{-1}}}
                \Bigg) = 1.
            \]
            Using Chebyshev's inequality, for $r < m$,
            \begin{align*}
                \hspace*{-0.4cm}
                \P( Z_{n-1} > r \cdot \phi(Z_{n-2}) \,|\, \phi(Z_{n-2}) )
                &\geq 1 - \P\big( |Z_{n-1} - m \cdot \phi(Z_{n-2}) | \geq (m - r) \cdot \phi(Z_{n-2}) \,\big|\, \phi(Z_{n-2}) \big) \\
                &\geq 1 - \frac{\sigma^2}{(m - r)^2 \cdot \phi(Z_{n-2})}
            \end{align*}
            By conditioning on the event $\{ Z_{n-1} > r \cdot \phi(Z_{n-2}) \}$, we then find that
            \begin{align*}
                &\lim_{n\to\infty} \P\Bigg(
                    \sqrt{\frac{\sum_{k=1}^{n-1} \phi(Z_{k-1})}{Z_{n-1}}} \leq \sqrt{\frac{r}{1 - s^{-1}}}
                \Bigg) \\
                &= \lim_{n\to\infty} \P\Bigg(
                    \sqrt{\frac{\sum_{k=1}^{n-1} \phi(Z_{k-1})}{Z_{n-1}}} \leq \sqrt{\frac{r}{1 - s^{-1}}}
                \,\Bigg|\, Z_{n-1} > r \cdot \phi(Z_{n-2}) \Bigg)
                \cdot \P( Z_{n-1} > r \cdot \phi(Z_{n-2}) ) \\
                &= 1
            \end{align*}
            on the event $\{ Z_n \to \infty \}$.
            Hence we see that $\sqrt{\frac{\sum_{k=1}^{n-1} \phi(Z_{k-1})}{Z_{n-1}}}$ is asymptotically bounded by a constant in probability.
            \medskip

            \item[(b)] Given that
            \[
                \E\bigg(
                    \frac{\phi(Z_{n-1}) - \alpha Z_{n-1}}{n^{1/4} \cdot \sqrt{Z_{n-1}}}
                \bigg| \F_{n-1} \bigg) = 0
                \quad\text{and}\quad
                \V\bigg(
                    \frac{\phi(Z_{n-1}) - \alpha Z_{n-1}}{n^{1/4} \cdot \sqrt{Z_{n-1}}}
                \bigg| \F_{n-1} \bigg) = \frac{\beta}{\sqrt{n}},
            \]
            we see that
            \[
                \E\bigg(
                    \frac{\phi(Z_{n-1}) - \alpha Z_{n-1}}{n^{1/4} \cdot \sqrt{Z_{n-1}}}
                \bigg)
                = 0
                \quad\text{and}\quad
                \V\bigg(
                    \frac{\phi(Z_{n-1}) - \alpha Z_{n-1}}{n^{1/4} \cdot \sqrt{Z_{n-1}}}
                \bigg)
                \stackrel{n\to\infty}{\longrightarrow} 0,
            \]
            so, as a consequence of Chebyshev's inequality, it follows that
            \[
                \frac{\phi(Z_{n-1}) - \alpha Z_{n-1}}{n^{1/4} \cdot \sqrt{Z_{n-1}}}
                \stackrel{P}{\longrightarrow} 0.
            \]

            \item[(c)] Defining
            \[
                U_{\phi_{n-1}} := \phi(Z_{n-1}) - \alpha Z_{n-1}
                \quad\text{and}\quad
                J_n := n^{3/4}\sqrt{\sum_{k=1}^{n-1} \phi(Z_{k-1})},
            \]
            we can see that $\{U_{\phi_{n-1}}\}_{n\in\N_1}$ is a martingale difference sequence with respect to the filtration $\{ \F_{\phi_n} \}_{n\in\N_0}$, since $\E(U_{\phi_{n-1}} | \F_{\phi_{n-2}}) = 0$ for all $n \in \N_1$ and
            \begin{align*}
                \E|U_{\phi_{n-1}}|
                &\leq \sqrt{\E(U_{\phi_{n-1}}^2)} \tag{H\"{o}lder's inequality}\\
                &= \sqrt{\E(\V( \phi(Z_{n-1}) | \F_{n-1}))} \\
                &= \sqrt{\beta \cdot \E(Z_{n-1})} \\
                &< \infty,
            \end{align*}
            and $\{J_n\}_{n\in\N_0}$ is positive, non-decreasing, with $\lim_{n\to\infty} J_n = \infty$.
            Additionally,
            \begin{align*}
                \sum_{n=1}^{\infty} J_n^{-2} \E(U_{\phi_{n-1}}^2 | \F_{\phi_{n-2}})
                &= \sum_{n=1}^{\infty} \frac{ \E(\V( \phi(Z_{n-1}) | \F_{n-1}) | \F_{\phi_{n-2}} )}{n^{3/2} \cdot \sum_{k=1}^{n-1} \phi(Z_{k-1})} \\
                &= \sum_{n=1}^{\infty} \frac{m \beta \cdot \phi(Z_{n-2})}{n^{3/2} \cdot \sum_{k=1}^{n-1} \phi(Z_{k-1})} \\
                &< \sum_{n=1}^{\infty} \frac{m \beta}{n^{3/2}} \\
                &< \infty \;\;\text{a.s.},
            \end{align*}
            so Theorem \ref{thm:SLLNforMDS} applies, such that
            $J_n^{-1} \sum_{k=1}^n U_{\phi_{k-1}} = \frac{\sum_{k=0}^{n-1} (\phi(Z_k) - \alpha Z_k)}{n^{3/4}\sqrt{\sum_{k=1}^{n-1} \phi(Z_{k-1})}} \stackrel{a.s.}{\longrightarrow} 0$ as $n \to \infty$. \\
        \end{enumerate}

        Then, since $(a)$ is asymptotically bounded by a constant in probability, and $(b)$ and $(c)$ both converge to zero in probability, $\frac{2(\alpha - \tilde{\alpha}_{n-1})}{n} \sum_{k=0}^{n-1} (\phi(Z_k) - \alpha Z_k) \stackrel{P}{\longrightarrow} 0$. \\

        \item[(II)] We can decompose
        \[
            \frac{(\alpha - \tilde{\alpha}_{n-1})^2}{n} \sum_{k=0}^{n-1} Z_k
            = \frac{(\phi(Z_{n-1}) - \alpha Z_{n-1})^2}{n \cdot Z_{n-1}} \cdot \sum_{k=1}^{n} \frac{Z_{k-1}}{Z_{n-1}},
        \]
        and have already shown in Lemma \ref{lma:NormedSumBoundedInProbability} that $\sum_{k=1}^n \frac{Z_{k-1}}{Z_{n-1}}$ is asymptotically bounded in probability by a constant. Therefore the result $\frac{(\alpha - \tilde{\alpha}_{n-1})^2}{n} \sum_{k=0}^{n-1} Z_k \stackrel{P}{\longrightarrow} 0$ will follow if we show that $\frac{(\phi(Z_{n-1}) - \alpha Z_{n-1})^2}{n \cdot Z_{n-1}} \stackrel{P}{\longrightarrow} 0$.
        Noting that $\frac{(\phi(Z_{n-1}) - \alpha Z_{n-1})^2}{n \cdot Z_{n-1}}$ is non-negative, and that
        \[
            \E\Bigg( \frac{(\phi(Z_{n-1}) - \alpha Z_{n-1})^2}{n \cdot Z_{n-1}} \Bigg)
            = \E\Bigg( \frac{
                \V\big( \phi(Z_{n-1}) \big| \F_{n-1}) \big)
            }{
                n \cdot Z_{n-1}
            } \Bigg)
            = \frac{\beta}{n}
            \stackrel{n\to\infty}{\longrightarrow} 0,
        \]
        this follows from Markov's inequality.
    \end{enumerate}

\end{proof}

\section*{Acknowledgements}

Sophie Hautphenne would like to thank the Australian Research Council (ARC) for support through her Discovery Project DP200101281.

\bibliographystyle{plain}
\bibliography{refs}

\end{document}